\documentclass[11pt,reqno]{amsart}
\usepackage[T1]{fontenc}
\usepackage{lscape}
\usepackage[utf8]{inputenc}
\usepackage{appendix}
\usepackage{paralist}
\usepackage{longtable}
\usepackage{graphicx}
\usepackage{amsmath}
\usepackage{fullwidth}
\usepackage{epstopdf}
\usepackage{lscape}
\usepackage{rotating}
\usepackage{marvosym}
\usepackage{amsthm}
\usepackage{rotating}
\usepackage{mwe}
\usepackage{caption}
\usepackage{subcaption}
\usepackage{amstext}
\usepackage{array}
\usepackage{adjustbox}
\usepackage{subcaption}
\usepackage{times}

\usepackage{xspace}
\usepackage{color}
\usepackage{amssymb}% http://ctan.org/pkg/amssymb
\usepackage{pifont}% http://ctan.org/pkg/pifont
\newcommand{\cmark}{\ding{51}}%
\newcommand{\xmark}{\ding{55}}%
\newcommand{\MATLAB}{\textsc{Matlab}\xspace}
\usepackage{tabularx,ragged2e,booktabs,caption}
\newcolumntype{C}[1]{>{\Centering}m{#1}}

\usepackage{kpfonts}
\usepackage[margin=1in]{geometry}
\numberwithin{equation}{section}
\usepackage{algorithm}% http://ctan.org/pkg/algorithms
\usepackage{algorithmic}% http://ctan.org/pkg/algorithms
\makeatletter
\def\listofalgorithms{\@starttoc{loa}\listalgorithmname}
\def\l@algorithm{\@tocline{0}{3pt plus2pt}{0pt}{1.9em}{}}
\renewcommand{\ALG@name}{Algorithm}
\renewcommand{\listalgorithmname}{List of \ALG@name s}
\numberwithin{algorithm}{section}
\theoremstyle{definition}
\newtheorem{defn}[algorithm]{Definition}
\theoremstyle{remark}
\newtheorem{rem}[algorithm]{Remark}
\theoremstyle{theorem}
\newtheorem{thm}[algorithm]{Theorem}
\theoremstyle{proposition}

\theoremstyle{example}

\newtheorem{exmp}{Example}[section]
\theoremstyle{corollary}

\theoremstyle{assumption}

\theoremstyle{lemma}

\makeatother
\newcommand{\R}{\mathbb{R}}
\newcommand{\fb}{f_{\mathrm{FB}}}
\newcommand{\psifb}{\psi_{\mathrm{FB}}}

\usepackage{booktabs}
\usepackage{multirow}
\usepackage{siunitx}
\newcommand{\liuhao}{\fontsize{9pt}{\baselineskip}\selectfont}
\usepackage{scalerel}

\usepackage{hyperref}
\hypersetup{
    colorlinks=true,
    linkcolor=blue,
    filecolor=magenta,
    urlcolor=cyan,
}
\usepackage[table]{xcolor}

\begin{document}
\title[Comparison of the KKT and value function reformulations in bilevel optimization]{Theoretical and numerical comparison of the Karush-Kuhn-Tucker \\and value function reformulations in bilevel optimization}
%\author[]{Alain B. Zemkoho
%\address{S. Zhou: School of Mathematics, University of Southampton, Southampton, UK}
%\email{shenglong.zhou@soton.ac.uk}}
%
%%\thanks{}%
%\author[]{Shenglong Zhou
%\address{A.B. Zemkoho: School of Mathematics, University of Southampton, Southampton, UK}
%\email{a.b.zemkoho@soton.ac.uk}}
%\author[]{Alain B. Zemkoho and Shenglong Zhou\\\\
%School of Mathematics, University of Southampton, SO17 1SS Southampton, UK\\
%$\{$a.b.zemkoho, shenglong.zhou$\}$@soton.ac.uk\\}
\author[]{Alain B. Zemkoho and Shenglong Zhou\\ \\
\emph{S\MakeLowercase{chool of} M\MakeLowercase{athematical} S\MakeLowercase{ciences}, U\MakeLowercase{niversity of} S\MakeLowercase{outhampton}, SO17 1BJ S\MakeLowercase{outhampton}, UK}\\[2ex]
\MakeLowercase{$\{$\textsf{a.b.zemkoho, shenglong.zhou}$\}$\textsf{@soton.ac.uk}}}
\thanks{This project was funded by the  EPSRC Grant EP/P022553/1}
%%
%\author[]{Shenglong Zhou$^{\ddag}$\\\\
%$^{\dag}$D\MakeLowercase{epartment of} M\MakeLowercase{athematics}, T\MakeLowercase{echnical} U\MakeLowercase{niversity of} D\MakeLowercase{resden}, G\MakeLowercase{ermany}\\
%\MakeLowercase{\textsf{andreas.fischer@tu-dresden.de}}\\\\
%$^{\ddag}$S\MakeLowercase{chool of} M\MakeLowercase{athematics}, U\MakeLowercase{niversity of} S\MakeLowercase{outhampton}, U\MakeLowercase{nited} K\MakeLowercase{ingdom}\\
%\MakeLowercase{$\textsf{\{}$\textsf{shenglong.zhou, a.b.zemkoho}$\textsf{\}}$\textsf{@soton.ac.uk}}}%
%
%\thanks{The work of ABZ and SZ is funded by the  EPSRC Grant {EP/P022553/1}.  Preliminary versions of this work were presented in 2016 at ICCOPT in Tokyo with partial funding from the Institute of Mathematics  \& its Applications, as well as in 2018 at IWOP18 in Lille, ISMP in Bordeaux, and OR60 in Lancaster.}%
%\subjclass[2010]{90C26, 90C30, 90C46, 90C53}%
%\keywords{Bilevel optimization, Newton method}%

\date{\today}

\begin{abstract}
The Karush-Kuhn-Tucker and  value function (\emph{lower-level value function}, to be precise) reformulations are the most common single-level transformations of the bilevel optimization problem. So far, these reformulations have either been studied independently or as a joint optimization problem in an attempt to take advantage of the best properties from each model. To the best of our knowledge, these reformulations have not yet been compared in the existing literature. This paper is a first attempt towards establishing whether one of these reformulations is best at solving a given class of the optimistic bilevel optimization problem. We design a comparison framework, which seems fair, considering the theoretical properties of these reformulations. This work reveals that although none of the models seems to particularly dominate the other from the theoretical point of view, the value function reformulation seems to numerically outperform the Karush-Kuhn-Tucker reformulation on a Newton-type algorithm. The computational experiments here are mostly based on test problems from the Bilevel Optimization LIBrary (BOLIB).
\end{abstract}
\maketitle
%\tableofcontents

\section{Introduction}\label{Introduction}
Our focus in this paper is the standard optimistic bilevel optimization problem
\begin{equation}\label{eq:P}\tag{P} %before we had \label{P}
  \underset{x,y}\min~F(x,y)\;\;\mbox{s.t.}\;\; G(x,y)\le 0,\;\, y\in S(x),
\end{equation}
also known as the upper-level problem. The functions $F:\R^n\times\R^m\to\R$ and $G:\R^n\times\R^m\to\R^p$ represent the upper-level objective and upper-level constraint functions, respectively, while $x\in \mathbb{R}^n$ (resp. $y\in \mathbb{R}^m$) corresponds to the upper-level (resp. lower-level) variable. Note that we have $n, m, p\in { {\mathbb{N}^*:=\{1, \, 2, \ldots\}}}$. In the sequel, we collect all the feasible upper-level variables as follows:
\begin{equation}\label{Upper-level-feasible set}
  X:=\left\{x\in \mathbb{R}^n|\; \exists y\in \mathbb{R}^m:\;\, G(x,y)\leq 0 \right\}.
\end{equation}
In problem \eqref{eq:P}, the set-valued mapping $S: \mathbb{R}^n \rightrightarrows \mathbb{R}^m$ describes the set of optimal solutions of the following parametric optimization problem, known as the lower-level problem:
\begin{equation}\label{lower-level problem}
    \underset{y}\min~\left\{f(x,y)\mid g(x,y)\le 0\right\}.
\end{equation}
That is, precisely, we have
\begin{equation}\label{S Map}
S(x):= \left\{\begin{array}{lll}
                \arg\underset{y}\min~\left\{f(x,y)\mid g(x,y)\le 0\right\} & \mbox{if} & x\in X,\\
                \emptyset & \mbox{if} & x\in \mathbb{R}^n \setminus X.
              \end{array}
\right.
\end{equation}
% for all $x\in X$. %and $S(x)=\emptyset$ otherwise.
The functions $f:\R^n\times\R^m\to\R$ and $g:\R^n\times\R^m\to\R^q$ (with $q\in \mathbb{N}^*$) correspond to the lower-level objective and lower-level constraint functions, respectively.

Throughout this paper, the upper- and lower-level problems are constrained only by inequality constraints, for the sake of simplicity. However, all the analysis conducted here remains valid (of course with the corresponding adjustments) if equality constraints are added to the upper-level and/or lower-level feasible set of problem \eqref{eq:P}. Furthermore, to focus our attention only on the main points, we assume throughout the paper that $S(x) \neq \emptyset $ for all $x\in X$.

In the pursue of tractable approaches to solve \eqref{eq:P} from the perspective of standard constrained optimization, two main approaches have been considered to reformulate the problem as a single-level optimization problem. Considering the Lagrangian function
%\begin{equation}\label{L-lower-level}
%  \mathcal{L}(x,y,z):= f(x,y) + z^\top g(x,y),
%\end{equation}
\begin{equation}\label{def-h}
  \ell(x,y,z):= f(x,y) + z^\top g(x,y),
\end{equation}
of problem \eqref{lower-level problem} and assuming that $f$ and $g$ are differentiable w.r.t. $y$, the first one-level reformulation of problem \eqref{eq:P} is the Karush-Kuhn-Tucker (KKT) reformulation  that can be written as
\begin{equation}\label{eq:KKTR}\tag{KKTR} %before, \label{eq-P}
\begin{array}{rl}
   \underset{x,\, y, \, z}\min & F(x,y) \\
  \, \mbox{ s.t. }              & G(x,y)\leq 0, \ \ \nabla_2\ell(x,y, z)=0,\\
                                & \,\,g(x,y)\leq 0, \ \ z \geq 0, \ \ z^\top g(x,y)=0,
\end{array}
\end{equation}
where $\nabla_2\ell$ corresponds to the gradient of $\ell$ w.r.t. $y$.
The second main approach to transform \eqref{eq:P} into a single-level optimization problem is the lower-level value function (LLVF) reformulation
\begin{equation}\label{eq:LLVFR}\tag{LLVFR} %before, \label{LLVF}
\begin{array}{lll}
 &  \underset{x,\,y}\min & F(x,y) \\
                          &  \mbox{ s.t. }        & G(x,y)\leq 0,  \ \   g(x,y)\leq 0,\ \ f(x,y)-\varphi(x)\leq 0,
\end{array}
\end{equation}
where  $\varphi$ denotes the optimal value function of the lower-level problem:
\begin{equation}\label{varphi}
   \varphi(x) := \underset{y}\min~\left\{f(x,y)\mid g(x,y)\leq 0\right\}.
\end{equation}

Both problems \eqref{eq:KKTR} and \eqref{eq:LLVFR} have been independently studied in various papers. For example, \cite{DempeDuttaBlpMpec2010} provides a detailed analysis of the relationship between \eqref{eq:KKTR} and the original problem \eqref{eq:P}. Solution algorithms specifically tailored to reformulation \eqref{eq:KKTR} can be found in \cite{BardBook,DempeFoundations2002,DempeFranke2019, MershaDempe2012}, for example. As for \eqref{eq:LLVFR}, most of the work so far has been dedicated to the development of optimality conditions (see, e.g., \cite{DempeDuttaMordukhovichNewNece,DempeZemkohoGenMFCQ,MehlitzZemkohoSufficient2019, YeZhuOptCondForBilevel1995}), but a few works on numerical methods have appeared recently. Namely, the papers \cite{LinXuYeOnSolving2014,XuYeASmoothing2014,XuYeZhang2015Smoothing} propose methods for nonlinear bilevel optimization based on \eqref{eq:LLVFR}; algorithms in \cite{DempeFrankeSolution2014,DempeFranke2016OntheConvex} suggest techniques to solve special cases of problem \eqref{eq:LLVFR}, where relaxation schemes are used to deal with the value function \eqref{varphi}; the authors of \cite{LamparielloSagratella2017,LamparielloSagratella2017Numerically}
proposed numerical methods to solve special bilevel programs by exploiting a connection between problem \eqref{eq:LLVFR}
and a generalized Nash equilibrium problem. A semismooth Newton-type method for \eqref{eq:LLVFR} is developed in \cite{FischerZemZhou2019}. A few papers (see, e.g., \cite{XuYeZhang2015Smoothing,YeZhuNewReform2010}) have also proposed methods based on a combination of \eqref{eq:KKTR} and \eqref{eq:LLVFR}, in order to take advantage of some interesting features from each of these reformulations.

%Note that in addition to the differentiability of the functions involved in the lower-level problem, \eqref{eq:KKTR} requires that problem \eqref{lower-level problem} be convex, in order to be properly defined; see more details in the next section. To mitigate this issue and enhance the chances of \eqref{eq:KKTR} or \eqref{eq:LLVFR} to generate solutions that are optimal for \eqref{eq:P}, Ye and Zhu \cite{YeZhuNewReform2010} proposes to consider the following combination of both reformulations:
%\begin{eqnarray}\label{Mixed Reform}
%\begin{array}{lrl}
%\hspace{2cm} &  \underset{x,\, y, \,z}\min & F(x,y) \\
%                         &  \, \mbox{ s.t. }     & G(x,y)\leq 0, \ \ h(x,y,z)=0,\\
%                         &                       & g(x,y)\leq 0, \ \ z \leq 0, \ \ z^\top g(x,y)=0,\\
%                          &                       & f(x,y)-\varphi(x)\leq 0.
%\end{array}
%\end{eqnarray}
%This model is a viable approach to deal with some weaknesses of problems \eqref{eq:KKTR} and \eqref{eq:LLVFR} considered separately. However, problem \eqref{Mixed Reform} combines the difficulties posed by both \eqref{eq:KKTR} and \eqref{eq:LLVFR} in a single framework; i.e., the complementarity conditions in the feasible set of the former and the nonsmoothness and implicit nature of the value function appearing in the latter. Clearly, considered separately, problems \eqref{eq:KKTR} and \eqref{eq:LLVFR} are already quite difficult classes of optimization problems and are mostly often considered separately in the literature.

Note however that problems \eqref{eq:KKTR} and \eqref{eq:LLVFR} taken separately, are reformulations of the same problem \eqref{eq:P}, but which by their nature, seem to be far apart from each other. It therefore seems interesting to find a way to compare them.  % to see if a specific model could be more suitable for a given special class of problem \eqref{eq:P}.
This paper is a first attempt towards establishing whether one of these reformulations is best at solving a given class of problem \eqref{eq:P}. Our framework for comparing problems \eqref{eq:KKTR} and \eqref{eq:LLVFR} revolves around the 5 questions below, which gradually go from the basic considerations often taken into account when solving an optimization problem to a numerical performance from the perspective of a certain numerical method:\\[-1.5ex]

\noindent \textbf{(Q1)} \emph{How are problems \eqref{eq:KKTR} and \eqref{eq:LLVFR} related to problem \eqref{eq:P} and how do the requirements for these problems to be smooth or locally Lischitz continuous optimization problems compare to each other}? This question is considered in Subsection \ref{Nature of reformulations and relationships to original problem}, where first of all, we discuss the challenges in solving \eqref{eq:P} via \eqref{eq:KKTR} or \eqref{eq:LLVFR} and highlight the fact that despite the apparent strong conditions needed to ensure a close link between the former problem and \eqref{eq:P}, aiming to solve \eqref{eq:LLVFR} as a smooth ($\mathcal{C}^1$) optimization problem requires even much stronger conditions.\\[-1.5ex]

\noindent \textbf{(Q2)} \emph{How do the qualification conditions needed to derive necessary optimality conditions for \eqref{eq:KKTR} and \eqref{eq:LLVFR} relate to each other}? The analysis of this question conducted in Subsection \ref{Qualification conditions} shows that on top of the aforementioned technical requirements to establish \eqref{eq:LLVFR} as a smooth or local Lipschitz optimization problem, the model does not seem to provide the same level of flexibility for the fulfilment of qualification conditions that \eqref{eq:KKTR} enjoys thanks to its connection to mathematical programs with equilibrium constraints (MPCCs).\\[-1.5ex]

\noindent \textbf{(Q3)} \emph{How do the optimality conditions resulting from \eqref{eq:KKTR} and \eqref{eq:LLVFR} compare to each other}? Acknowledging the fact that each of these reformulations can lead to a wide variety of optimality concepts, we identify specific classes of conditions and problems that enable sensible relationships between the two models and subsequently to tractable solution algorithms; see Subsection \ref{Optimality conditions} for  details on the optimality conditions and Section \ref{Newton method for the auxiliary equation} for the algorithms.\\[-1.5ex]

\noindent \textbf{(Q4)} \emph{How do the qualification conditions necessary to establish the convergence results for a corresponding version of the semismooth Newton method compare to each other}? As the standard framework to analyse and solve \eqref{eq:LLVFR} is via the partial exact penalization  \cite{YeZhuOptCondForBilevel1995}, considering a similar approach for \eqref{eq:KKTR} seems to be the most sensible methodological approach to compare our reformulations. Hence, as a by-product of this paper, we develop, probably for the first time, a semismooth Newton-type algorithm for \eqref{eq:KKTR} and compare it to the corresponding algorithm for \eqref{eq:LLVFR} developed in \cite{FischerZemZhou2019}. As expected, \eqref{eq:KKTR} appears to be more demanding, in terms of the derivative requirement for the lower-level problem (3rd order derivatives are necessary for convergence analysis and implementation). However, from a theoretical point of view, there does seem to be a clear dominance of the qualification conditions for convergence of the method of one model on the other one; cf. Section \ref{Newton method for the auxiliary equation}.\\[-1.5ex]

\noindent \textbf{(Q5)} \emph{Which one from problems \eqref{eq:KKTR} and \eqref{eq:LLVFR} leads to a more efficient algorithm, in terms of number of iterations, computing time, numerical accuracy, and rate of convergence}? Our computational experiments, based on various test problems, including those from the BOLIB Library \cite{BOLIB2017}, show that our semismooth Newton method for \eqref{eq:LLVFR} generally outperforms the one based on the \eqref{eq:KKTR} model, for all the aforementioned performance measures. One of the surprising observation from the numerical computation is that the superiority of \eqref{eq:LLVFR} remains for all the aforementioned measures, even for problem classes where 3rd order derivatives for lower-level problems are non-zero, despite the expectation that in this case, better approximations of the curvature of the lower-level optimal solution set could potentially reduce the number of iterations.\\[-2ex]

Overall, as it will be clear from the analysis in the remaining sections, the main take away of this paper is that from the theoretical point of view, it is not possible to claim that one of the reformulations is better than the other, although \eqref{eq:KKTR} seems to provide a framework for more tractable qualification conditions for optimality conditions an convergence analysis, thanks to its relationship to MPCCs. However, from the numerical perspective, the framework and test problems considered in this paper suggest that problem \eqref{eq:LLVFR} is a much better option.

After addressing (Q1) in Subsection \ref{Nature of reformulations and relationships to original problem}, the general framework for the analysis of the other questions is introduced in Subsection \ref{Framework for optimality conditions and numerical comparison}. Questions (Q2), (Q3), (Q4), and (Q5) are then addressed, in this order, in the subsequent parts of the paper.

\section{Links to the original problem and framework for comparison}
To start this section, we introduce some notation that will be used throughout the paper. Namely, we associate a number of index sets to the inequality constraints involved in problems \eqref{eq:KKTR} and \eqref{eq:LLVFR}. For instance, for a point  $(\bar x, \bar y)$  in the upper-level feasible set of problem \eqref{eq:P}, we denote the indices of the constraints active at this point by
\begin{equation}\label{I1}
   I^1 := I^G(\bar x, \bar y) :=  \left\{i~|\;\, G_i(\bar x,\bar y)=0\right\}.
\end{equation}
Since part of the analysis to be conducted in this paper will be based on the stationary points of each of the reformulations above, we associate to a point $(\bar x, \bar y)$ in the upper-level feasible set, a Lagrange multiplier $\bar u$. Then, considering the fact that the optimality conditions for problem \eqref{eq:KKTR} or \eqref{eq:LLVFR} will lead to the complementarity system $\bar u \geq 0$, $G(\bar x, \bar y)\leq 0$, $\bar u^\top G(\bar x, \bar y)=0$, we partition the corresponding indices in the following standard way:
\begin{equation}\label{multiplier sets}
\begin{array}{l}
  \eta^1   :=   \eta^G(\bar x,    \bar y ,  \bar u) := \{i~|\;\, \bar u_i =0, \;\, G_i(\bar x,\bar y)<0\},\\
\theta^1   :=   \theta^G(\bar x,  \bar y ,  \bar u) := \{i~|\;\, \bar u_i =0, \;\, G_i(\bar x,\bar y)=0\},\\
   \nu^1   :=   \nu^G(\bar x,     \bar y ,  \bar u) := \{i~|\;\, \bar u_i >0, \;\, G_i(\bar x,\bar y)=0\}.
\end{array}
\end{equation}
Similarly, considering the constraint $g(x, y)\leq 0$ appearing in \eqref{eq:KKTR} and \eqref{eq:LLVFR}, as well as $z \geq 0$ in \eqref{eq:KKTR}, the corresponding index sets at $(\bar x, \bar y)$ (resp. $(\bar x,\bar y , \bar v)$), $\bar z$ (resp. $(\bar z, \bar w)$), and $(\bar x, \bar z)$  (resp. $(\bar x, \bar z , \bar w)$), where $\bar v$, $\bar w$, and $\bar w$ represent the Lagrange multipliers, are respectively defined as %are presented below in the order of reformulations \eqref{eq:KKTR} and \eqref{eq:LLVFR}, where they appear:
\begin{equation}\label{nu2nu3}
    \begin{array}{llll}
 I^2:= I^g(\bar x,\bar y), & \eta^2:= \eta^g(\bar x,\bar y , \bar v), & \theta^2:=\theta^g(\bar x,\bar y , \bar v), & \nu^2:=\nu^g(\bar x,\bar y , \bar v),\\
 I^3:= I^{z}(\bar z), & \eta^3  := \eta^z(\bar z, \bar w), & \theta^3  :=  \theta^z(\bar z, \bar w), & \nu^3  :=  \nu^z(\bar z, \bar w),\\
I^4 := I^g(\bar x, \bar z), & \eta^4 := \eta^g(\bar x, \bar z, \bar w), & \theta^4:=\theta^g(\bar x, \bar z, \bar w), & \nu^4 := \nu^g(\bar x,\bar z , \bar w)
\end{array}
\end{equation}
with the first and second lines here being associated to \eqref{eq:KKTR} and the last one related to problem \eqref{eq:LLVFR}.
For vectors $d^i\in \mathbb{R}^n$, $d^j\in \mathbb{R}^m$, and $d^k \in \mathbb{R}^p$, for example, with $i, j, k\in \mathbb{N}$,  $d^{ij}$ and $d^{ijk}$ represent the combined vectors
\[
d^{ij} := \left[\begin{array}{c}
                  d^i\\
                  d^j
                \end{array}\right] \;\; \mbox{ and } \;\; d^{ijk} := \left[\begin{array}{c}
                  d^i\\
                  d^j\\
                  d^k
                \end{array}\right],
\]
respectively. For a function $\psi: \mathbb{R}^{\tilde{n}}\times \mathbb{R}^{\tilde{m}}\times \mathbb{R}^{\tilde{p}} \rightarrow \mathbb{R}^{\tilde{q}}$, $\nabla_i \psi(a, b, c)$ with $i\in \{1, 2, 3\}$, corresponds to the gradient of $\psi$ w.r.t. the $i$th variable $a$, $b$ or $c$. Furthermore, unless otherwise stated, $\nabla_{i, j}\psi(a, b, c)$ (resp. $\nabla^2_{ij}\psi(a, b, c)$) with $i, j\in \{1, 2, 3\}$, denotes the gradient (resp. second order derivative) of $\psi$ w.r.t. the $i$th and $j$th variables. Note that in the sequel, the corresponding versions of the function $\psi$ could have 2, 3, 4 or more variables. In those cases, the same logic presented here will be used.

\subsection{Nature of reformulations and relationships to original problem} \label{Nature of reformulations and relationships to original problem} %Still to do in this subsection:
%\begin{enumerate}
%  \item Replace Slater' CQ with LMFCQ as lower-level regularity conditions and check whether this changes anything in the statement of Dempe and Dutta's paper. Also include any other lower-level regularity here; namely the LLICQ (needed for Theorm 2.4(iii) and the strict complementarity condition.
%  \item Check whether the Theorem 2.2. works with the new formulation of $X$ or whether including this set is even appropriate or not. Or maybe we just need to treat the problem as the bilevel optimization problem with coupled upper-level constraint.
%  \item What does Example 2.1 actually means? How to introduce it in a powerful manner? What does it illustrate?
%  \item Clarify the statement of Theorem 2.4. Then clarify and emphasize the fact that the resulting smoothness is only local. Then find an example for Example 2.2, where $\varphi$ is smooth in a neighborhood of a point but not in the neighborhood of the point of interest, which would be optimal for \eqref{eq:P}.
%  \item Clarify Theorem 2.5 and clarify the fact that the requirements needed to ensure that $\varphi$ is locally Lipschitz continuous are comparable to the ones needed to ensure that \eqref{eq:KKTR} is related to \eqref{eq:P}.
%  \item Update Table 1 to include information on the differentiability and Lipschitz continuity of \eqref{eq:LLVFR}.
%  \item For the whole of the paper, enter all the differentiability assumptions at the required level in the corresponding results.
%\end{enumerate}
We start this subsection by looking at the relationships between \eqref{eq:P} and its reformulations \eqref{eq:KKTR} and \eqref{eq:LLVFR}. For the link between \eqref{eq:KKTR} and \eqref{eq:P}, we need two properties; i.e., convexity and a lower-level regularity condition. These assumptions are needed to help ensure that inclusion $y\in S(x)$ can be written in terms of the KKT conditions present in the feasible set of problem \eqref{eq:KKTR}.

\begin{defn}[lower-level convexity]
The lower-level optimization problem \eqref{lower-level problem} is said to be \emph{convex} if the functions $f(x, .)$ and $g_i(x, .)$, $i=1, \ldots, q$ are convex for all $x\in X$. The problem will be said to be \emph{fully convex} if the latter functions are convex w.r.t. $(x,y)$.
\end{defn}
{ {For the lower-level regularity, we use the standard \emph{lower-level} Mangasarian-Fromowitz constraint qualification (LMFCQ), which holds at $(\bar x, \bar y)$ if there exits $d$ such that
\begin{equation}\label{LMFCQ}
 \nabla_2 g_i(\bar x, \bar y)^\top d < 0 \;\, \mbox{ for all } \;\, i\in I^2.
\end{equation}}}
%lower-level Slater constraint qualification (LSCQ). The LSCQ is said to be satisfied for the lower-level problem at $\bar{x}\in X$ if there exists some vector $y(\bar x)$ such that  $g_i(\bar x, y(\bar x))<0$, $i=1, \ldots, q$.
For a point $(x, y)$ such that $G(x,y)\leq 0$ and $(x,y)\in \text{gph}\,S$, it is well-known that if the LMFCQ holds at $(x, y)$, then the following set of lower-level Lagrange multipliers is non-empty:
\begin{equation}\label{Lambda(x,y)}
\Lambda(x, y):=\left\{\left. z\in \mathbb{R}^q\right|\;\, \nabla_2\ell(x, y, z)=0, \;\, z\geq 0, \;\, g(x, y)\leq 0, \;\, z^{\top}g(x, y)=0\right\}.
\end{equation}
We have the following result established in \cite{DempeDuttaBlpMpec2010}. %when the upper-level feasible set is independent from $y$. The proofs there remain valid in this theorem, where $X$ defined as in \eqref{Upper-level-feasible set}.
\begin{thm}[local and global relationship between \eqref{eq:KKTR} and \eqref{eq:P}]\label{DempeDutta-Stuff}Let $G$ be independent from $y$ and $f(x, .)$ and $g_i(x, .)$, $i=1, \ldots, q$ be convex and  $\mathcal{C}^1$  for all $x\in X$. Then, the following statements hold:
\begin{itemize}
  \item[(i)] Let $(\bar x, \bar y)$ be  globally (resp. locally) optimal for \eqref{eq:P} and the LMFCQ be satisfied at $(\bar x, y)$, $y\in S(\bar x)$. Then, for each $z \in \Lambda(\bar x, \bar y)$, the point $(\bar x, \bar y, z)$ is a global (resp. local) optimal solution of \eqref{eq:KKTR}.
  \item[(ii)] Let the LMFCQ hold at all $(x, y)$, $y\in S(x)$, $x\in X$ (resp. at $(\bar x, y)$, $y\in S(\bar x)$) and $(\bar x, \bar y, z)$ be a global (resp. local) optimal solution (resp. for all $z \in \Lambda(\bar x, \bar y)$) of  \eqref{eq:KKTR}, then the point $(\bar x, \bar y)$ is a global (resp. local) optimal solution of problem \eqref{eq:P}.
\end{itemize}
\end{thm}
%\begin{proof}
%The proof of this result is provided in \cite{DempeDuttaBlpMpec2010} in the upper-level feasible set $X$ is independent from $y$. One can easily check that the result remains valid with $X$ defined as in \eqref{Upper-level-feasible set}.
%\end{proof}
It is shown in \cite{DempeDuttaBlpMpec2010} that this result is very sensitive to convexity and the constraint qualification, as it typically fails if one of these assumptions does not hold. More details on the results can be found in that paper; some of the key issues faced when attempting to solve \eqref{eq:P} via \eqref{eq:KKTR} can be found in the next examples below. Before that, we state the equivalence between \eqref{eq:P} and \eqref{eq:LLVFR}, which is valid locally and globally without any assumption (apart from requiring that $\mbox{gph}S \neq \emptyset$, assumed to be valid throughout the paper as stated in the introduction), given that the lower-level optimal solution set-valued mapping $S$ \eqref{S Map} can be equivalently written as
$$
S(x):=\left\{y\in \mathbb{R}^m\,|\; g(x,y)\leq 0, \;\; f(x,y) - \varphi(x) \leq 0 \right\} \;\;\mbox{ for all }\;\, x\in X.
$$
\begin{thm}[local and global relationship between (LLVF) and \eqref{eq:P}]\label{LLVFR eq P}
$(\bar x, \bar y)$ is a local (resp. global) optimal solution of \eqref{eq:LLVFR} if and only if the point is a local (resp. global) optimal solution of \eqref{eq:P}.
\end{thm}

%We now use a few examples to illustrate differences between problems \eqref{eq:KKTR} and \eqref{eq:LLVFR}.

\begin{exmp}[problem \eqref{eq:LLVFR} has an optimal solution but \eqref{eq:KKTR} does not have one]
Consider an example of problem \eqref{eq:P} from \cite{DempeDuttaBlpMpec2010}  with the corresponding functions defined by
\begin{equation*}\label{Example1}
 F(x,y) := x, \;\, G(x,y):=-x, \;\, f(x,y):= y_1, \;\, \mbox{ and } \;\, g(x,y):= \left(y^2_1 - y_2 - x, \; y^2_1 + y_2\right)^\top.
\end{equation*}
For this example, it is shown in \cite{DempeDuttaBlpMpec2010} that $(\bar x, \bar y)=(0, 0)$ is the global optimal solution of \eqref{eq:P} and hence of problem \eqref{eq:LLVFR} (cf. Theorem \ref{LLVFR eq P}), but \eqref{eq:KKTR} does not have any solution.
\end{exmp}

%----------------------------------------------
\begin{exmp}[problems \eqref{eq:LLVFR} and \eqref{eq:KKTR} both have optimal solutions, but which are completely different from each other] Consider an example of problem \eqref{eq:P} from \cite{DempeDuttaBlpMpec2010}  with the corresponding functions defined by
\begin{equation}\label{Example1}
F(x,y):= (x-1)^2 + y^2, \;\; f(x,y):= x^2y, \; \mbox{ and } \; g(x,y):= y^2.
\end{equation}
The global optimal solution of the corresponding problem \eqref{eq:P} is $(\bar x, \bar y)=(1, 0)$. The feasible sets of \eqref{eq:KKTR} and \eqref{eq:LLVFR} can be  respectively obtained as
\[
S_1 = \left\{\left.(x,y,z)\right|z\in \mathbb{R}_+,\;\,x=y=0\right\}\; \mbox{ and } \; S_2 = \left\{\left.(x,y)\right|~x\in\mathbb{R},\;\, y=0\right\}.
\]
Obviously, $(1, 0)$ is also the optimal solution of \eqref{eq:LLVFR}. However, the global optimal solution of \eqref{eq:KKTR} is $(0, 0, z)$ for any $z\geq 0$. It is pointed out in \cite{DempeDuttaBlpMpec2010} that this is due to the failure of the lower-level regularity condition; cf. Theorem \ref{DempeDutta-Stuff}.
This clearly shows that the optimal solution of problem \eqref{eq:KKTR} need not be optimal for \eqref{eq:P}, while this cannot be the case for \eqref{eq:LLVFR}; cf. Theorem \ref{LLVFR eq P}.\\[1ex]
\end{exmp}
%--------------------------------------------

\begin{exmp}[problems \eqref{eq:LLVFR} and \eqref{eq:KKTR} both have optimal solutions which are equivalent in the sense of Theorem \ref{DempeDutta-Stuff}]\label{Example with pics} Here, we consider a variant of the problem described in \eqref{Example1}; i.e.,
\begin{equation}\label{Example1-new}
 F(x,y):= (x-1)^2 + y^2, \;\; f(x,y):= x^2y, \; \mbox{ and } \; g(x,y):= y^2-1,
\end{equation}
where the lower-level constraint function is slightly modified. The lower-level problem is convex and the LMFCQ holds at any lower-level feasible point. The points $(\bar x, \bar y) = (1, -1)$ and  $(\bar x, \bar y)=(0, 0)$ are global optimal solutions of the version of \eqref{eq:P} defined in \eqref{Example1-new} and hence of the corresponding \eqref{eq:LLVFR}. The feasible sets of \eqref{eq:KKTR} and \eqref{eq:LLVFR} in the context of \eqref{Example1-new} are respectively given by %the following sets:
\[
S_1 = \left\{\left.(x,y,z)\right|~\left(x=z=0, \, y^2\leq 1\right)~{\rm or}~\left(x^2=2z~\geq 0, \, y=-1\right)\right\}\, \mbox{ and }\, S_2 = \left\{\left.(x,y)\right|~x^2y+x^2=0, \, y^2 \leq 1\right\}.
\]
Clearly, both $(\bar x, \bar y, \bar z)=(1, -1, 0.5)$ and $(\bar x, \bar y, \bar z)=(0, 0, 0)$ are optimal solutions of \eqref{eq:KKTR}. The feasible sets in Figure $\ref{FS}$ clearly exhibit the differences between the sizes of problems \eqref{eq:KKTR} and \eqref{eq:LLVFR}; the former has more variables and constraints than the latter. { {Table \ref{Table 1}, provided later in this section, contains more general details on the dimensions of the problems.}}
\end{exmp}
\begin{figure}[H]
\centering
    \includegraphics[width=0.50\linewidth]{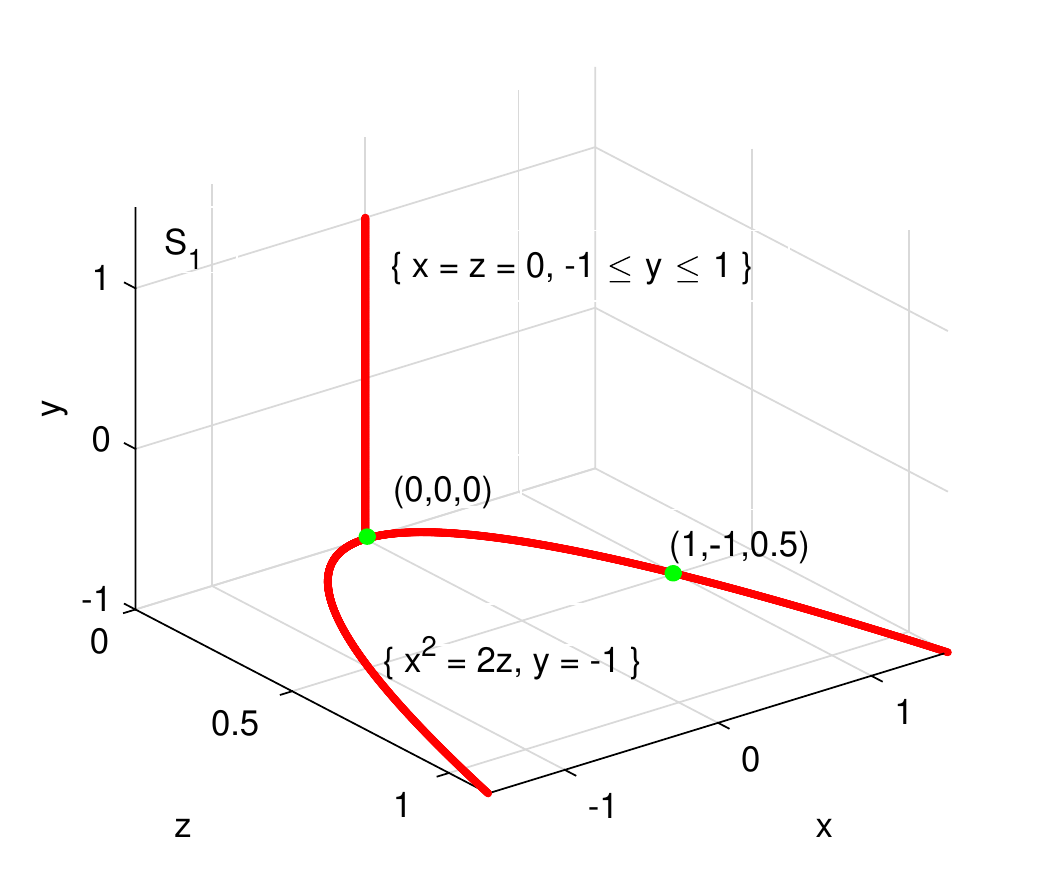}
    \includegraphics[width=0.44\linewidth]{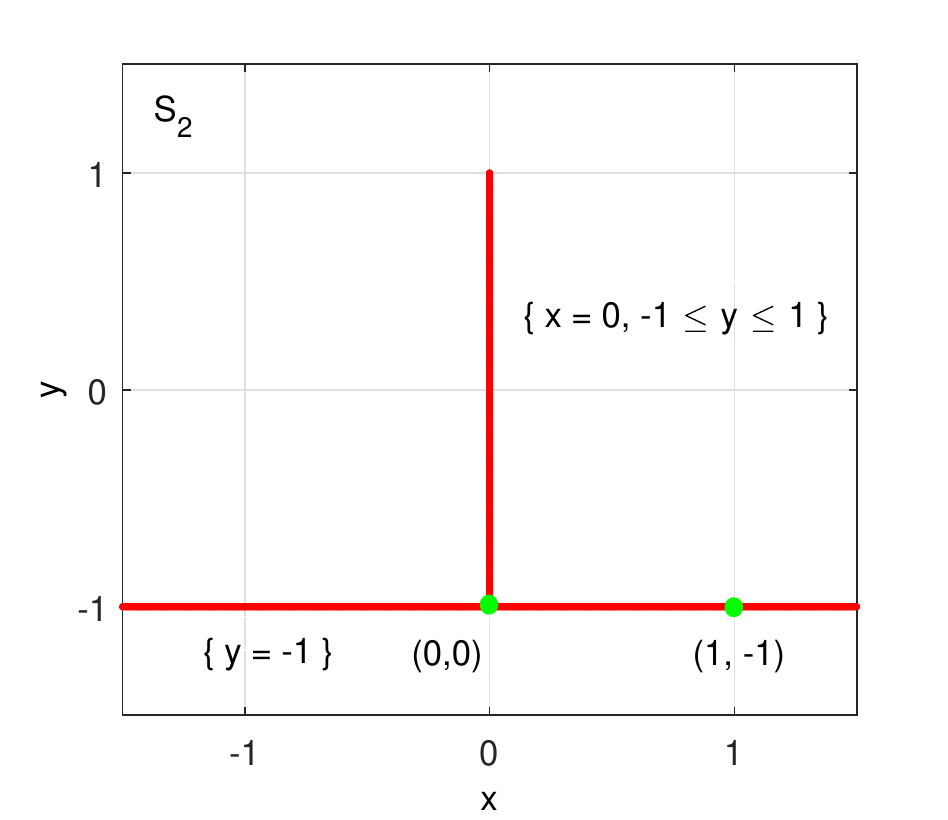}
  \caption{Feasible sets to \eqref{eq:KKTR} and \eqref{eq:LLVFR}. The red lines represent the feasible sets and green points stand for the global optimal solutions.}
  \label{FS}
\end{figure}

Now, we go beyond the relationship analysis  above and start to consider the question of how to solve problems \eqref{eq:KKTR} and \eqref{eq:LLVFR} and how the processes of getting to a \emph{solvable} model compare to each other. To design a solution algorithm for an optimization problem, the nature of the functions involved is critical in choosing the most suitable approach. For instance, we would like to look at the framework, based on problem \eqref{eq:P} data, ensuring that the reformulations under consideration are smooth ($\mathcal{C}^1$) or Lipschitz continuous optimization problems.  Obviously, \eqref{eq:KKTR} is a smooth optimization problem if the following conditions are satisfied:
\begin{equation}\label{smooth KKTR}
\left\{\begin{array}{l}
    F \; \mbox{ and } \; G \; \mbox{ are } \; \mathcal{C}^1;\\
   f(x, .) \; \mbox{ and } \; g(x, .) \; \mbox{ are } \; \mathcal{C}^1 \; \mbox{ for all } x\in \mathbb{R}^n;\\
   \nabla_2 f \; \mbox{ and } \; \nabla_2 g \; \mbox{ are } \; \mathcal{C}^1 \; \mbox{ in } (x, y).
  \end{array}\right.
\end{equation}
On the other hand, to ensure that \eqref{eq:LLVFR} is smooth, the following conditions need to hold:
\begin{equation}\label{smooth LLVFR}
  \begin{array}{l}
    F, \; G,\; f, \; g,\; \mbox{ and } \; \varphi \; \mbox{ are } \; \mathcal{C}^1.
  \end{array}
\end{equation}
Unfortunately, it is quite demanding to ensure that $\varphi$ is a smooth function. To guaranty that this is the case, we present the following result due to Fiacco \cite{FiaccoBook1983}.

For this result, we need another, much stronger, lower-level regularity condition. Namely, the lower-level linear independence constraint qualification (LLICQ) that will be said to hold at $(\bar x, \bar y)$ if the following family of gradients is linearly independent:
\begin{equation}\label{LLICQ}
  \left\{\left.\nabla_2 g_i(\bar x, \bar y)\,\right|\; i\in I^2\right\}.
\end{equation}
It is well-known that under this condition, the set of lower-level Lagrange multipliers  \eqref{Lambda(x,y)} is single-valued at the point $(\bar x, \bar y)$. Furthermore, considering $(\bar x, \bar z, \bar v)$ such that $\bar v \in \Lambda(\bar x, \bar z)\neq \emptyset$, the lower-level strict complementarity condition (LSCC) will be said to hold at this point if
\begin{equation}\label{LSCC}
 \theta^g(\bar x, \bar y, \bar v) = \emptyset.
\end{equation}
\begin{thm}[continuous differentiability of $\varphi$]\label{Differentiability of phi}
Let the functions $f$ and $g$ be $\mathcal{C}^2$. Suppose that the LLICQ and holds at $(\bar x, \bar y)$ and the lower-level  second order sufficient condition (LSOSC)
\begin{equation}\label{SSOC}
d^{\top}\nabla_{22}^2 \ell(\bar x, \bar y, \bar v)d >0, \;\, \forall d\neq 0 \;\, \mbox{ s.t. }\;\,\nabla_2 f(\bar x, \bar y)^\top d=0, \;\, \nabla_2 g_i(\bar x, \bar y)^\top d\leq 0 \; \mbox{ for } \; i\in I^2,
%\begin{array}{l}
%\nabla_y g_i(\bar x, \bar y)d=0, \mbox{ for each } i\in J:=\{j| \lambda_j>0\}\\
%\nabla_y h_j(\bar x, \bar y)d=0, \mbox{ for } j=1, \ldots, q
%\end{array}
\end{equation}
%we have
%$$
%d^{\top}\nabla_{yy}^2 L(\overline{x},\overline{y},\lambda,\mu)d >0;
%$$
with $\Lambda(\bar x, \bar y)=\left\{\bar v\right\}$, holds. Furthermore, if the LSCC holds at $(\bar x, \bar y, \bar v)$, then, the function $\varphi$ is continuously differentiable around the point $\bar x$.
%\begin{itemize}
%  \item[(i)] The functions $f$ and $g$ are $\mathcal{C}^2$;
%  \item[(ii)] (SSOC) is said to hold at $(\overline{x},\overline{y})$ if, for each $(\lambda,\mu)\in \Lambda(\overline{x},\overline{y})$ and for each $d\neq 0$ satisfying
%\begin{equation}\label{SSOC}
%\begin{array}{l}
%\nabla_y g_i(\bar x, \bar y)d=0, \mbox{ for each } i\in J:=\{j| \lambda_j>0\}\\
%\nabla_y h_j(\bar x, \bar y)d=0, \mbox{ for } j=1, \ldots, q
%\end{array}
%\end{equation}
%we have
%$$
%d^{\top}\nabla_{yy}^2 L(\overline{x},\overline{y},\lambda,\mu)d >0;
%$$
%\item[(iii)] The lower-level linear independence constraint qualification (LLICQ) holds at $(\bar x, \bar y)$:
%the gradients $\nabla_y g_i(\bar x, \bar y)$, for $i\in I$, are linearly independent.
%\item[(iv)] The strict complementarity conditions holds: $\theta = \emptyset$.
%\end{itemize}
\end{thm}

Obviously, these conditions are far much stronger than the ones needed for the KKT reformulation \eqref{eq:KKTR}. Moreover, even if one adds the assumptions of Theorem \ref{DempeDutta-Stuff} to \eqref{smooth KKTR}, the framework for \eqref{eq:LLVFR} to be smooth is still far more demanding than the one necessary for building \eqref{eq:KKTR} and ensuring that it is a smooth optimization problem. Furthermore, it is important to observe that the assumptions in Theorem \ref{Differentiability of phi} ensure the differentiability of $\varphi$ \eqref{varphi} only locally. Hence, it could happen that they hold everywhere but at the point of interest for a given bilevel optimization, as shown in the following example.
\begin{exmp}[$\varphi$ can be continuously differentiable everywhere but at the point of interest]
Consider the bilevel optimization problem \eqref{eq:P} in the case where
\begin{equation}\label{Exp2}
F(x,y):=x^2 + (y-1)^2, \;\; G(x,y):=x, \;\; f(x,y):= x(y-1), \;\mbox{ and } \; g(x,y):=\left(-y, \; y-1\right)^\top.
\end{equation}
One can easily check that the optimal solution and value mappings can be written as %The lower-level of this problem is taken from \cite{DempeFoundations2002} and its optimal solution map can be obtained as
\[
S(x)=\left\{\begin{array}{lll}
              0 & \mbox{ if } & x>0,\\
               1 & \mbox{ if } & x<0,\\
               \left[0, 1\right] & \mbox{ if } & x=0,
            \end{array}
\right. \qquad \mbox{ and } \qquad \varphi(x)=\left\{\begin{array}{rll}
             -x & \mbox{ if } & x>0,\\
              0 & \mbox{ if } & x\leq 0,
            \end{array}
\right.
\]
respectively. Obviously, $\varphi$ is continuously differentiable everywhere, except at the point $\bar x=0$. Also, it is clear that the point $(\bar x, \bar y)=(0, \, 1)$ is globally optimal for our bilevel program. Looking at the nondifferentiability of $\varphi$ at $0$ in the context of Theorem \ref{Differentiability of phi}, the LLICQ holds at $(0, 1)$, but the LSOSC fails.
%
%We can check that the family of points $(x,y, v)$ such that $v\in \Lambda(x,y)$ with $y\in S(x)$ can be described as
%%$$
%%\begin{array}{lll}
%%  x=0 & y=0 & \Lambda(x,y)=\emptyset \\
%%  x=1 & y=1 & u=2 \\
%%   x\in]0,1[\cup]1, 9[ & y=\sqrt{x} & u=\frac{3}{\sqrt{x}} -1 >0 \\
%%  x\geq 9 & y=3 & u=0
%%\end{array}
%%$$
%$$
%\begin{array}{lll}
%  x=0, & y=0, & \Lambda(x,y)=\emptyset, \\
%   x\in (0,\, 9), & y=\sqrt{x}, & v=\frac{3}{\sqrt{x}} -1 >0, \\
%  x\geq 9, & y=3, & v=0.
%\end{array}
%$$
%Observing that
%$$
%\nabla^2_{22} \ell(x,y,v)= 2(1+v) \; \mbox{ and } \; \nabla_2 g(x,y)= 2y,
%$$
%it becomes obvious that all the assumptions of Theorem \ref{Differentiability of phi} hold for all $(x,y, v)$ such that $v\in \Lambda(x,y)$ with $y\in S(x)$ except $(x, y, v)$ with $x=9$, $y=3$, and $v=0$, where the strict complementarity conditions fails. Indeed, the map of $\varphi$ is continuously differentiable everywhere on $(0, \infty)$ except at $x = 9$:
%$$
%\varphi(x)=\left\{\begin{array}{lll}
%              \infty & \mbox{ if } & x<0,\\
%               (\sqrt{x} - 3)^2 & \mbox{ if } & 0\leq x \leq 9,\\
%                0 & \mbox{ if } & x>9.
%            \end{array}
%\right.
%$$
%Unfortunately, for the problem described in \eqref{Exp2}, the point $(x,y)=(9, 3)$  is the optimal solution.
\end{exmp}

{ {The following collection of well-known results provides two main scenarios in which $\varphi$ can be locally Lipschitz continuous, see, e.g., \cite{BonnansShapiroBook2000, MordukhovichBook2006, RockafellarWetsBook1998} for details.}} %To proceed, we need another, much weaker, lower-level regularity condition; i.e.,  the lower-level Mangasarian-Fromowitz constraint qualification (LMFCQ), which holds at $(\bar x, \bar y)$ if there exits $d$ such that
%\begin{equation}\label{LMFCQ}
% \nabla_2 g_i(\bar x, \bar y)^\top d < 0 \;\, \mbox{ for all } \;\, i\in I^2.
%\end{equation}
For the second scenario of the result, we will need some continuity requirement on $S$ \eqref{S Map}. Namely, the set-valued mapping $S$ will be said to be \emph{inner semicontinuous} at $(\bar{x}, \bar{y})$, with $\bar y\in S(\bar x)$, if for every sequence $x^k\rightarrow \bar{x}$, there is a sequence of $y^k\in S(x^k)$ that converges to $\bar{y} $ as $k\rightarrow \infty$.
\begin{thm}[local Lipschitz continuity of the optimal value function]\label{Lipschitz continuity of phi} Considering the function $\varphi$ defined in \eqref{varphi}, the following statements hold true:
\begin{itemize}
  \item[(i)] If the lower-level problem is fully convex, then $\varphi$ is locally Lipschitz continuous. If additionally, LMFCQ holds at $(\bar x, \bar y)$, then we have
  \begin{equation}\label{Subdifferential of phi}
    \partial \varphi (\bar x) \; \subseteq \; \left\{\left.\nabla_1 \ell(\bar x, \bar y, v)\,\right| \;\; v\in \Lambda (\bar x, \bar y)\right\}.
  \end{equation}
  \item[(ii)] If $S$ is inner semicontinuous at $(\bar x, \bar y)\in \text{gph}~S$ and the LMFCQ holds at $(\bar x, \bar y)$, then $\varphi$ is Lipschitz continuous around $\bar x$ and moreover, inclusion \eqref{Subdifferential of phi} holds.
\end{itemize}
\end{thm}
This result shows that the framework to ensure that $\varphi$ is just locally Lipschitz continuous, with corresponding subdifferential formula, is closely related to the setup required to establish that problem \eqref{eq:KKTR} is closely related to \eqref{eq:P}; cf. Theorem \ref{DempeDutta-Stuff}. However, the conditions in Theorem \ref{Lipschitz continuity of phi} are far much weaker than the ones needed in Theorem \ref{Differentiability of phi}. In fact, it is common in the literature to analyse \eqref{eq:LLVFR} as a Lipschitz optimization problem. Hence, to make sure that our comparisons are fair, from here on, we will treat \eqref{eq:LLVFR} as such, while the framework ensuring that \eqref{eq:KKTR} is a continuously differentiable optimization problem will be considered.

To close this subsection, we summarize the main points about problems \eqref{eq:KKTR} and \eqref{eq:LLVFR}, and their distinctive features, in the following table (Table \ref{Table 1}). Note that there are many other types of conditions ensuring that $\varphi$ is locally Lipschitz continuous and leading to estimates of its subdifferential; see the aforementioned references.
\begin{table}[!th]
\begin{center}
\addtolength{\leftskip} {-2cm} % increase (absolute) value if needed
\addtolength{\rightskip}{-2cm}
\begin{tabular}{p{35mm}|l|l|l}
\hline
\multicolumn{2}{c|}{${}$}   & \textbf{KKTR}&    \textbf{LLVFR}      \\
\hline
\multirow{3}{*}{\textbf{Model} \textbf{requirements}} & $f$ and $g$ are $\mathcal{C}^1$ w.r.t. $y$      & \cmark   &         \xmark     \\
                                              & lower-level convexity              & \cmark     &         \xmark     \\
                                              & lower-level regularity             & \cmark    &         \xmark     \\
\hline
\multirow{1}{*}{\textbf{Relationship to } \eqref{eq:P}}& locally/globally equivalent       & \textbf{?}   &         \cmark     \\
\hline
%\multirow{6}{*}{\textbf{Reformulation's} \textbf{nature}}
&                 &  $F$ and $G$ are $\mathcal{C}^1$   &      $F$ and $G$ are $\mathcal{C}^1$     \\
&                 & $f$ and $g$ are $\mathcal{C}^2$ w.r.t. $y$   &    $f$ and $g$ are $\mathcal{C}^2$      \\
\textbf{Reformulation's}&  smooth problem &  $g$ is $\mathcal{C}^1$           &  LSOSC    \\
 \textbf{nature}&                 &              &    LLICQ     \\
&                 &              &   LSCC   \\\cline{2-4}
%&                 &              &      \\
& \multirow{2}{*}{Lipschitz continuous} &   \multirow{2}{*}{---}         &    $f$ and $g$ convex    \\
&                      &             &   or MFCQ + $S$ isc    \\
\hline
\multirow{2}{*}{\textbf{Problem size}} & number of variables      & $n+m+q$      &           $n+m$    \\
                              & number of constraints    & $p + 2q + m + 1$       &       $p+q+1$     \\
\hline
\end{tabular}
\end{center}
\caption{Summary for requirements needed to derive KKT and LLVF reformulations; links to original problem; main characteristics of the reformulations and snapshot of requirements for reformulations to be continuously differentiable or locally Lipschitz continuous problems. $isc$ stands for \emph{inner semicontinuity}.}\label{Table 1} %\\[-1.5ex]
\end{table}

\subsection{Framework for optimality conditions and numerical comparison}\label{Framework for optimality conditions and numerical comparison}
In this subsection, we present the general framework that will be used to solve \eqref{eq:KKTR} and \eqref{eq:LLVFR}; going from the corresponding optimality conditions to the Newton-type method to solve the problem. To proceed, let $\tilde{n}, \tilde{p}, \tilde{q}\in \mathbb{N}$ be given such that $\tilde{n}, \tilde{p}, \tilde{q}\geq 1$ and consider the optimization problem
\begin{equation}\label{Basic problem}
  \min~\tilde{f}(x)\;\;\mbox{ s.t. }\;\; \gamma(x) = 0, \;\; \tilde{g}(x)\leq 0,\;\; \tilde{h}(x)=0,
\end{equation}
where the functions $\tilde{f}:\mathbb{R}^{\tilde{n}} \rightarrow \mathbb{R}$, $\tilde{g}:\mathbb{R}^{\tilde{n}} \rightarrow \mathbb{R}^{\tilde{p}}$, and $\tilde{h}:\mathbb{R}^{\tilde{n}} \rightarrow \mathbb{R}^{\tilde{q}}$ are continuously differentiable and $\gamma:\mathbb{R}^{\tilde{n}} \rightarrow \mathbb{R}$ is locally Lipschitz continuous. The setup of problem \eqref{Basic problem} is such that the presence of the constraint $\gamma(x)=0$ can potentially lead to the failure of standard constraint qualifications. Hence, our aim is to remove this from the feasible, in order to get a more tractable feasible set for the problem. To proceed, we use the concept of partial calmness \cite{YeZhuOptCondForBilevel1995,YeZhuZhuExactPenalization1997}  defined as follows.
\begin{defn}[partial calmness condition]\label{partial-calmness}  Let $\bar x$ be a local optimal solution of problem \eqref{Basic problem}. Problem \eqref{Basic problem} is said to be partially calm on $\gamma$ at $\bar x$ provided that there exist $\delta>0$ and $\lambda>0$ such that for all $\sigma\in U(0,\delta)$ and all $x\in U(\bar x, \delta)$ with
$$
\gamma(x)+\sigma= 0,  \ \   \tilde{g}(x)\leq 0,\;\; \tilde{h}(x)=0,
$$
we have
$$
\tilde{f} (x) - \tilde{f}(\bar x)+\lambda |\sigma|\geq 0,
$$
where $U(z,\delta)$ is a neighbourhood of $z$ with radius $\delta$, i.e., $U(z,\delta):=\{x\in\mathbb{R}^n~|~\|x-z\|<\delta\}$.
\end{defn}
Based on this, the following partial penalization from \cite[Proposition 2.2]{YeZhuZhuExactPenalization1997} can be used to move constraint $\gamma(x)=0$ from the feasible set to the objective function to get a tractable feasible set.
\begin{thm}[partial exact penalization]\label{equivalen}  Let $\bar x$ is a local minimizer of problem \eqref{Basic problem}. Then  this problem is partially calm on $\gamma$ at $\bar x$ if and only if there exists $\lambda \in (0, \infty)$ such that $\bar x$ is also a local optimal solution of the following problem:
\begin{eqnarray}\label{partial-penalty-Pc}
\begin{array}{l}
  \underset{x}\min~\tilde{f}(x)+\lambda |\gamma(x)| \;\;   \mbox{ s.t. } \;\; \tilde{g}(x)\leq 0,\;\; \tilde{h}(x)=0.
\end{array}
\end{eqnarray}
\end{thm}

To derive the optimality conditions for problem \eqref{partial-penalty-Pc}, generalized differentiation tools will be needed, considering the potential nonsmoothness of $\gamma$. For a function $\psi : \mathbb{R}^n \rightarrow \mathbb{R}$, its directional derivative at $\bar x \in \mathbb{R}^n$, in direction $d \in \mathbb{R}^n$, is the following limit when it exists:
\begin{equation}\label{Directional Derivative}
    \psi'(\bar x; d):=\underset{t\downarrow 0}\lim \frac{1}{t}\left[\psi(\bar x + td)-\psi(x)\right].
\end{equation}
 Differentiable and convex (not necessarily differentiable) functions are examples of directionally differentiable functions \cite{RockafellarConvexAnalBook1970}. The optimal value function $\varphi$ \eqref{varphi} can well be directionally differentiable without necessarily being differentiable nor convex \cite{GauvinDubeau1982}.

Proceeding further, recall that the definition of the usual directional derivative \eqref{Directional Derivative} relies on the existence of a limit, which in fact does not exist for various classes of functions. To extend the concept of directional derivative to a wider class of function, Clarke \cite{ClarkeBook1983} introduced the notion of  {generalized directional derivative}, defined for a function $\psi : \mathbb{R}^n \rightarrow \mathbb{R}$ by
 \begin{equation}\label{Clarke Directional Derivative}
    \psi^o(\bar x; d):=\underset{t\downarrow 0}{\underset{x \rightarrow \bar x}\limsup} \frac{1}{t}\left[\psi(x + td)-\psi(x)\right].
\end{equation}
 This quantity exists if $\psi$ is any function Lipschtz continuous around $\bar x$ \cite[Proposition 2.1.1]{ClarkeBook1983}. Utilizing this notion, Clarke also introduced the  {generalized subdifferential}
\begin{equation}\label{Clarke Subdifferential}
    \partial \psi(\bar x):= \left\{\xi\in \mathbb{R}^n|\; \psi^o(\bar x; d)\geq \langle\xi, d\rangle, \; \forall d\in \mathbb{R}^n  \right\}.
\end{equation}
$\partial \psi(\bar x) = \left\{\nabla \psi(\bar x)\right\}$ if $\psi$ is differentiable at $\bar x$. Also, if $\psi$ is convex, $\partial \psi$ coincides with the subdifferential in sense of convex analysis, which can be defined in way similar to \eqref{Clarke Subdifferential} while instead using \eqref{Directional Derivative}.
Furthermore, note that $\psi$ being Lipschitz continuous around $\bar x$, it is differentiable almost everywhere
around this point; hence the subdiffential \eqref{Clarke Subdifferential} can also be written as
\begin{equation}\label{Clarke Subdifferential-Derivative}
    \partial \psi(\bar x):= \mbox{co}~\left\{\left.\lim~\nabla \psi(x^n)\,\right|\; x^n \rightarrow \bar x, \;\, x^n\in D_\psi  \right\},
\end{equation}
{ {where ``co'' stands for the convex hull and  $D_\psi$ represents the set of points where $\psi$ is differentiable \cite{ClarkeBook1983}.}} The latter concept remains valid for a vector-valued function and is called the generalized Jacobian, with $\nabla \psi$ in \eqref{Clarke Subdifferential-Derivative} denoting the Jacobian of $\psi$ at points where the function is differentiable. Following the expression in \eqref{Clarke Subdifferential-Derivative}, the $B$-subdifferential (see, e.g., \cite{Qi1993convergence}) can be defined by
\begin{equation}\label{B-Subdifferential}
    \partial_B \psi(\bar x):= \left\{\left.\lim~\nabla \psi(x^n)\,\right|\; x^n \rightarrow \bar x, \;\, x^n\in D_\psi  \right\}.
\end{equation}

We are now ready to derive the necessary optimality conditions for problem \eqref{partial-penalty-Pc}. To proceed, recall that a point $\bar x$ feasible to problem \eqref{partial-penalty-Pc} satisfies the MFCQ if the gradients
\begin{equation}\label{MFCQ general}
  \nabla \tilde{h}_i(\bar x), \;\, i=1, \ldots, \tilde{q}  \;\; \mbox{ are linearly independent}
\end{equation}
and there exists a vector $d\in \mathbb{R}^{\tilde{n}}$ such that
\begin{equation}\label{MFCQ general II}
 \begin{array}{l}
   \nabla \tilde{h}_i(\bar x)^\top d=0, \;\, i=1, \ldots, \tilde{q}, \\
   \nabla \tilde{g}_i(\bar x)^\top d < 0, \;\, i\in I(\bar x):=\left\{i|\;\,  \tilde{g}_i(\bar x)=0\right\}.
 \end{array}
\end{equation}
\begin{thm}[necessary optimality conditions] Let $x$ be a local optimal solution of problem \eqref{Basic problem}. Assume that the problem is partially calm at $\bar x$, $\gamma$ is nonnegative and locally Lipschitz continuous around $x$ and the MFCQ holds for the remaining constraints. Then, there exist some $\lambda>0$,  $u$, and  $v$ such that
\begin{eqnarray}
  \nabla \tilde{f}(x) + \lambda \partial \gamma(x) + \nabla \tilde{g}(x)^\top u + \nabla \tilde{h}(x)^\top v\ni 0, \label{opt1-1}\\
  u\geq 0, \;\; u^\top \tilde{g}(x)=0, \;\; \tilde{g}(x)\leq 0, \;\; \tilde{h}(x)=0.\label{opt1-2}
\end{eqnarray}
\end{thm}
Assuming that there exists a function $\vartheta$ such that any $\xi \in \partial \gamma(x)$ can be written as  $\xi= \vartheta(x, w)$, for some $w$,  the conditions \eqref{opt1-1}--\eqref{opt1-2} can be relaxed to a certain system of equations of the form
\begin{equation}\label{Eq-Main}
  \Phi^{\lambda}(x, u, v, w)=0.
\end{equation}
More details on the nature of $\Phi^{\lambda}$ will be clear for each specific reformulation of the bilevel program in the next section. For the reminder of this section, we assume that $\Phi^{\lambda}$ is a semismooth function \cite{Mifflin1977}, which is useful for the convergence result of the Newton method to be discussed in this paper. A locally Lipschitz continuous function  $\psi : \mathbb{R}^{n} \rightarrow \mathbb{R}^{m}$  is semismooth at $\bar x$ if the following limit exists for all $d\in \mathbb{R}^n$:
$$
\lim \left\{Vd'\,|\; V\in \partial \psi (\bar x+ td'), \; d' \rightarrow d, \; t\downarrow 0\right\}.
$$
 If in addition,
$
Vd - \psi'(\bar x; d) = O(\|d\|^2)
$
for all $V\in \partial \psi (\bar x+d)$ with $d \rightarrow 0$, then $\psi$ is said to be strongly semismooth at $\bar x$. $\psi$ is SC$^1$ if it is continuously differentiable and $\nabla \psi$ is semismooth. Also, $\psi$ will be LC$^2$ if $\psi$ is twice continuously differentiable and $\nabla^2 \psi$ is locally Lipschitzian.

If the function $\Phi^\lambda$ is semismooth, then a generalized Newton-type method can be used to solve equation \eqref{Eq-Main}.
One of our main goals in this paper is to solve this system of equations for the corresponding bilevel programs. We will show that \eqref{Eq-Main} when specified for our problems of interest is a square system. This therefore allows for a natural extension of standard versions of the semismooth Newton method (see, e.g., \cite{DeLuca1996, FischerASpecial1992, QiJiangSemismooth1997, QiSun1999, QiSunANonsmoothVersion1993}) to the bilevel optimization setting.
In order to take full advantage of the structure of the function $\Phi^{\lambda}$ \eqref{Eq-Main} for our problems \eqref{eq:KKTR} and \eqref{eq:LLVFR}, we will use the following globalized version of the semismooth Newton method developed by De Luca et al. \cite{DeLuca1996}.
%Recall that there are various other classes of functions generally known as NCP (nonlinear complementarity problem) functions that have been used in the literature to reformulate complementarity conditions into equations; see \cite{Galantai2012} and references therein for an extended list and related properties.
To proceed, we also introduce the merit function
\begin{equation}\label{Merit function}
    \Psi^{\lambda}(\zeta) := \frac{1}{2}\left\|\Phi^\lambda(\zeta)\right\|^2
\end{equation}
of equation \eqref{Eq-Main}, that we assume to be differentiable, as mild assumptions will ensure this for problems \eqref{eq:KKTR} and \eqref{eq:LLVFR}. Hence, permitting the global convergence of this algorithm.

%%%%%%%%%%%%%%%%%%%%%%%%%%%%%%%%%%%%%%%%%%%%%%%%%%%%%%%%%%%%%%%%%%%%%%%%
\begin{algorithm}
\caption{Semismooth Newton method for equation  \eqref{Eq-Main}}
\label{algorithm 1}
\begin{algorithmic}
 \STATE \textbf{Step 0}: Choose $\lambda>0$, $\beta>0$, $\epsilon \geq 0$, $\rho\in(0,1)$, $\sigma\in(0,1/2)$, $t>2$, $\zeta^o$ and set $k:=0$.
 \STATE \textbf{Step 1}: If $\left\|\Phi^{\lambda} (\zeta ^k)\right\|\leq \epsilon$, then stop.
 \STATE \textbf{Step 2}: Choose $W^k\in \partial \Phi^{\lambda}(\zeta ^k)$ and find the solution $d^k$ of the system
  $$W^kd^k=-\Phi^{\lambda}(\zeta ^k).$$
\hspace{1.25cm}If this equation is not solvable or if the condition
  $$  \nabla\Psi^{\lambda}(\zeta ^k)^\top  d^k  \leq - \beta \Vert d^k\Vert^t$$
\hspace{1.25cm}is not satisfied, set $d^k=-\nabla\Psi^{\lambda}(\zeta ^k)$.
 \STATE \textbf{Step 3}: Find the smallest nonnegative integer $s_k$ such that
$$
\Psi^{\lambda}(\zeta^k + \rho^{s_k}d^k)   \; \leq \;  \Psi^{\lambda}(\zeta^k) +2\sigma\rho^{s_k}\nabla\Psi^{\lambda}(\zeta ^k)^\top    d^k.
$$
\hspace{1.25cm} Then set $\alpha_k := \rho^{s_k},\; \zeta ^{k+1} := \zeta ^k + \alpha_k d^k$, $k:=k+1$ and go to \textbf{Step 1}.
\end{algorithmic}
\end{algorithm}
%%%%%%%%%%%%%%%%%%%%%%%%%%%%%%%%%%%%%%%%%%%%%%%%%%%%%%%%%%%%%%%%%%%%%%

Note that the only difference between this algorithm and the original one in \cite{DeLuca1996} is that in Step 0, we also have to provide the partial penalization parameter $\lambda$; cf. \eqref{partial-penalty-Pc}. Also recall that in Step 2, $\partial \Phi^{\lambda}$ denotes the Clarke subdifferential \eqref{B-Subdifferential}. Obviously, equation $W^k d =-\Phi^{\lambda}(\zeta^k)$ has a solution if the matrix $W^k$ is nonsingular. The latter holds in particular if the function $\Phi^{\lambda}$ is \emph{CD-regular}. The function $\Phi^{\lambda}$ is said to be CD-regular at a point $\zeta$ if each element from $\partial\Phi^{\lambda}(\zeta)$ is nonsingular. %These concepts and other definitions useful for the analysis of the algorithm are provided next; see \cite{QiJiangSemismooth1997,QiSunANonsmoothVersion1993} for details.
Using this property, the convergence of Algorithm \ref{algorithm 1} can be established as follows \cite{DeLuca1996,QiJiangSemismooth1997}:
%%%%%%%%%%%%%%%%%%%%%%%%%%%%%%%%%%%%%%%%%%%%%%%%%%%%%%
\begin{thm}\label{convergence result}  Suppose that the functions involved in problem \eqref{Basic problem} are SC$^1$  and let  $\bar \zeta:=(\bar x, \bar u, \bar v)$ be an accumulation point of a sequence generated by Algorithm $\ref{algorithm 1}$ for some parameter $\lambda > 0$. Then $\bar \zeta$ is a stationary point of the problem of minimizing $\Psi^\lambda$, i.e., $\nabla \Psi^\lambda(\bar \zeta)=0$. If $\bar \zeta$ solves $\Phi^\lambda (\zeta)=0$ and the function $\Phi^\lambda$ is CD-regular at $\bar \zeta$, then the algorithm converges to  $\bar \zeta$ superlinearly and quadratically if the functions involved in problem \eqref{Basic problem}  are LC$^2$.
\end{thm}
%Note that problem \eqref{eq:P} is SC$^1$  (resp. LC$^2$) if the functions $F$, $G_i$ with $i=1, \ldots, p$, $f$, and $g_j$ with $j=1, \ldots, q$ are all SC$^1$  (resp. LC$^2$). Also note that problem \eqref{eq:P} being SC$^1$ (resp. LC$^2$) guaranties that  $\Phi^\lambda$ is semismooth (resp. strongly semismooth), cf. \cite{QiJiangSemismooth1997}. Results closely related to Theorem \ref{convergence result} are developed in \cite{FischerASpecial1992,QiJiangSemismooth1997} and many other references therein. Observe that BD-regularity in this theorem be replaced by the stronger \emph{CD-regularity}, which refers to the non-singularity of all matrices in  $\partial \Phi^{\lambda}(\bar \zeta)$.\\
{ {Observe that the CD-regularity in this theorem can be replaced by the weaker \emph{BD-regularity}, referring to the nonsingularity of all matrices from   $\partial_B \Phi^{\lambda}(\bar \zeta)$. In this case, inclusion $W^k\in \partial \Phi^{\lambda}(\zeta ^k)$ appearing in Algorithm \ref{algorithm 1} would just need to be replaced by $W^k\in \partial_B \Phi^{\lambda}(\zeta ^k)$.}}

\section{Necessary conditions for optimality}\label{Necessary conditions for optimality}
Here, we implement the optimality conditions aspect of the previous section on \eqref{eq:KKTR} and \eqref{eq:LLVFR}. We start with the relevant qualification conditions in the next subsection and subsequently, we apply them to derive necessary optimality conditions.%  Partial calmness appears to be the only qualification conditions, which unites both problem, as it can for both problem, under the same conditions. Hence, we use it here to derive the optimality conditions for both problems. In this framework, the MFCQ-type condition needed to derive optimality conditions for the KKT reformulation appears to be stronger than what is needed for the LLVF reformulation. And subsequently, though under a very restrictive framework, we can show that the resulting S-type necessary conditions are stronger than the ones obtained from the LLVF reformulation.
\subsection{Qualification conditions}\label{Qualification conditions}
%Start this subsection with a discussion of the following things:
%\begin{itemize}
%  \item We might need to set X s.t. G(x)\leq 0 and not G(x,y)\leq 0 especially for conditions ensuring partial calmness
%  \item Failure of MFCQ and remedy for KKTR. At this spot, discuss the S and possibly M-Stationary conditions.
%  \item Discuss the relationship of LLVFR to GSIP as paralell to the one enjoyed by KKTR to MPEC. But discuss the fact that unfortunately, CQ cannot hold for the corresponding GSIP model of bilevel.
%  \item Also discuss the fact the recent QC for calmness fails for \eqref{eq:LLVFR}; see the condition introduced in Gferrer and Outrata's paper.
%  \item After these two points have been discussed, then introduce the partial exact penalization.
%\end{itemize}
%In this subsection, we discuss conditions that can lead to necessary optimality conditions for problems \eqref{eq:KKTR} and \eqref{eq:LLVFR}.
First considering \eqref{eq:KKTR}, it is well-known that the standard MFCQ \eqref{MFCQ general}--\eqref{MFCQ general II} fails at any feasible point; see \cite{FlegelKanzowOutrata2007, YeZhuZhuExactPenalization1997}. However, reformulating the feasible set of the problem can lead to a tractable MFCQ-type constraint qualification (CQ) known as MPEC-MFCQ that can help generate optimality conditions; cf. \cite{DempeZemkohoOnTheKKTRef, FlegelKanzowOutrata2007, YeNecessary2005}. Many other specifically tailored CQs (e.g., MPEC-LICQ, MPEC-Abadie CQ, and MPEC-Guignard CQ) have been proposed and analysed in the literature; see, e.g., \cite{FlegelKanzowOutrata2007,YeNecessary2005} and references therein.
In this paper, we are not going to follow any of these standard approaches to derive necessary optimality conditions for \eqref{eq:KKTR}. Instead, we will use the penalization approach introduced in Subsection \ref{Framework for optimality conditions and numerical comparison}, as it is conducive to a sensible comparison of \eqref{eq:KKTR} and \eqref{eq:LLVFR}. To see why, we focus our attention next on tools to derive necessary optimality conditions for \eqref{eq:LLVFR}.

Considering \eqref{eq:LLVFR} as a Lipschitz optimization problem (cf. Theorem \ref{Lipschitz continuity of phi}), it is also well-known that the corresponding extension of the MFCQ systematically fails \cite{DempeZemkohoGenMFCQ,YeZhuOptCondForBilevel1995}. Similarly to \eqref{eq:KKTR}, \eqref{eq:LLVFR} is closely related to another important class of optimization problem; namely, the generalized semi-infinite programming problem (GSIP). In fact, problem \eqref{eq:LLVFR} is equivalent to the following special class of GSIP:
\begin{equation}\label{GSIP}
\begin{array}{lll}
\hspace{2cm} &  \underset{x,\,y}\min & F(x,y) \\
             &  \mbox{ s.t. }        & G(x,y)\leq 0,  \ \   g(x,y)\leq 0,\\
             &                       & f(x,y)-f(x, z)\leq 0,\;\; \forall z:\; g(x,z)\leq 0.
\end{array}
\end{equation}
The version of the MFCQ tailored to this class of problem is called the \emph{extended Mangasarian-Fromowitz constraint qualification} (EMFCQ) \cite{JongenRueckmannStein1998} and will be said to hold at a feasible point $(\bar x, \bar y)$ of problem \eqref{GSIP} if there exists a vector $d:=(d^1, d^2)\in \mathbb{R}^n\times \mathbb{R}^m$ such that
\begin{equation}\label{GMFCQ-GSIP}
\begin{array}{l}
 \nabla G_i(\bar x, \bar y)^\top d < 0, \;\; i\in I^1,\\
 \nabla g_j(\bar x, \bar y)^\top d < 0, \;\; j\in I^2,\\
 \nabla\ell^o(\bar x, \bar y, \bar z, v)^\top d < 0, \;\; \forall \bar z\in S(\bar x),\;\; \forall v\in \Lambda^o(\bar x, \bar y, \bar z),
\end{array}
\end{equation}
where $\ell^o(x, y, z, v):= v_o\left[f(x,y) - f(x, z)\right] - \sum_{i\in I^4} v_i g_i(x, z)$ and $\nabla\ell^o$ represents the gradient of the function w.r.t. its first and second variables. Also note that
\[
\Lambda^o(\bar x, \bar y, \bar z):=\left\{v\left|\;v_o\geq 0, \;\; v_i\geq 0, \;\; i\in I^4, \;\; v_o+ \sum_{i\in I^4} v_i =1,\;\; \nabla_3 \ell^o(\bar x, \bar y, \bar z, v)=0\right.\right\}.
\]
Having $(\bar x, \bar y)$ as a feasible point of problem \eqref{eq:LLVFR} implies that we automatically have $\bar y\in S(\bar x)$ and hence,  $\Lambda^o(\bar x, \bar y, \bar y)\neq \emptyset$ by the Fritz-John rule for the lower-level problem \eqref{lower-level problem}. Considering any vector $(u,v)\in \Lambda^o(\bar x, \bar y, \bar y)$, it follows that for $z:=\bar y$, we have
\begin{eqnarray}
 \sum_{i\in I^2} v_i \nabla_1 g_i(\bar x, \bar y)^\top d^1 +   \sum_{i\in I^2} v_i \nabla_2 g_i(\bar x, \bar y)^\top d^2 \leq 0,\label{eq:Moli1}\\
  v_o \nabla_2 f(\bar x, \bar y)^\top d^2 - \sum_{i\in I^2} v_i \nabla_1 g_i(\bar x, \bar y)^\top d^1 < 0, \label{eq:Moli2}\\
v_o \nabla_2 f(\bar x, \bar y)^\top d^2 + \sum_{i\in I^2} v_i \nabla_2 g_i(\bar x, \bar y)^\top d^2 = 0, \label{eq:Moli3}
\end{eqnarray}
where \eqref{eq:Moli1} and \eqref{eq:Moli2} respectively follow from the second and third lines of \eqref{GMFCQ-GSIP}, while \eqref{eq:Moli3} results from the definition of $\Lambda^o(\bar x, \bar y, \bar y)$. Considering \eqref{eq:Moli1} and \eqref{eq:Moli2},
$$
v_o \nabla_2 f(\bar x, \bar y)^\top d^2 < \sum_{i\in I^2} v_i \nabla_1 g_i(\bar x, \bar y)^\top d^1 \leq -  \sum_{i\in I^2} v_i \nabla_2 g_i(\bar x, \bar y)^\top d^2.
$$
This obviously contradicts \eqref{eq:Moli3} and thus confirming that the EMFCQ systematically fails at any feasible point of problem \eqref{eq:LLVFR}. Therefore, unlike for \eqref{eq:KKTR}, there does not seem to be any hope to restore a MFCQ-type CQ for \eqref{eq:LLVFR} via a transformation of its feasible set.

So far, the main qualification condition that has been successfully applied to derive necessary optimality conditions for \eqref{eq:LLVFR} is the partial calmness on $(x,y) \rightarrow f(x,y) - \varphi(x)$ (cf. Definition \ref{partial-calmness}). However, this condition is very restrictive as demonstrated in \cite{MehlitzMichenkoZemkoho}. There are various characterizations of the condition, and they are overviewed in the latter reference.
An interesting perspective of the partial calmness condition for \eqref{eq:LLVFR} is that a qualification condition ensuring that it holds is also a sufficient condition for the partial calmness of problem \eqref{eq:KKTR} on $(x, y, z)  \mapsto z^\top g(x,y)$ to hold \cite{YeZhuZhuExactPenalization1997}. This connection between \eqref{eq:KKTR} and \eqref{eq:LLVFR} is one of the main motivations of the comparison approach adopted in this paper in addition to the fact this framework enables standard-type CQs to be subsequently applied on both problems.

%It is well known that standard constraint qualifications do not hold for \eqref{eq:KKTR} because of the complementary equations  $g(x,y)^\top z=0$ \cite{YeZhuZhuExactPenalization1997} and also do not hold for \eqref{eq:LLVFR} due to the equations $f(x,y)=\varphi(x)$ \cite{YeZhuOptCondForBilevel1995}.
Recall that the concept of partial calmness was introduced in \cite{YeZhuOptCondForBilevel1995} in the context of problem \eqref{eq:LLVFR} and in line with Theorem \ref{equivalen}, it is said to hold on $(x,y) \mapsto f(x,y) - \varphi(x)$ at one of its local optimal solution $(\bar x, \bar y)$ if  this point is also locally optimal for problem
\begin{equation}\label{Penalized-LLVF}
    \underset{x,y}\min~F(x,y) + \lambda \left(f(x,y) -\varphi(x)\right) \;\mbox{ s.t. }\; G(x,y)\leq 0, \; g(x,y)\leq 0,
\end{equation}
for some $\lambda >0$, given that $\gamma(x,y) := f(x,y) - \varphi(x) \geq 0$ for all $(x, y)$ such that $x\in X$ and $g(x,y)\leq 0$.
% Various conditions ensuring that partial calmness is satisfied for \eqref{eq:LLVFR} can be found in \cite{YeZhuOptCondForBilevel1995,YeZhuZhuExactPenalization1997,DempeZemkohoGenMFCQ}.
%There is also a very interesting connection between the partial concept for \eqref{eq:LLVFR} and that of \eqref{eq:KKTR}. Essentially, according to \cite{YeZhuZhuExactPenalization1997}, any condition ensuring that problem \eqref{eq:LLVFR} is partially calm implies that problem \eqref{eq:KKTR} is also partially calm.
Similarly, \eqref{eq:KKTR} will be said to be partially calm on $(x, y, z)  \mapsto z^\top g(x,y)$ at a locally optimal point $(\bar x, \bar y, \bar z)$ if there exists a number $\lambda>0$  such that this point also locally solves
\begin{equation}\label{re-eq-P}
  \underset{x,\,y,\,z}\min~F(x,y) -  \lambda  z^\top g(x,y) \; \mbox{ s.t. } \;  G(x,y)\leq 0, \; g(x,y)\leq 0, \; \nabla_2\ell(x,y,z)=0, \; z \geq 0,
\end{equation}
as for any $(x, y, z)$ such that $z \geq 0$ and $g(x,y)\leq 0$, we have $\gamma(x,y,z):=g(x,y)^\top z\leq 0$. %For conditions ensuring that partial calmness holds, see Appendix \ref{Summary of relationships between the problems}.

One thing that is clear by now is that the partial penalization above is not enough to completely develop necessary optimality conditions for our problems \eqref{eq:KKTR} and \eqref{eq:LLVFR}. Hence, we introduce the versions of the MFCQ \eqref{MFCQ general}--\eqref{MFCQ general II} tailored to these problems. The KKT-MFCQ will be said to hold at a feasible point $(x,y,z)$ of problem \eqref{re-eq-P} if the gradients
 \begin{eqnarray}\label{ind-cond}
\nabla(\nabla_{2_i}\ell)(x,y,z),\;\; i=1, \ldots, m \;\; \mbox{ are linearly independent}
\end{eqnarray}
($\nabla_{2_i}\ell$ representing the $i$th component of the derivative of $\ell$ w.r.t. the second variable $y$) and there exist vectors $d^1\in \mathbb{R}^{n+m}$ and $d^2\in \mathbb{R}^q$  such that the following conditions hold:
\begin{eqnarray}\label{kkt-mfcq}
\begin{array}{rl}
\nabla(\nabla_{2_i}\ell)(x,y,z)^\top d^{12}=0, & i=1, \ldots, m,\\
\nabla  G_j(x,y)^\top d^1 <0, & j\in I^1,\\
\nabla g_k(x,y)^\top d^1 <0, & k\in I^2,\\
 d^2_l<0, & l\in I^3.\\
\end{array}
\end{eqnarray}
Looking more closely at the terms in KKT-MFCQ involving the function $\ell$, note that the linear independence condition \eqref{ind-cond} is equivalent to  the full rank condition for the matrix
 \begin{eqnarray}\label{ind-cond-1}
  \nabla(\nabla_{2}\ell)(x,y,z)^\top &=&
 \left[
 \begin{array}{l}
 (\nabla^2_{12}f(x,y))^\top + \sum_{k=1}^q z_k(\nabla^2_{12}g_k(x,y))^\top\\
 \nabla^2_{22}f(x,y) + \sum_{k=1}^q z_k\nabla^2_{22}g_k(x,y)\\
 \nabla_{2}g(x,y)
 \end{array}
 \right].
\end{eqnarray}
Furthermore, based on \eqref{ind-cond-1}, the first condition in \eqref{kkt-mfcq} is equivalent to $\nabla(\nabla_{2}\ell)(x,y,z)d^{12}=0$. Also note that if $f$ and $g$ are all linear in $(x,y)$, then $\nabla(\nabla_{2}\ell)(x,y,z)^\top$ satisfies the full column rank condition if the Jacobian matrix $\nabla_{2}g(x,y)$ has a full column rank. In the latter case, the condition $\nabla(\nabla_{2}\ell)(x,y,z)d^{12}=0$ in \eqref{kkt-mfcq} can be replaced by $\nabla_2 g (x,y)^\top d^2=0$.

 %{\color{blue}{
%  \begin{enumerate}
%    \item I think there were some errors/typos in the initial version of this CQ that you wrote?
%    \item  $h$ still needs to be written in detail, in terms of $f(x,y)$ and $g(x,y)$; it also seems that the linear independence part can be written just in terms of the components of $g$ or hold under such a condition? The further we can go to write the KKT-MFCQ just in terms of $G$ and $g$, the better.
%    \item  An example or class of problem where this condition can automatically hold needs to be found.
%    \item More importantly, it seems like there is a strong connection between this KKT-MFCQ and the following LLVF-MFCQ. It seems like latter would imply the former. We need to check this and possibly find an example to illustrate the LLVF-MFCQ does not imply KKT-MFCQ.
%  \end{enumerate}
%}}

Similarly to problem  \eqref{eq:LLVFR},  we consider the LLVF-MFCQ, which will be said to hold at a point $(x,y)$ if there exists a vector  $d\in \mathbb{R}^{n+m}$ such that we have
\begin{equation}\label{MFCQ-UL}
\begin{array}{rl}
\nabla G_j(x,y)^\top d <0, &j\in I^1,\\
\nabla g_k(x,y)^\top d <0, &k\in I^2.
\end{array}
\end{equation}
Clearly, if KKT-MFCQ holds at some point $(x, y, z)$, then LLVF-MFCQ holds at $(x, y)$, which means that the former condition is stronger than the latter one. In the next example, we show that the converse of this implication is not true.
 \begin{exmp}[the KKT-MFCQ can fail while the LLVF-MFCQ is satisfied]\label{ex-1} Consider the example of problem \eqref{eq:P} taken from \cite{LamparielloSagratella2017} with the following data:
$$
F(x,y) := x^2 + (y_1 + y_2)^2, \; G(x,y) := -x + 0.5, \; f(x,y) :=  y_1, \; g(x,y) :=  -\left(x+y_1+y_2 -1, \; y_1, \; y_2\right)^\top.
$$
The optimal solution of the problem is $(\bar x, \bar y)$ with $\bar x=0.5$ and $\bar y=(0,\;0.5)^\top$. It is easy to verify that $\bar z=(0,\;1,\;0)^\top$ is the only point  in $\Lambda(\bar x, \bar y)$. Hence, $(\bar x, \, \bar y, \, \bar z)$ is feasible to  the problem \eqref{re-eq-P} and  $I^1=\{1\}$, $I^2=\{1, \,2\}$, and $I^3=\{1, \, 3\}$. The first and last conditions are equivalent to
\[
d_1^2+d_2^2=0, \;\; d_1^2+d_3^2=0, \;\; d^2_1 < 0, \; \mbox{ and } \; d^2_3 < 0.
\]
This system is obviously infeasible. Thus, demonstrating the KKT-MFCQ does not hold  at $(\bar x, \bar y, \bar z)$. On the other hand, it easy to find a vector $(d_1, d_2, d_3)$ such that we have
$d_1 > 0$, $d_2 > 0$,  and $d_1 + d_2 + d_3 > 0$; confirming that the LLVF-MFCQ holds at the point $(\bar x,\, \bar y)$.
%In fact,
%$ \nabla(\nabla_{2_i}\ell)(\bar x, \bar y, \bar z)^\top d=0,$ $i=1, \, 2 $, where $d=(d_1^1, \, d_2^1, \, d_3^1, \, d_1^2, \, d_2^2, \, d_2^3)^\top$, indicates
%$d_1^2+d_2^2=d_1^2+d_3^2=0$. On the other hand  $d^2_l<0$, $l\in I^3=\{1, \, 3\}$  means $d_1^2<0$ and $d_3^2<0$, which contradicts to $d_1^2+d_3^2=0$. Therefore, KKT-MFCQ does not hold  at $(\bar x, \, \bar y, \, \bar z)$. However, condition \eqref{MFCQ-UL} with $d=(d_1, \, d_2, \, d_3)^\top$ indicates $-d_1<0$, $-d_1-d_2-d_3<0$ and $-d_2<0$, which can be satisfied easily.
\end{exmp}

{ {To conclude this subsection, we would like to emphasize that problems \eqref{eq:KKTR} and \eqref{eq:LLVFR}, taken from the point of view of standard constrained optimization, they both violate most well-known CQs. However, dealing with the former problem as an MPCC, various reformulations of the feasible set can enable the generic satisfaction of many dual CQs \cite{DempeZemkohoOnTheKKTRef, FlegelKanzowOutrata2007, YeNecessary2005}. This is unfortunately not the case for the latter problem, as even a weaker CQ like the \emph{calmness} of the set-valued mapping
\[
\Psi(\delta):=\left\{(x, y)\in \mathbb{R}^n\times \mathbb{R}^m\left|\, g(x,y)\leq 0, \;\, f(x,y)-\varphi(x)\leq \delta \right.\right\},
\]
which can ensure the fulfilment of the partial calmness condition for \eqref{eq:LLVFR}, is shown in \cite{HenrionSurowiec2011} to systematically fail for important problem classes, while the corresponding CQ for a version of \eqref{eq:KKTR} automatically holds. For the definition of the calmness of $\Psi$ and its ramifications in bilevel optimization, see, e.g., \cite{HenrionSurowiec2011, MehlitzMichenkoZemkoho} and references therein.
}}

\subsection{Optimality conditions}\label{Optimality conditions}
%Things to sort here are:
%\begin{enumerate}
%  \item Notations;
%  \item Dimensions of problem $p$, $q$, etc.
%  \item Harmonize the index sets, $\theta$, etc., with suitable indices ($\theta^1$, etc.)
% \item Harmonize the notation for the lower-level Lagrangian function/mulipliers structure
% Reflect on the notation of $\lambda$ (with indices?) and all the Lagrange multipliers
%\end{enumerate}
Our main aim here is to provide necessary optimality conditions for the KKT and LLVF reformulations \eqref{eq:KKTR} and \eqref{eq:LLVFR}, respectively, using the corresponding partial exact penalization approaches introduced in the previous subsection. %We start with the KKT reformulation, where $\nabla h(x,y,u)$ represents the gradient of $h$ w.r.t. $(x,y)$.

\begin{thm}[necessary optimality conditions for \eqref{eq:KKTR}]\label{Necessary conditions re-eq-P} Consider problem \eqref{eq:KKTR} while assuming that $F$ and $G$ (resp. $f$ and $g$) are $\mathcal{C}^1$ (resp. $\mathcal{C}^2$). Suppose that $(x, y, z)$ is a local optimal solution of the problem and let it be partially calm on $(x, y, z) \mapsto z^\top g(x,y)$ at a point $(x, y, z)$, where the KKT-MFCQ is also assumed to hold. Then, there exist $\lambda > 0$, $u\in \mathbb{R}^p$, $(v, w)\in \mathbb{R}^{2q}$, and $s\in\mathbb{R}^{m}$ such that
 \begin{eqnarray}\label{optimality conditions re-eq-P}
 \nabla F(x,y)+ \nabla G(x,y)^{\top}u + \nabla g(x,y)^{\top}(v - \lambda z) + \nabla_{1, 2}(\nabla_2 \ell)(x,y,z)^\top s =0,\label{L-x}\\
%  \nabla_y F(x,y)+ \nabla_y G(x,y)^{\top}\alpha + \nabla_y g(x,y)^{\top}(\beta - \lambda u) + \nabla_y \mathcal{L}(x,y,u)^\top \gamma =0,\label{L-y}\\
\nabla_2 f(x,y) + \nabla_2 g(x,y)^\top z=0, \label{L-eq}\\
- \lambda  g(x,y) + \nabla_2 g(x,y) s + w =0,\label{L-z}\\
u\geq 0, \; G(x,y)\leq 0,\; u^{\top}G(x,y) = 0,\label{CP-u}\\
v\geq 0, \; g(x,y)\leq 0,\; v^{\top} g(x,y) = 0, \label{CP-v}\\
w \leq 0, \; z\geq 0,\; w^{\top} z = 0. \label{CP-w}
\end{eqnarray}
\end{thm}
\begin{proof}
Considering $(x, y, z)$ as a local optimal solution of problem \eqref{eq:KKTR}, it follows from the partial calmness assumption that this point is also a local optimal solution for problem \eqref{re-eq-P} for some $\lambda >0$. Applying the standard Lagrange multiplier rule to the latter problem, we have the result under the fulfilment of the KKT-MFCQ at the point $(x, y, z)$.
\end{proof}
Next, we state a relationship between the optimality conditions obtained in this result and the S-stationarity conditions of problem \eqref{eq:KKTR}, known to be the strongest in the context of MPCCs, class of problem that  \eqref{eq:KKTR} belongs to.
\begin{thm}[relationship between the optimality conditions in Theorem \ref{Necessary conditions re-eq-P} and S-stationarity]\label{proposition 1 section 4}
$(\bar x, \bar y, \bar z, \bar s, \bar u, \bar v)$  with $\bar z^\top g(\bar x, \bar y)=0$ satisfies  \eqref{L-x}--\eqref{CP-w} for some $\lambda >0$ if and only if there exist $\tilde{z}\in \mathbb{R}^p$, $\tilde{s}\in \mathbb{R}^m$, $\tilde{u}\in \mathbb{R}^k$, and $\tilde{v}\in \mathbb{R}^p$ such that the following conditions hold:
\begin{eqnarray}
\nabla F(\bar x, \bar y) + \nabla G(\bar x, \bar y)^\top \tilde{u} + \nabla g(\bar x, \bar y)^\top \tilde{v} +  \nabla_{1, 2}(\nabla_2 \ell)(\bar x, \bar y, \tilde{z})^\top \tilde{s} =0, \label{M-1}\\
\nabla_2 f(\bar x, \bar y) + \nabla_2 g(\bar x, \bar y)^\top \tilde{z}=0,\\
\tilde{u} \geq 0, \;\, G(\bar x, \bar y)\leq 0, \; G(\bar x, \bar y)^\top \tilde{u}=0,\label{KM-5}\\
\tilde{z} \geq 0, \; g(\bar x, \bar y)\leq 0, \; \tilde{z}^\top g(\bar x, \bar y)=0, \\
\forall i\in \nu^2: \;\, \nabla_2 g_i(\bar x, \bar y)\tilde{s} =0, \;\, \forall i\in \eta^2: \;\,  \tilde{v}_i =0,\;\, %\label{M-3}\\
\forall i\in \theta^2: \;\, \tilde{v}_i \geq 0 \wedge \sum^m_{l=1}\tilde{s}_l\nabla_{2_l} g_i(\bar x, \bar y) \geq 0. \label{S-0}
\end{eqnarray}
%where
%\begin{equation}\label{Jacobian of g times gamma}
%\nabla_2 g(x, y)\gamma^* := \left[\sum^m_{l=1}\gamma^*_l\nabla_{y_l} g_1(x, y), \ldots, \sum^m_{l=1}\gamma^*_l\nabla_{y_l} g_q(x, y)\right]^\top
%\end{equation}
%and $\nabla_y g_\nu(x, y)\gamma^*$ denotes the components of the vector in the right-hand-side for which $i\in \nu$.
\end{thm}
\begin{proof}
  See \cite[Section 3.3]{ZemkohoThesis}.
\end{proof}
The conditions in \eqref{M-1}--\eqref{S-0} correspond to the S-type stationary conditions for problem \eqref{eq:KKTR}, in the sense of MPCCs; see \cite{DempeZemkohoOnTheKKTRef}. As the algorithm to be designed in next section to solve this problem will be computing points of the form  \eqref{L-x}--\eqref{CP-w}, the key message from Theorem \ref{proposition 1 section 4} is that for such point a $(\bar x, \bar y, \bar z, \bar s, \bar u, \bar v)$ to be S-stationary, we just need to have $\bar z^\top g(\bar x, \bar y)=0$.

For the necessary optimality conditions for \eqref{eq:LLVFR}, we have the following result from \cite{FischerZemZhou2019}.% with
% $\ell(x,y', w):= f(x,y') + w^\top g(x,y')$. % { {To proceed, we also need the MFCQ for the lower-level problem \eqref{lower-level problem}. The  {lower-level Mangasarian-Fromowitz constraint qualification} (LMFCQ) will be satisfied at the point  $(x, y')$ if there exists a vector $d\in \mathbb{R}^m$ such that
%\begin{equation}\label{MFCQ-follower}
% \nabla_y g_k (x, y')^\top d < 0, \;\; k\in I^4.
%\end{equation}}}
\begin{thm}[necessary optimality conditions for \eqref{eq:LLVFR}]\label{KN stationarity partial calm} Consider problem \eqref{eq:LLVFR} while assuming that $F$, $G$, $f$, and $g$ are $\mathcal{C}^1$. Let $(x,y)$ be a local optimal solution of problem \eqref{eq:LLVFR}, where the functions  $f$ and $g_j$ with $j=1, \ldots, q$ are also assumed to be fully convex. Furthermore, suppose  that  \eqref{eq:LLVFR} is partially calm on $(x,y) \mapsto f(x,y)-\varphi(x)$  at $(x,y)$, where the LMFCQ and LLVF-MFCQ are also assumed to hold.
%\begin{itemize}
%\item[$(i)$] the problem \eqref{eq:LLVFR} is partially calm on $f(x,y)-\varphi(x)$  at $(x,y)$,
%\item[$(ii)$] UMFCQ holds at $(x,y)$ and
%\item[$(iii)$]  LMFCQ  holds at $(x, y')$ for all $y'\in S(x)$.
%\end{itemize}
Then there exist $\lambda >0$, $u\in \mathbb{R}^p$, $(v, w)\in \mathbb{R}^{2q}$, and $z \in \mathbb{R}^m$ such that we have
 \begin{eqnarray}\label{optimality conditions 0}
 \nabla F(x,y) + \nabla G(x,y)^{\top}u +\nabla g(x,y)^{\top}v+\lambda \nabla f(x,y) - \lambda \left[\begin{array}{c}
                                                                                                               \nabla_1 \ell(x, z, w)\\
                                                                                                               0
                                                                                                             \end{array}
 \right] =0,\label{KS-1}\\
%\nabla_y F(x,y) + \nabla_y G(x,y)^{\top}u +\nabla_yg(x,y)^{\top}v +\lambda_2\nabla_y f(x,y)\qquad \qquad \qquad\;\;  =0, \label{VS-2}\\
\nabla_2f(x, z)+\nabla_2 g(x, z)^{\top}w=0,\label{KS-2}\\
u\geq 0, \; G(x, y)\leq 0,\; u^{\top}G(x,y)= 0,\label{VS-3}\\
v\geq 0, \; g(x, y)\leq 0,\; v^{\top} g(x,y)= 0, \label{VS-4}\\
w\geq 0, \; g(x, z)\leq 0,\; w^{\top} g(x, z)= 0. \label{KS-3}
\end{eqnarray}
\end{thm}

\begin{rem}
{ {There are important classes of functions that satisfy the full convexity assumption imposed on the lower-level problem in Theorem \ref{KN stationarity partial calm}; cf. \cite{LamparielloSagratella2017Numerically}. However, when it is not possible to guaranty that this assumption is satisfied, it can be replaced by the inner semicontinuity of the lower-level optimal solution set-valued mapping $S$, thanks to Theorem \ref{Lipschitz continuity of phi}(ii).
%there are at least two alternative scenarios to obtain the same optimality conditions. The first is to replace the full convexity assumption by the \emph{inner semicontinuity} of the set-valued map $S$ \eqref{S Map}. Secondly, note that a much weaker qualification condition known as \emph{inner semicompactness} can also be used here. However, under the latter assumption, it will additionally be required to have $S(\bar x)=\{\bar y\}$ in order to arrive at the optimality conditions \eqref{KS-1}--\eqref{KS-3}. The concept of inner semicontinuity (resp. semicompactness) of $S$ is closely related to the lower semicontinuity (resp. upper semicontinuity) of set-valued mappings; for more details on these notions and their ramifications on bilevel programs, see \cite{DempeDuttaMordukhovichNewNece, DempeZemkohoGenMFCQ}. %lead to complicated optimality conditions, as the estimate of the subdifferential of $\varphi$ would then involve a convex operator; Hence, this case is out of the scope of this paper.
}}
\end{rem}

%Next, we show that there can be close relationship between the optimality conditions of problems
%\newpage
\begin{thm}[relationship between the optimality conditions of problems \eqref{eq:KKTR} and \eqref{eq:LLVFR}]\label{relationships between the stationarity concepts} The following statements hold true:
\begin{itemize}
  \item[(i)] %If a point $(x, y, z, s, u, v, w)$ with $z^\top g(x,y)=0$ and $\nabla_{1, 2}(\nabla_2 \ell)(x, y, z) =0$ satisfies
  Assume that the conditions \eqref{L-x}--\eqref{CP-w} hold, with $z^\top g(x,y)=0$, $\nabla_{1, 2}(\nabla_2 \ell)(x, y, z) =0$, $z:=w$, and $w:=t$, for some $\lambda >0$. Then, conditions \eqref{KS-1}--\eqref{KS-3} are satisfied with $z:=y$.
  \item[(ii)] Suppose that the conditions \eqref{KS-1}--\eqref{KS-3} hold with $z=y$, $w:=z$, $\nabla_{1, 2}(\nabla_2 \ell)(x, y, z) =0$, and there exists a vector $s\in \mathbb{R}^m$ such that
  \begin{equation}\label{foman}
\nabla_2 g(x, y)s - \lambda g(x, y) \geq 0 \;\mbox{ and }\; z^\top \nabla_2 g(x, y)s=0
  \end{equation}
are satisfied for some $\lambda$. Then, the conditions \eqref{L-x}--\eqref{CP-w} also hold.
\end{itemize}
\end{thm}
%\begin{proof}
%  See \cite[Theorem 3.1.9]{ZemkohoThesis}.
%\end{proof}
\begin{proof}For $(i)$, if we consider $(x, y, z, u, v, w)$ with $\nabla_{1, 2}(\nabla_2 \ell)(x, y, z) =0$  such that  \eqref{L-x}--\eqref{CP-w} hold for some $\lambda >0$, we precisely have from \eqref{L-x} that
\begin{eqnarray}
\nabla_1 F(x, y) + \nabla_1 G(x, y)^\top u + \nabla_1 g(x, y)^\top v - \lambda \nabla_1 g(x, y)^\top z = 0, \label{Bam1}\\
\nabla_2 F(x, y) + \nabla_2 G(x, y)^\top u + \nabla_2 g(x, y)^\top v - \lambda \nabla_2 g(x, y)^\top z = 0. \label{Bam2}
\end{eqnarray}
Obviously, equation \eqref{Bam1} is equivalent to
\[
\nabla_1 F(x, y) + \nabla_1 G(x, y)^\top u + \nabla_1 g(x, y)^\top v +\lambda \nabla_1 f(x,y) - \lambda \nabla_1 \ell(x, y, z) = 0.
\]
This confirms that the $x$-component of \eqref{KS-1} holds. Furthermore, we also have from equation \eqref{L-eq} that $\nabla_2 g(x, y)^\top z = - \nabla_2 f(x,y)$. Inserting this expression in \eqref{Bam2}, it follows that
$$
\nabla_2 F(x, y) + \nabla_2 G(x, y)^\top u + \nabla_2 g(x, y)^\top v + \lambda \nabla_2 f(x, y) =0.
$$
Hence, the $y$-component of \eqref{KS-1} is also satisfied. Subsequently, the whole system \eqref{KS-1}--\eqref{KS-3} is satisfied with $z:=y$, considering the assumption that $z^\top g(x,y)=0$ (i.e., $w^\top g(x,y)=0$).

For $(ii)$, consider  $(x, y, z, u, v, w)$ satisfying \eqref{KS-1}--\eqref{KS-3} with the related assumptions, then
\begin{equation}\label{Naws1}
\nabla_1 F(x,y)+ \nabla_1 G(x,y)^{\top}u + \nabla_1 g(x,y)^\top(v - \lambda z) + \nabla^2_{12}\ell(x, y, z)^\top s =0.
\end{equation}
Secondly, considering equality  $\nabla_2 f(x,y) = - \nabla_2 g(x, y)^\top w$ from \eqref{KS-2} (with  $z=y$ and $w:=z$) and inserting it in the $y$-component of \eqref{KS-1}, it holds that
\begin{equation}\label{Naws2}
\nabla_2 F(x,y)+ \nabla_2 G(x,y)^{\top}u + \nabla_2 g(x,y)^{\top}(v - \lambda w) + \nabla^2_{22}\ell(x,y,z)^\top s =0.
\end{equation}
It is clear that combining \eqref{Naws1} and \eqref{Naws2}, we have the fulfilment of equation \eqref{L-x}.  If additionally,  \eqref{foman} holds, then it follows that the whole system \eqref{L-x}--\eqref{CP-w} is satisfied.% at $(x, y, \alpha, \beta, \gamma)$ for some $\lambda >0$.
\end{proof}
%\newpage
%% Table 2
The assumption $\nabla_{1, 2}(\nabla_2\ell)(x, y, z) =0$ automatically holds if the functions $f$ and $g$ defining the lower-level problem \eqref{lower-level problem} take the form $f(x,y):= a(x)+b^\top y$ and $g(x,y):= C(x)+D^\top y$, respectively. Here, $a : \mathbb{R}^n \rightarrow \mathbb{R}$ and $C : \mathbb{R}^n \rightarrow \mathbb{R}^q$ while $b\in \mathbb{R}^m$ and $D\in \mathbb{R}^{q\times m}$.
In general, the stationarity conditions for problem \eqref{eq:P} obtained via \eqref{eq:LLVFR} differ significantly from those derived through \eqref{eq:KKTR}, especially due to the second order term appearing in the latter case. This theorem establishes a clear link between both classes of conditions, though under a very restrictive framework.  { {However, other setups different from the ones considered in Theorem \ref{relationships between the stationarity concepts} could lead to the same results. For example, reconsidering the bilevel program in Example \ref{Example with pics}, the optimality conditions \eqref{L-x}--\eqref{CP-w}, with $z:=w$ and $w:=t$, are verified by the point defined by
\[
(\bar x,\, \bar y,\, \bar v,\, \bar w,\, s,\, t)= \left(1, \, -1, \,1, \, \frac{1}{2}, \, 0, \, 0\right) \;\; \mbox{ for } \;\; \lambda = 4.
\]
Subsequently, we can also easily check that
\[
(\bar x,\, \bar y,\, \bar v,\, \bar w)= \left(1, \, -1, \, 1, \, \frac{1}{2}\right)
\]
satisfies \eqref{KS-1}--\eqref{KS-3}, with $z:=y$, for $\lambda = 4$. Note that $\nabla_{1, 2}(\nabla_2 \ell)(\bar x, \bar y, \bar w) =[2, \, 1]^\top \neq 0$. Hence, in this case, Theorem \ref{relationships between the stationarity concepts}(i) holds despite the failure of the imposed condition $\nabla_{1, 2}(\nabla_2 \ell)(\bar x, \bar y, \bar w) = 0$.
%\begin{exmp}
%Reconsider the bilevel program in Example \ref{Example with pics}. Then, for the optimality conditions \eqref{L-x}--\eqref{CP-w}, with $z:=w$, are verified by the point defined by
%\[
%(\bar x, \bar y, \bar v, \bar w, \, s, w)= \left(1, \, -1, \, 0.5, \, 0,~ 0\right) \;\; \mbox{ for } \;\; \lambda = 4.
%\]
%\end{exmp}
}}

The main observations and relationships between \eqref{eq:KKTR} and \eqref{eq:LLVFR} from this section are summarized in Table \ref{Table 2}.

\begin{table}[!th]
\begin{center}
\captionsetup{width=1.5\linewidth}
\begin{tabular}{l|l|l|c|l}
\hline
\multicolumn{2}{c|}{${}$}   & \textbf{KKTR}&  &   \textbf{LLVFR}      \\
\hline
\multirow{4}{*}{\textbf{Basic requirements}} & $F$, $G$, $f$, and $g$ are $C^{1}$                    & \cmark       &    &         \cmark     \\
                              & $f$ and $g$ are $C^{2}$                            & \cmark       &    &         \xmark     \\
                              & LMFCQ              & \cmark       &    &    \xmark     \\
                              & Convexity/\emph{isc} of $S$             & \xmark       &    &    \cmark     \\
\hline
\multirow{4}{*}{\textbf{Qualification conditions}}        & MFCQ can hold                   &  \xmark    &   &         \xmark      \\
                                                 & Remedy for MFCQ failure exists  &  \cmark    &   &         \xmark     \\
                                                 & Partial calmness can hold       &  KKT-PCAL    & $\Longleftarrow$  &     LLVF-PCAL   \\
                                                 &  MFCQ for penalized problem     &  KKT-MFCQ    & $\Longrightarrow$  &        LLVF-MFCQ     \\
\hline
\multirow{2}{*}{\textbf{Stationarity conditions}} &                     & \multirow{2}{*}{\eqref{L-x}--\eqref{CP-w}}     & $\overset{(\ast)}{\Longrightarrow}$  &  \multirow{2}{*}{\eqref{KS-1}--\eqref{KS-3}}           \\[1ex]
                                         &                     &      & $\overset{(\ast\ast)}{\Longleftarrow}$  &             \\
\hline
\end{tabular}
\end{center}
\caption{Requirements for necessary optimality conditions for problems \eqref{eq:KKTR} and \eqref{eq:LLVFR} and relationships between them. Here, \emph{isc} stands for inner semicontinuity and (*) and (**) refer to the assumptions in Theorem \ref{relationships between the stationarity concepts}(i) and (ii), respectively. KKT-PACAL and LLVF-PACAL represent the partial calmness condition for  \eqref{eq:KKTR} and \eqref{eq:LLVFR}, respectively.} \label{Table 2}%\\[-1.5ex]
\end{table}
\section{Semismooth Newton-type method}\label{Newton method for the auxiliary equation}
In this section, we implement and compare the semismooth Newton scheme discussed in Subsection \eqref{Framework for optimality conditions and numerical comparison} on the necessary optimality conditions for \eqref{eq:KKTR} and \eqref{eq:LLVFR} presented in the previous section. Precisely, the optimality conditions of interest will be \eqref{L-x}-\eqref{CP-w} and \eqref{KS-1}-\eqref{KS-3}, respectively. To completely formulate these conditions as systems of equations, we use the \emph{Fischer-Burmeister function} \cite{FischerASpecial1992}  defined from $\mathbb{R}^2$ to $\mathbb{R}$ by
\begin{equation}\label{FB}
\fb(a,b):=\sqrt{a^2 + b^2} - a-b.
\end{equation}
For instance, we have $\left[u\geq 0, \; G(x,y)\leq 0,\; u^{\top}G(x,y)= 0\right] \Longleftrightarrow \psifb(-G(x,y), u)=0$ with
\begin{eqnarray}\label{FischerBurg}
 \psifb(-G(x,y), u) := \left[\begin{array}{c}
 \fb(-G_1(x,y), u_1) \\
\vdots \\
 \fb(-G_p(x,y), u_p)
 \end{array}
 \right].
\end{eqnarray}
To reformulate the optimality conditions resulting from the KKT reformulation as a system of equations, we denote the Lagrangian function of the corresponding problem by
%Now take a close look at Algorithm \ref{algorithm 1}, the main issue confronted us is the choice of $W^k\in \partial_B \Phi^{\lambda}(\zeta ^k)$. To see this we need respectively consider two systems  \eqref{L-x}-\eqref{CP-w} and \eqref{KS-1}-\eqref{KS-3}. For the former system, define corresponding Lagrangian-type function as
\begin{eqnarray}\label{KKT-Lagrangian}
\begin{array}{c}
\mathcal{L}^{\lambda_1}_{1}(\zeta^1):= F(x,y) + u^{\top} G(x,y) + (v + \lambda_1  z)^\top g(x,y) + s^\top  \nabla_2\bar{\ell}(x,y,z) + w^\top z,
\end{array}
\end{eqnarray}
where we set $\zeta^1 := (x,y,z,s,u,v,w)$ and $\bar{\ell}(x,y,z):=f(x, y) - z^\top g(x, y)$ for convenience in the presentation and comparison purpose in this section. Based on \eqref{FischerBurg} and \eqref{KKT-Lagrangian}, the counterpart of equation \eqref{Eq-Main} for the system \eqref{L-x}--\eqref{CP-w} can be obtained as
\begin{eqnarray}\label{KKT-Phi-lambda}
\Phi_{1}^{\lambda_1}(\zeta^1):=\left[
\begin{array}{l}
  \nabla \mathcal{L}^{\lambda_1}_{1}(\zeta^1)\\
  h(x,y,z) \\
 \psifb(-G(x,y), u)\\
  \psifb(-g(x,y), v)\\
  \psifb(-z, w)
\end{array}
\right]=0,
\end{eqnarray}
where $\lambda := \lambda_1$ and   $\nabla \mathcal{L}^{\lambda_1}_{1}$ representing the gradient of  $\mathcal{L}^{\lambda_1}_{1}$ w.r.t. $(x,y,z)$.  This is a square  equation system with $n+2m+p+3q$ variables and $n+2m+p+3q$ nonlinear equations.

Similarly, we consider the Lagrangian function of problem \eqref{eq:LLVFR}
\begin{eqnarray}\label{VF-Lagrangian}
\begin{array}{ll}
\mathcal{L}^{\lambda_2}_{2}(\zeta^2):=&   F(x,y)+u^{\top} G(x,y)+ v^{\top}g(x,y) +\lambda_2 f(x,y)-\lambda_2 \ell(x,z,w),
\end{array}
\end{eqnarray}
where $\zeta^2 := (x, y, z, u, v, w)$ and $\ell$ is defined in \eqref{def-h}. Then, the expression of \eqref{Eq-Main} in the context of the system \eqref{KS-1}-\eqref{KS-3} can be rewritten as
\begin{eqnarray}\label{VF-Phi-lambda}
 \Phi_{2}^{\lambda_2}(\zeta^2) := \left[
\begin{array}{l}
  \nabla \mathcal{L}^{\lambda_2}_{2}(\zeta^2)\\
 \psifb(-G(x,y), u)\\
  \psifb(-g(x,y), v)\\
  \psifb(-g(x,z), w)
\end{array}
\right]=0,
\end{eqnarray}
where $\lambda := \lambda_2$ and $\nabla \mathcal{L}^{\lambda_2}_{2}$ representing the gradient of  $\mathcal{L}^{\lambda_2}_{2}$ w.r.t. $(x, y, z)$.  This is also a square system of equations of dimension variables and $(n+2m+p+2q) \times (n+2m+p+2q)$. Clearly, the system of equations resulting from \eqref{eq:KKTR} is $q \times q$ larger than the one resulting from \eqref{eq:LLVFR}.

For the convergence of Algorithm \ref{algorithm 1} for equations \eqref{KKT-Phi-lambda} and \eqref{VF-Phi-lambda}, it follows from Theorem \ref{convergence result} that it suffices to develop conditions ensuring that the functions $\Phi_{1}^{\lambda_1}$ and $\Phi_{2}^{\lambda_2}$ are semismooth (and/or strongly semismoothness) and CD-regular. To proceed, note that a vector-valued function $\psi :\mathbb{R}^{\tilde{n}}\rightarrow \mathbb{R}^{\tilde{p}}$ is SC$^1$ (resp. LC$^2$) if its all components $\psi_i$, $i=1, \ldots, {\tilde{p}}$  are SC$^1$ (resp. LC$^2$).
\begin{thm}[semismoothness and strong semismoothness]\label{cond-check} The following statements hold true:
\begin{itemize}
\item[(i)] Suppose that $f$ and $g$ are $\mathcal{C}^1$. If $F$, $G$, $\nabla f$, and $\nabla g_i$, $i=1,\ldots, q$, are SC$^1$ (resp. LC$^2$), then $\Phi_{1}^{\lambda_1}$ is semismooth (resp. strongly semismooth).
\item[(ii)] If $F$, $G$, $f$, and $g$ are SC$^1$ (resp. LC$^2$), then $\Phi_{2}^{\lambda_2}$ is semismooth (resp. strongly semismooth).
\end{itemize}
\end{thm}
One can see that if a function $\psi$ is twice continuously differentiable and $\nabla^2 \psi$  is semismooth, then  $\nabla \psi$ is SC$^1$. In a similar way, if $\psi$ is  thrice continuously differentiable and $\nabla^3 \psi$  is locally Lipschitzian, then $\nabla \psi$ is LC$^2$. Clearly, the conditions in (i) are stronger than the ones in (ii).

%%\subsection{CD-Regularity}
%Recall that the index sets $\eta^i$, $\nu^i$ and $\theta^i$ with $i=1, 2, 3$ that we used here
% are defined in \eqref{multiplier sets}--\eqref{nu2nu3}.

\begin{thm}[estimate of the generalized Jacobian of $\Phi_{1}^{\lambda_1}$]\label{KKT Jacobian Phi}  Let $F$ and $G$ (resp. $f$ and $g$) be twice (resp. thrice) continuously differentiable  at $\bar\zeta:=(\bar x,\bar y, \bar z, \bar s,\bar u,\bar v,\bar w)$. If $\lambda_1>0$, then $\Phi_{1}^{\lambda_1}$ is semismooth at $\bar \zeta$ and any matrix $W^{\lambda_1}\in \partial \Phi_{1}^{\lambda_1}(\bar \zeta)$ can take the form
$ W^{\lambda_1} =\left[
\begin{array}{cccc }
A &C \\
 B &D
\end{array}
\right]$
with
\begin{eqnarray*}
A&:=&\left[
\begin{array}{cccc}
\nabla^2_{11} \mathcal{L}_{1}^{\lambda_1}(\bar\zeta) & \nabla^2_{12} \mathcal{L}_{1}^{\lambda_1}(\bar\zeta)^\top   & \nabla^2_{13} \mathcal{L}_{1}^{\lambda_1}(\bar\zeta)^\top & \nabla^2_{12} \bar{\ell}(\bar x,\bar y, \bar z)^{\top}  \\
\nabla^2_{12} \mathcal{L}_{1}^{\lambda_1}(\bar\zeta) &  \nabla^2_{22} \mathcal{L}_{1}^{\lambda_1}(\bar\zeta) & \nabla^2_{23} \mathcal{L}_{1}^{\lambda_1}(\bar\zeta)^\top &\nabla^2_{22}\bar{\ell}(\bar x,\bar y, \bar z)^{\top} \\
\nabla^2_{13} \mathcal{L}_{1}^{\lambda_1}(\bar\zeta)   &    \nabla^2_{23} \mathcal{L}_{1}^{\lambda_1}(\bar\zeta)  &O &  - \nabla_2 g(\bar x, \bar y)   \\
\nabla^2_{12}\bar{\ell}(\bar x, \bar y, \bar z) &   \nabla^2_{22}\bar{\ell}(\bar x,\bar y, \bar z)  & -\nabla_2 g(\bar x,\bar y)^{\top}& O
\end{array}
\right],\\
B&:=&\left[
\begin{array}{cccc}
  \Lambda_1 \nabla_{1} G(\bar x,\bar y) &    \Lambda_1 \nabla_{2} G(\bar x,\bar y)   & O & O \\
    \Lambda_2 \nabla_{1} g(\bar x,\bar y) &    \Lambda_2 \nabla_{2} g(\bar x,\bar y)   & O & O \\
       O &    O   &  \Lambda_3 & O
\end{array}
\right],\\
C&:=&\left[
\begin{array}{ccccccc}
 \nabla_1  G(\bar x,\bar y)^{\top} & \nabla_1  g(\bar x,\bar y)^{\top}&O\\
 \nabla_2  G(\bar x,\bar y)^{\top} & \nabla_2 g(\bar x,\bar y)^{\top}&O\\
 O & O & \mathcal{I} \\
 O & O & O
\end{array}
\right], \quad \mbox{ and } \quad D:=
\left[
\begin{array}{ccc}
 \Gamma_1& O & O\\
 O&  \Gamma_2& O\\
 O& O & \Gamma_3
\end{array}
\right],
\end{eqnarray*}
%\begin{equation*}
%\left[
%\begin{array}{ccccccc}
%\nabla^2_{xx} \mathcal{L}_{\rm KKT}^{\lambda_1} & (\nabla^2_{xx} \mathcal{L}_{\rm KKT}^{\lambda_1})^\top   & (\nabla^2_{xz} \mathcal{L}_{\rm KKT}^{\lambda_1})^\top &\nabla_x  h(\bar x,\bar y, \bar z)^{\top} & \nabla_x  G(\bar x,\bar y)^{\top} & \nabla_x  g(\bar x,\bar y)^{\top}&O\\
%\nabla^2_{xx} \mathcal{L}_{\rm KKT}^{\lambda_1} &  \nabla^2_{yy} \mathcal{L}_{\rm KKT}^{\lambda_1} & (\nabla^2_{yz} \mathcal{L}_{\rm KKT}^{\lambda_1})^\top&\nabla_y  h(\bar x,\bar y, \bar z)^{\top} & \nabla_y  G(\bar x,\bar y)^{\top} & \nabla_y g(\bar x,\bar y)^{\top}&O\\
% \nabla^2_{xz} \mathcal{L}_{\rm KKT}^{\lambda_1}   &    \nabla^2_{yz} \mathcal{L}_{\rm KKT}^{\lambda_1}  &O &  - \nabla_y g(\bar x, \bar y) & O & O & \mathcal{I} \\
% \nabla_x  h(\bar x,\bar y, \bar z) &   \nabla_y  h(\bar x,\bar y, \bar z)  & -\nabla_y g(\bar x,\bar y)^{\top} & O & O & O & O \\
%  \Lambda_1 \nabla_{x} G(\bar x,\bar y) &    \Lambda_1 \nabla_{y} G(\bar x,\bar y)   & O & O & \Gamma_1& O & O\\
%    \Lambda_2 \nabla_{x} g(\bar x,\bar y) &    \Lambda_2 \nabla_{y} g(\bar x,\bar y)   & O & O & O&  \Gamma_2& O\\
%       O &    O   &  \Lambda_4 & O & O& O & \Gamma_4
%\end{array}
%\right]
%\end{equation*}
where  $\Lambda_i :={\rm diag} (a^i)$ and $\Gamma_i :={\rm diag}(b^i)$, $i=1, 2, 3$, are such that
 \begin{equation}\label{KKT ab definition}
    (a^i_j,b^i_j)\left\{\begin{array}{lll}
                  =&(0,-1) & \mbox{ if } \;j\in  \eta^i, \\
                  =&(1,0) & \mbox{ if } \;j\in \nu^i, \\
                  \in& \{(\alpha, \beta): \; (\alpha-1)^2 + (\beta+1)^2\leq 1\} & \mbox{ if }\; j\in \theta^i,
                \end{array}
\right.
 \end{equation}
 with the index sets $\eta^i$, $\nu^i$, and $\theta^i$, $i=1, 2, 3$ defined in \eqref{multiplier sets}--\eqref{nu2nu3}.
\end{thm}

The next result provides a framework for the CD-regularity of the function $\Phi^{\lambda_1}$. To perform this, we define  the cone of feasible directions for problem \eqref{re-eq-P},
\begin{eqnarray}\label{cone of feasible}
Q_1(\bar x, \bar y, \bar z)  := \left\{\left(d^1, \, d^2, \, d^3\right)\in\mathbb{R}^{n+m+q}~\left|~
\begin{array}{rl}
\nabla G_i(\bar x, \bar y)^\top d^{12}=0,  &i\in \nu^1 \\
 \nabla g_j(\bar x, \bar y)^\top d^{12}=0,  &j\in \nu^2\\
  d_j^3=0, & j\in \eta^3
\end{array}\right.
\right\},
\end{eqnarray}
%where $d^{12}:=\left[\begin{array}{c}
%                       d^1\\
%                       d^2
%                     \end{array}\right]$,
%$d^{123} := \left[\begin{array}{c}
%                      d^{12}\\
%                       d^3
%                     \end{array}\right]$
%with $d^1\in \mathbb{R}^n$, $d^2\in \mathbb{R}^m$, and $d^3\in \mathbb{R}^q$.
We denote by $\nabla^2\mathcal{L}_1^{\lambda_1}$ the Hessian matrix of $\mathcal{L}_1^{\lambda_1}$ w.r.t $( x,   y,  z)$, i.e.,
$$
\nabla^2\mathcal{L}_1^{\lambda_1}(  \zeta)=\left[
\begin{array}{clc}
\nabla^2_{11}\mathcal{L}_1^{\lambda_1}( \zeta)&\nabla^2_{12}\mathcal{L}_1^{\lambda_1}( \zeta)^\top&\nabla^2_{13}\mathcal{L}_1^{\lambda_1}(  \zeta)^\top\\
\nabla^2_{12}\mathcal{L}_1^{\lambda_1}( \zeta)&\nabla^2_{22}\mathcal{L}_1^{\lambda_1}(  \zeta)&\nabla^2_{23}\mathcal{L}_1^{\lambda_1}( \zeta)^\top\\
\nabla^2_{13}\mathcal{L}_1^{\lambda_1}( \zeta)&\nabla^2_{23}\mathcal{L}_1^{\lambda_1}(  \zeta) &O
\end{array}
\right].
$$
%Moreover, $\nabla_{12} h(x,y,z)$ and  $\nabla h(x,y,z)$ represent the gradient of $h$ w.r.t $(x,y)$ and $(x,y,z)$, respectively.

 \begin{thm}[CD-regularity of $\Phi^{\lambda_1}$ in the general case]\label{SOSSC-Theorem 1-0}{Suppose that $f$ and $g$ are $\mathcal{C}^1$ and let $F$, $G$, $\nabla f$, $\nabla g_i$, $i=1, \ldots, q$ be SC$^1$. If the point $\bar \zeta:=(\bar x, \bar y, \bar z, \bar s, \bar u, \bar v, \bar w)$ satisfy the optimality conditions  \eqref{L-x}-\eqref{CP-w} for some $\lambda_1 >0$, then $\Phi^{\lambda_1}$ is CD-regular at $\bar \zeta$, provided that the following conditions hold:
\begin{itemize}
\item[(i)] The family of vectors $\left\{\left.\nabla G_i(\bar x,\bar y)\right|\,i\in I^1\right\} \cup  \left\{\left.\nabla g_j(\bar x,\bar y)\right|\,j\in I^2\right\} \cup \left\{\left.\nabla_{1, 2}\left(\nabla_{2_i}\bar\ell\right)(\bar x, \bar y, \bar z)\right|\,i=1, \ldots, m\right\}$ is linearly independent;
%\item[ii)] $\theta^3=\{j~|\; \bar w_j =0, \,\bar z_j =0\}=\emptyset$ and
\item[(ii)] For all $\left(d^1, \, d^2, \, d^3\right) \in Q_1(\bar x, \bar y, \bar z)$ with $d^{12} \neq 0$, we have
  \begin{equation}\label{SOSSC-0}
 (d^{123})^{\top} \nabla^2\mathcal{L}_1^{\lambda_1}( \bar\zeta)d^{123} >0.
  \end{equation}
  \end{itemize}
}
\end{thm}
\begin{proof}
Let $W^{\lambda_1}$ be any element from $\partial \Phi^{\lambda_1}(\bar \zeta)$. Then, it can take the form described in Theorem \ref{KKT Jacobian Phi}. Hence, to prove that $W^{\lambda_1}$ is non-singular, we need to show that for $d:=(d^1, d^2, d^3, d^4, d^5, d^6, d^7)$ with $d^1\in \mathbb{R}^n$, $d^2\in \mathbb{R}^m$, $d^3\in \mathbb{R}^q$, $d^4\in \mathbb{R}^m$, $d^5\in \mathbb{R}^p, d^6\in \mathbb{R}^q$, and $d^7\in \mathbb{R}^q$, we have $d =0$ whenever $W^{\lambda_1} d =0$. To proceed, start by noticing that from Theorem \ref{KKT Jacobian Phi}, $W^{\lambda_1} d =0$ is equivalently to
\begin{eqnarray}
  \nabla^2_{11} \mathcal{L}_1^{\lambda_1} d^1 +\nabla^2_{21} \mathcal{L}_1^{\lambda_1} d^2+\nabla^2_{31} \mathcal{L}_1^{\lambda_1} d^3+ \nabla_1\left(\nabla_{2}\bar\ell\right)^\top d^4+ \nabla_1 G^\top d^5 + \nabla_1 g^\top d^6=0,\label{p11}\\
 \nabla^2_{12} \mathcal{L}_1^{\lambda_1} d^1 +\nabla^2_{22} \mathcal{L}_1^{\lambda_1} d^2+\nabla^2_{32} \mathcal{L}_1^{\lambda_1} d^3+  \nabla_2\left(\nabla_{2}\bar\ell\right)^\top d^4+ \nabla_2 G^\top d^5 + \nabla_2 g^\top d^6=0,\label{p12}\\
 \nabla^2_{13} \mathcal{L}_1^{\lambda_1} d^1 +\nabla^2_{23} \mathcal{L}_1^{\lambda_1} d^2 - \nabla_{2} g d^4 +  d^7=0,\label{p13}\\
   \nabla^2_{12}\bar{\ell}d^1 +  \nabla^2_{22}\bar{\ell}d^2 - \nabla_{2} g^\top  d^3 =0,\label{p131}\\
  \forall j=1, \ldots, p: \;\; a^1_j \nabla G_j^\top d^{12} + b^1_j d^5_j =0,\label{p14}\\
    \forall j=1, \ldots, q: \;\; a^2_j \nabla g_j^\top d^{12} + b^2_j d^6_j =0,\label{p15}\\
     \forall j=1, \ldots, q: \;\; a^3_j d^{3}_j + b^3_j d^7_j =0,\label{p16}
\end{eqnarray}
where $\mathcal{L}_1^{\lambda_1}:=\mathcal{L}_1^{\lambda_1}(\bar \zeta)$, $\bar{\ell} := \bar{\ell}(\bar x, \bar y, \bar z)$, $G:=G(\bar x, \bar y)$, and $g:=g(\bar x, \bar y)$. Recall that $p$ and $q$ represent the number of components of upper- (resp. lower-) constraint functions of problem \eqref{eq:P}. For $i=1, 2, 3$, let $p^1:=p$, $p^2:=q$,  and $p^3:=q$. Since $(\bar x, \bar y, \bar z)$ satisfies the optimality conditions  \eqref{L-x}-\eqref{CP-w}, then it is feasible to problem \eqref{eq:KKTR}. Hence, from \eqref{KKT ab definition}, we have
\begin{eqnarray}
  \label{Pij-1}\theta^i&=&\left\{j\in\{1 \ldots, p^i\}~|~(a^i_j-1)^2 + (b_j^i+1)^2\leq 1 \right\},\\
  \label{Pij-2}  \eta^i&=&\left\{j\in\{1 \ldots, p^i\}~|~ a^i_j=0,~b^i_j=-1\right\},\\
 \label{Pij-3} \nu^i&=&\left\{j\in\{1 \ldots, p^i\}~|~ a^i_j=1,~b^i_j=0\right\}, \\
 \label{Pij-4} \{1, \ldots, p^i\}&=&\theta^i\cup \eta^i \cup \nu^i,~i=1,2,3.
\end{eqnarray}
%$P^i_1$ as the set of indices $j=1 \ldots, p_i$ such that  $a^i_j>0$ and $b^i_j<0$; $P^i_2$ as the set of indices $j=1 \ldots, p_i$ such that  $a^i_j=0$ and $b^i_j=-1$; and $P^i_3$ as the set of indices $j=1 \ldots, p_i$ such that  $a^i_j=1$ and $b^i_j=0$.
Clearly,  $a^i_j>0$ and $b^i_j<0$ for any $j\in \theta^i$, and \eqref{p14}--\eqref{p16} and \eqref{Pij-1}--\eqref{Pij-4} lead to the table
\vspace{-5mm}
\begin{table}[H]
{\renewcommand\baselinestretch{1.5}\selectfont
\begin{tabular}{p{2cm} p{5cm} p{3cm}l}\\\hline
 &$j \in \theta^i$&$j \in \eta^i$&$j \in \nu^i$\\\hline
 $i=1$&$\nabla G_j^\top d^{12}=c^1_j d^5_j$& $d^5_j=0$ & $\nabla G_j^\top d^{12}=0$ \\
  $i=2$&$\nabla g_j^\top d^{12}=c^2_j d^6_j$& $d^6_j=0$ & $\nabla g_j^\top d^{12}=0$\\
 $i=3$& $d^3_j=c^3_j d^7_j$ & $d^7_j=0$ & $d^3_j=0$ \\ \hline
\end{tabular}} %$\theta^3=\emptyset$
%\caption{Values of some terms under different scenarios}\label{perf-tab:1}
\end{table}
\noindent where $c^1_j:=-b^1_j/a^1_j>0$ for $j\in \theta^1$, $c^2_j:=-b^2_j/a^2_j>0$ for $j\in \theta^2$, and $c^3_j:=-b^3_j/a^3_j>0$ for $j\in \theta^3$. By respectively multiplying \eqref{p11}, \eqref{p12}, and \eqref{p13} from the left-hand-side by $(d^1)^\top$, $(d^2)^\top$  and $(d^3)^\top$, and adding the resulting sums together,
\begin{equation}\label{Quad-term}
    \begin{array}{l}
 (d^{123})^\top \nabla^2 \mathcal{L}_1^{\lambda_1} d^{123}+  (d^{4})^\top \nabla \left(\nabla_2 \bar\ell \right) d^{123} + (d^5)^\top\nabla G  d^{12}
  +
  (d^6)^\top\nabla g  d^{12} +

  (d^3)^\top  d^{7}=0.
\end{array}
\end{equation}
%where $\nabla^2 \mathcal{L}_1^{\lambda_1}$ and $\nabla  h$ respectively represent  the Hessian matrix of $\nabla^2 \mathcal{L}_1^{\lambda_1}$ and the gradient of $h$ w.r.t. $(x,y,z)$, and $\nabla G$ and $\nabla g$ respectively represent the gradient of $G$ and $g$ w.r.t. $(x,y)$.
Note that from \eqref{p131}, we have
\begin{equation}\label{d37} \nabla \left(\nabla_2 \bar\ell \right) d^{123}=\nabla_{1} \left(\nabla_2 \bar\ell \right)  d^1 + \nabla_{2} \left(\nabla_2 \bar\ell \right) d^2 - \nabla_{2} g^\top  d^3=0,\end{equation}
while from the table above, one can see that
\begin{eqnarray}\label{d567}
(d^5)^\top\nabla G  d^{12}+ (d^6)^\top\nabla g  d^{12}+ (d^3)^\top  d^{7}&=&\sum_{j=1}^{p^1} d^5_j\nabla G_j^\top  d^{12}+\sum_{j=1}^{p^2} d^6_j\nabla g_j ^\top d^{12}+\sum_{j=1}^{p^3}d^3_j  d^{7}_j\nonumber\\
&=&\sum_{j\in\theta^1} c^1_j(d^5_j)^2+\sum_{j\in\theta^2} c^2_j(d^6_j)^2+\sum_{j\in\theta^3} c^3_j(d^7_j)^2.
  \end{eqnarray}
Combining \eqref{Quad-term}--\eqref{d567} yields that
\begin{equation}\label{Quad-term-1}
    \begin{array}{l}
 (d^{123})^\top \nabla^2 \mathcal{L}_1^{\lambda_1} d^{123}+ \sum_{j\in\theta^1} c^1_j(d^5_j)^2+\sum_{j\in\theta^2} c^2_j(d^6_j)^2+\sum_{j\in\theta^3} c^3_j(d^7_j)^2=0.
\end{array}
\end{equation}
Since $c^1_j >0$ for $j\in \theta^1$, $c^2_j >0$ for $j\in \theta^2$, $c^3_j >0$ for $j\in \theta^3$, and $(d^{123})^\top \nabla^2 \mathcal{L}_1^{\lambda_1} d^{123}>0$ for any $d^{123}\in Q_1(\bar x, \bar y, \bar z)\setminus \{0\}$ with $d^{12}\neq0$, we have $d^5_j=0$ for $j\in \theta^1$ and $d^6_j=0$ for $j\in \theta^2$, and $d^7_j=0$ for $j\in \theta^3$. Hence, $d^3_j=c^3_j d^7_j=0$ for $j\in \theta^3$. Also note from the table above, $d^3_j=0$ for $j\in \nu^3$ and by definition \eqref{cone of feasible},  $d^3_j=0$ for $j\in \eta^3$. Inserting these values in \eqref{p11}--\eqref{p13}, it holds that
 \begin{eqnarray}
  \sum^m_{i=1}d^4_i \nabla_1 \left(\nabla_{2_i}\bar\ell\right) + \sum_{j\in \nu^1}d^5_j\nabla_1 G_j + \sum_{j\in \nu^2}d^6_j\nabla_1 g_j &=&0,\label{p111}\\
   \sum^m_{i=1}d^4_i \nabla_2 \left(\nabla_{2_i}\bar\ell\right) + \sum_{j\in \nu^1}d^5_j\nabla_2 G_j + \sum_{j\in \nu^2}d^6_j\nabla_2 g_j &=&0,\label{p122}\\
    -\nabla_2 gd^4+  d^7&=&0, \label{p133}
% -\nabla_y g_j^\top d^4 &=&0,~j\in \eta^3.\label{p133-1}
\end{eqnarray}
while  considering the table above. Observe that \eqref{p111} and \eqref{p122} lead to
 \begin{eqnarray}
   \sum^m_{i=1}d^4_i \nabla_{1, 2} \left(\nabla_{2_i}\bar\ell\right) + \sum_{j\in \nu^1}d^5_j\nabla  G_j + \sum_{j\in \nu^2}d^6_j\nabla  g_j &=&0.
\end{eqnarray}
Assumption (i) implies that the family $\left\{\left.\nabla_{1, 2}\left(\nabla_{2_i}\bar\ell\right)\right|~i=1, \ldots, m\right\} \cup \left\{\left.\nabla  G_j\right|~j\in \nu^1\right\} \cup \left\{\left.\nabla  g_j\right|~j\in \nu^2 \right\}$ is linearly independent given to $\nu^1\subseteq I^1$ and $\nu^2\subseteq I^2$. This suffices to ensure that $d^4=0$, $d^5_j=0$, $j\in \nu^1$, and $d^6_j=0$, $j\in \nu^2$, and hence $d^7_j=0$, $j\in \nu^3$ by \eqref{p133}, which concludes the proof as we have shown that all the components of the vector $d$ are zero.
\end{proof}

We impose $d^{12}\neq0$ in assumption (ii) because we automatically have
$$
(d^{123})^{\top} \nabla^2\mathcal{L}_1^{\lambda_1}(\bar\zeta)d^{123} =0 \; \mbox{ for any } \; \left(d^1, d^2, d^3\right)\; \mbox{ with }\; d^{12}=0.
$$
%Hence, having $d^{12}=0$ will lead to the failure of \eqref{SOSSC-0}.
Furthermore, one can observe that   assumption (i) is not appropriate for situations where the functions $f$ and $g$ are linear functions,  as $\nabla_{1, 2}\left(\nabla_2 \bar\ell \right)(\bar x,\bar y, \bar z)=0$ in this case. To deal with such a scenario, we propose the following modification of the above result. %In this theorem, $g_{\theta^3\cup\eta^3}$ represents vector-valued function made only of the components of $g : \mathbb{R}^{n+m} \rightarrow \mathbb{R}^q$ in $\theta^3\cup\eta^3$. %Therefore, under extra conditions, we could get rid of the assumption on $\{\nabla_{12} h(\bar x,\bar y, \bar z)^\top\}$, which is presented in our next result.

\begin{thm}[CD-regularity of $\Phi^{\lambda_1}$ under a full rank condition]\label{SOSSC-cor 1-0}Let the functions  $F$, $G$, $\nabla f$, and $\nabla g_i$, $i=1, \ldots, q$ be SC$^1$ and let the point $\bar \zeta:=(\bar x, \bar y, \bar z, \bar s, \bar u, \bar v, \bar w)$ satisfy the optimality conditions  \eqref{L-x}--\eqref{CP-w} for some $\lambda_1 >0$.  Then  $\Phi^{\lambda_1}$ is CD-regular at $\bar \zeta$ provided the following conditions hold:
\begin{itemize}
\item[(i)] The family of vectors $\left\{\left.\nabla G_i(\bar x,\bar y)\right|~i\in I^1\right\} \cup  \left\{\left.\nabla g_j(\bar x,\bar y)\right|~j\in I^2\right\}$ is linearly independent;
\item[(ii)]$\nabla_2 g_{\theta^3\cup\eta^3}(\bar x,\bar y)$ has a full column rank, where $\nabla_2 g_{\theta^3\cup\eta^3}(\bar x,\bar y)$ is the submatrix containing rows of $\nabla_2 g(\bar x,\bar y)$ indexed on $\theta^3\cup\eta^3$;
\item[(iii)] For all $\left(d^1, \, d^2, \, d^3\right) \in Q_1(\bar x, \bar y, \bar z)$ with $d^{12} \neq 0$, we have
  \begin{equation*}
 (d^{123})^{\top} \nabla^2\mathcal{L}_1^{\lambda_1}( \bar\zeta)d^{123} >0.
  \end{equation*}
  \end{itemize}
\end{thm}
\begin{proof} The proof follows on the lines of that of Theorem \ref{SOSSC-Theorem 1-0} till \eqref{p111}--\eqref{p133} and from the last of these equations, we have
$$
-\nabla_2 g_j(\bar x, \bar z)^\top d^4 =0 \; \mbox{ for } \; j\in \theta^3\cup\eta^3 \; \mbox{ and } \; -\nabla_2 g_j(\bar x, \bar z)^\top d^4 +d_j^7=0 \; \mbox{ for } \; j\in\nu ^3.
$$
If the matrix $\nabla_2 g_{\theta^3\cup\eta^3}(\bar x,\bar y)$ has full column rank, then it holds that $d^4=0$ and hence $d_j^7=0,~j\in\nu ^3$.   This together with \eqref{p111} and \eqref{p122} leads to
$$
\sum_{j\in \nu^1}d^5_j\nabla G_j(\bar x, \bar y) + \sum_{j\in \nu^2}d^6_j\nabla g_j(\bar x, \bar y) = 0.
$$
Then considering assumption (i), the remaining part of the proof follows.
\end{proof}

Recall that the index sets $\eta^i$, $\nu^i$ and $\theta^i$ with $i=1, 2, 4$ that we used here are defined in \eqref{multiplier sets} and \eqref{nu2nu3}.  The following two results are the counterparts of Theorems \ref{KKT Jacobian Phi} and \ref{SOSSC-Theorem 1-0} in the context of problem \eqref{eq:LLVFR} developed in \cite{FischerZemZhou2019}.

\begin{thm}[estimate of the generalized Jacobian of $\Phi_{2}^{\lambda_2}$]\label{VF Jacobian Phi} Let the functions $F$, $G$, $f$, and $g$ be continuously differentiable at $\bar\zeta:=(\bar x,\bar y, \bar{z},  \bar u,\bar v,\bar w)$. For any $\lambda_2>0$, the function $\Phi_{2}^{\lambda_2}$ is semismooth at $\bar \zeta$ and any matrix $W^{\lambda_2}\in \partial \Phi_{2}^{\lambda_2}(\bar \zeta)$ can take the form
\begin{equation*}
W^{\lambda_2} =\left[
\begin{array}{cccccc}
\nabla^2_{11} \mathcal{L}_{2}^{\lambda_2}(\bar\zeta) & \nabla^2_{12} \mathcal{L}_{2}^{\lambda_2}(\bar\zeta)^\top   & \nabla^2_{13} \mathcal{L}_{2}^{\lambda_2}(\bar\zeta)^\top  & \nabla_1  G(\bar x,\bar y)^{\top} & \nabla_1  g(\bar x,\bar y)^{\top}&-\lambda_2 \nabla_1  g(\bar x,\bar z)^{\top}\\
\nabla^2_{12} \mathcal{L}_{2}^{\lambda_2}(\bar\zeta) &  \nabla^2_{22} \mathcal{L}_{2}^{\lambda_2}(\bar\zeta) & O&  \nabla_2  G(\bar x,\bar y)^{\top} & \nabla_2 g(\bar x,\bar y)^{\top}&O\\
 \nabla^2_{13} \mathcal{L}_{2}^{\lambda_2}(\bar\zeta)   &    O  & \nabla^2_{33} \mathcal{L}_{2}^{\lambda_2}(\bar\zeta) &   O & O & -\lambda_2 \nabla_{2}  g(\bar x,\bar{z})^{\top}\\
  \Lambda_1 \nabla_{1} G(\bar x,\bar y) &    \Lambda_1 \nabla_{2} G(\bar x,\bar y)   & O & \Gamma_1& O & O\\
    \Lambda_2 \nabla_{1} g(\bar x,\bar y) &    \Lambda_2 \nabla_{2} g(\bar x,\bar y)   & O & O&  \Gamma_2& O\\
       \Lambda_4 \nabla_{1} g(\bar x,\bar{z})& O    &  \Lambda_4\nabla_{2} g(\bar x,\bar{z})  & O& O & \Gamma_4
\end{array}
\right]
\end{equation*}
with  $\Lambda_i :={\rm diag} (a^i)$ and $\Gamma_i :={\rm diag}(b^i)$, $i=1, 2, 4$ defined in a way similar to \eqref{KKT ab definition}.
 %\begin{equation}\label{ VF ab definition}
%    (a^i_j,b^i_j)\left\{\begin{array}{lll}
%                  =&(0,-1) & \mbox{ if } \;j\in  \eta^i, \\
%                  =&(1,0) & \mbox{ if } \;j\in \nu^i, \\
%                  \in& \{(\alpha, \beta): \; (\alpha-1)^2 + (\beta+1)^2\leq 1\} & \mbox{ if }\; j\in \theta^i.
%                \end{array}
%\right.
% \end{equation}
\end{thm}
In the next result, we provide conditions ensuring that the function $\Phi^{\lambda_2}$ is CD-regular. To proceed, let us introduce the cone of feasible directions for problem \eqref{Penalized-LLVF},
\begin{equation*}\label{Ponana}
Q_2(\bar x, \bar y, \bar z)   := \left\{\left(d^1, \, d^2, \, d^3\right)\in\mathbb{R}^{n+m+m}~\left|
\begin{array}{rl}
 \nabla G_i(\bar x, \bar y)^\top d^{12}=0, & i\in \nu^1 \\
\nabla g_j(\bar x, \bar y)^\top d^{12}=0, & j\in \nu^2 \\
\nabla g_j(\bar x, \bar z)^\top d^{13}=0, & j\in \nu^4
\end{array}
\right.\right\}.
\end{equation*}
%where $d^{12}:=\left[\begin{array}{c}
%                       d^1\\
%                       d^2
%                     \end{array}\right]$,
%$d^{13} := \left[\begin{array}{c}
%                      d^{1}\\
%                       d^3
%                     \end{array}\right]$,
%$d^{123} := \left[\begin{array}{c}
%                      d^{12}\\
%                       d^3
%                     \end{array}\right]$
%with $d^1\in \mathbb{R}^n$, $d^2\in \mathbb{R}^m$, and $d^3\in \mathbb{R}^m$.
%%Recall the definition of $\ell$ in \eqref{ell}, we write $\nabla^2\ell=\left[
%%\begin{array}{cc}
%%\nabla^2_{xx}\ell(\bar x,   \bar y ',   \bar w)&\nabla^2_{xy'}\ell(\bar x,   \bar y ',   \bar w)^\top\\
%%\nabla^2_{xy'}\ell(\bar x,   \bar y ',   \bar w)&\nabla^2_{y'y'}\ell(\bar x,   \bar y ',   \bar w)
%%\end{array}
%%\right]$.
\begin{thm}[CD-regularity of $\Phi^{\lambda_2}$]\label{SOSSC-Theorem 1-1}{Let the functions  $F$, $G$, $f$, and $g$ are SC$^1$ and assume the point $\bar \zeta:=(\bar x, \bar y, \bar z, \bar u, \bar v, \bar w)$ satisfy the optimality conditions \eqref{KS-1}--\eqref{KS-3} for some $\lambda_2 >0$. Then  $\Phi^{\lambda_2}$ is CD-regular at $\bar \zeta$ provided that the following conditions hold:
\begin{itemize}
\item[(i)] $\left\{\left.\nabla G_i(\bar x,\bar y)\right|\; i\in I^1\right\} \cup  \left\{\left.\nabla g_j(\bar x,\bar y)\right|\; j\in I^2\right\}$ linearly independent as well as $\left\{\left.\nabla g_j(\bar x,\bar z)\right|~j\in I^4\right\}$;
\item[(ii)] $\theta^4 =\theta^g(\bar x,\bar z, \bar w)=\left\{j~|~\bar w_j=0, \; g_j(\bar x,\bar z)=0\right\}=\emptyset$;
\item[(iii)] For all $\left(d^1, \, d^2, \, d^3\right)\in Q_2(\bar x, \bar y, \bar z)\setminus \{0\}$, we have
  \begin{equation}\label{SOSSC-LLVFR}
  \begin{array}{l}
 (d^{123})^{\top} \nabla^2\mathcal{L}_2^{\lambda_2}(\bar \zeta) d^{123} > 0.
  \end{array}
  \end{equation}
  \end{itemize}
}
\end{thm}
We are now going to present some examples to illustrate the conditions in Theorems \ref{SOSSC-Theorem 1-0},  \ref{SOSSC-cor 1-0}, and   \ref{SOSSC-Theorem 1-1} and the fact that the set of assumptions required to ensure the convergence of Algorithm \ref{algorithm 1} in the context of \eqref{eq:KKTR} and \eqref{eq:LLVFR}, respectively, are not necessarily related to each other. %for  The first example shows that all those conditions are satisfied. For the second example, conditions in  Theorems \ref{SOSSC-Theorem 1-0} and  \ref{SOSSC-cor 1-0} fail to be met while succeed to be guaranteed in Theorem \ref{SOSSC-Theorem 1-1}. The last example makes at least one of conditions of each Theorem  violated.
\begin{exmp}[the sufficient conditions for CD-regularity hold for both $\Phi^{\lambda_1}$ and $\Phi^{\lambda_2}$] Consider an example of problem \eqref{eq:P} from \cite{SC98} with the data
$$
F(x,y) := (x-3)^2 +(y-2)^2, \;\, G(x,y) := \left[\begin{array}{c}
                                    - x\\
  							  x -8
                                  \end{array}\right], \;\;
f(x,y) := (y-5)^2, \;\, g(x,y):= \left[\begin{array}{c}
                                    -2x +y -1\\
                        x -2y +2\\
                        x +2y -14
                                  \end{array}\right].
$$
As the global optimal solution of the problem is $\bar x=1, \; \bar y=3$, let
\begin{eqnarray*}\bar \zeta_1& := &\Big(\underset{\bar x}{\underbrace{1}},\underset{\bar y}{\underbrace{3}},\underset{\bar z}{\underbrace{-4,~0,~0}},~\underset{\bar s}{\underbrace{0}},~\underset{\bar u}{\underbrace{0,~0}},~\underset{\bar v}{\underbrace{62,~0,~0}},~\underset{\bar w}{\underbrace{0,~48,~112}}~\Big)^\top,\\
\bar \zeta_2& :=&\Big(\underset{\bar x}{\underbrace{1}},\underset{\bar y}{\underbrace{3}},\underset{\bar z}{\underbrace{3}},~\underset{\bar u}{\underbrace{0,~0}},~\underset{\bar v}{\underbrace{6,~0,~0}},~\underset{\bar w}{\underbrace{4,~0,~0}}~\Big)^\top.
\end{eqnarray*}
Direct calculations show that $\bar \zeta_1$ and $\bar \zeta_2$ satisfy \eqref{L-x}-\eqref{CP-w} with $\lambda_1=16$ and \eqref{KS-1}-\eqref{KS-3} with $\lambda_2=2$, respectively. In addition, we have $I^1=\emptyset$, $I^2=I^4=\{1\}$, $I^3=\{2,3\}$, $\nu^1=\emptyset$, $\nu^2=\nu^4=\{1\}$, $\nu^3=\{2,3\}$, $\theta^3\cup\eta^3=\{1\}$, and $\theta^4=\emptyset$. % Then condition ii) in  Corollary $\ref{SOSSC-cor 1-0}$ does not hold while holds for Theorem $\ref{SOSSC-Theorem 1-1}$.
 One can easily check that
\[
\left\{\left.\nabla g_j(\bar x,\bar y)\right| j\in I^2\right\}= \left\{\left.\nabla g_j(\bar x,\bar z)\right| j\in I^4\right\}=\left[\begin{array}{r}
 -2\\
 1
  \end{array}\right],\;\,
  \left\{\nabla_{1, 2}\left(\nabla_2\bar\ell\right)(\bar x,\bar y, \bar z)^\top\right\}=\left[\begin{array}{r}
0\\
 2
  \end{array}\right],\;\,
%  %\nabla_y g(\bar x,\bar y)=\left[\begin{array}{r}
%1\\
%- 2\\
%2
%  \end{array}\right],
  \nabla_2 g_{\theta^3\cup\eta^3}(\bar x,\bar y)=1.
\]
Hence, condition (i) in Theorem $\ref{SOSSC-Theorem 1-0}$, (i)-- ii) in Theorem $\ref{SOSSC-cor 1-0}$, and (i)-- (ii) in Theorem $\ref{SOSSC-Theorem 1-1}$ hold. Moreover,

\[
Q_1(\bar x, \bar y, \bar z) = \left\{\left.\left(
d^1,\, 2d^1,\, d^3_1, \,0, \,0\right)^\top~\right|~ d^1,\, d_1^3\in\mathbb{R}
\right\}, \;\;
Q_2(\bar x, \bar y, \bar z) = \left\{\left.\left(
d^1, \, 2d^1,\, 2d^1\right)^\top~\right|~ d^1\in\mathbb{R}\right\},
\]
\[
 \nabla^2\mathcal{L}_1^{\lambda_1} (\bar\zeta_1)
=\left[\begin{array}{rrrrr}
2 &    0 &  -32 &   16 &   16\\
     0  &   2  &  16 &  -32&    32\\
   -32 &   16   &  0  &   0  &   0\\
    16&   -32  &   0  &   0  &   0 \\
    16 &   32   &  0  &   0  &   0
  \end{array}\right]\;\mbox{ and }\;
 \nabla^2\mathcal{L}_2^{\lambda_2} (\bar\zeta_2)
=\left[\begin{array}{rrr}
 2 &0 &  -4\\
0  &  2 &    2\\
   -4  & 2   &  0
    \end{array}\right].
\]
Hence,  $(d^{123})^{\top} \nabla^2\mathcal{L}_1^{\lambda_1}( \bar\zeta)d^{123}=10(d^1)^2>0$ for any $\left(d^1, d^2, d^3\right)\in Q_1(\bar x, \bar y, \bar z)\setminus \{0\}$ with $ d^{12}\neq 0$ and $(d^{123})^{\top} \nabla^2\mathcal{L}_2^{\lambda_2}( \bar\zeta_2)d^{123}=10(d^1)^2>0$ for any $\left(d^1, d^2, d^3\right)\in Q_2(\bar x, \bar y, \bar z)\setminus \{0\}$.
Overall, the conditions in Theorems $\ref{SOSSC-Theorem 1-0}$,  $\ref{SOSSC-cor 1-0}$, and  $\ref{SOSSC-Theorem 1-1}$  all hold; thus $\Phi^{\lambda_1}$ and $\Phi^{\lambda_2}$ are CD-regular at $\bar \zeta_1$ and $\bar \zeta_2$, respectively.
\end{exmp}

\begin{exmp}[the sufficient conditions for CD-regularity hold for $\Phi^{\lambda_1}$ but the ones for $\Phi^{\lambda_2}$ fail] Consider the example of problem \eqref{eq:P} with
$$
F(x,y) := x + y_2, \;\; G(x,y) := \left[\begin{array}{c}
                                    -x+2\\
                                      x-4
                                  \end{array}\right], \;\;
f(x,y) := 2y_1 + xy_2, \;\; g(x,y):= \left[\begin{array}{c}
                                   x - y_1 - y_2+4\\
                                     -y_1\\
                                     -y_2
                                  \end{array}\right]
$$
taken from \cite{Bard91}.
%$$
%\begin{array}{rll}
%  F(x,y)& := & x + y_2\\
%  G(x,y)& := &\left[\begin{array}{r}
%                       -x+2\\
%                       x-4
%                     \end{array}
%     \right]\\
%  f(x,y)& := &2y_1 + xy_2\\
%  g(x,y)& := & \left[\begin{array}{r}
%                       x - y_1 - y_2+4\\
%                       -y_1\\
%                       -y_2
%                     \end{array}
%     \right]
%\end{array}
%$$
The unique optimal solution being $(\bar x, \bar y)$ with $\bar x=2$ and $\bar y=(6,\, 0)^\top$, let
\begin{eqnarray*}\bar \zeta_1&:=&\Big(\underset{\bar x}{\underbrace{2}},~\underset{\bar y}{\underbrace{6,~0}},~\underset{\bar z}{\underbrace{-2,~0,~0}},~\underset{\bar s}{\underbrace{0.0077,~-0.0077}},~\underset{\bar u}{\underbrace{0.9923,~0}},~\underset{\bar v}{\underbrace{2,~0,~1}},~\underset{\bar w}{\underbrace{0,~5.9923,~0.0077}}~\Big)^\top,\\
\bar \zeta_2&:=&\Big(\underset{\bar x}{\underbrace{2}},~\underset{\bar y}{\underbrace{6,~0}},~\underset{\bar z}{\underbrace{5.5207,~0.4793}},~\underset{\bar u}{\underbrace{0.0415,~0}},~\underset{\bar v}{\underbrace{4,~0,~1}},~\underset{\bar w}{\underbrace{2,~0,~0}}~\Big)^\top.
\end{eqnarray*}
One can verify that $\bar \zeta_1$ and $\bar \zeta_2$ satisfy \eqref{L-x}--\eqref{CP-w} with $\lambda_1=1$ and \eqref{KS-1}--\eqref{KS-3} with $\lambda_2=2$, respectively. Furthermore, as $I^1=\{1\}$, $I^2=\{1,\, 3\}$, $I^3=\{2,\,3\}$, $I^4=\{1\}$, $\nu^1=\{1\}$, $\nu^2=\{1,\,3\}$, $\nu^3=\{2,\,3\}$, $\nu^4=\{1\}$, $\theta^3\cup\eta^3=\{1\}$, and $\theta^4 = \emptyset$, one can quickly check that
\begin{eqnarray*}
\left\{\left.\nabla G_i(\bar x,\bar y)\right|\, i\in I^1\right\} \cup \left\{\left.\nabla g_j(\bar x,\bar y)\right|\, j\in I^2\right\} = \left[\begin{array}{rrr}
 -1&1&0\\
 0&-1&0\\
 0&-1&-1
  \end{array}\right],\;\;\, \left\{\left.\nabla g_j(\bar x,\bar z)\right|\, j\in I^4\right\} = \left[\begin{array}{r}
 1\\
 -1\\
 -1
  \end{array}\right], \\
  \nabla_{1, 2} (\nabla_2 \bar\ell)(\bar x, \bar y, \bar z)^\top = \left[\begin{array}{rr}
 0&1 \\
 0&0\\
 0&0
  \end{array}\right], \;\; \nabla_2 g_{\theta^3\cup\eta^3}(\bar x,\bar y)=\left[\begin{array}{rr}
 1 &-1
  \end{array}\right],
\end{eqnarray*}
which imply that the conditions (i) in Theorem $\ref{SOSSC-Theorem 1-0}$ and  (ii) in Theorem $\ref{SOSSC-cor 1-0}$ do not hold but  (i)--(ii) in Theorem $\ref{SOSSC-Theorem 1-1}$  are satisfied. Moreover, as
$
Q_2(\bar x, \bar y, \bar z)=\{0\},
$
%and
%\begin{eqnarray*}
% \nabla^2\mathcal{L}_1^{\lambda_1} (\bar\zeta_1)
%=\left[\begin{array}{rrrrrr}
%2 &    0 &  0 &  -1& 0 & 0\\
%0  &   2  &  2 &  -1&    -1& 0\\
%0  &   2  &  2 &  -1&   0 & -1\\
%-1&  -1  &   -1  &   0  &   0 & 0\\
%0 &   -1   &  0  &   0  &   0& 0\\
%0&  0 &   -1  &   0  &   0 & 0
%  \end{array}\right].
%\end{eqnarray*}
it follows that $\Phi^{\lambda_2}$ is CD-regular at $\bar \zeta_2$.
\end{exmp}

To conclude this section, we would like to point out the analogy between the assumptions in Theorems \ref{SOSSC-Theorem 1-0}, \ref{SOSSC-cor 1-0}, and \ref{SOSSC-Theorem 1-1} with corresponding conditions ensuring the convergence of the semismooth Newton method in a standard nonlinear optimization problem \cite{QiJiangSemismooth1997}. An interesting point though is that the corresponding conditions in the context of standard nonlinear optimization also guaranty that a stationarity point satisfying them is locally optimal; cf. latter reference. This is unfortunately not the case for the bilevel optimization problem. In the next example, we show that assumptions  Theorems \ref{SOSSC-Theorem 1-0}, \ref{SOSSC-cor 1-0}, and \ref{SOSSC-Theorem 1-1} can all fail at a stationary point, which corresponds to a locally optimal solution of a given bilevel program.

\begin{exmp}[a point is locally optimal while the sufficient conditions for convergence for Algorithm \ref{algorithm 1} fail] Considering the problem in  Example $\ref{ex-1}$ again, but with
\begin{eqnarray*}\bar \zeta_1&:=&\Big(\underset{\bar x}{\underbrace{0.5}},~\underset{\bar y}{\underbrace{0,~0.5}},~\underset{\bar z}{\underbrace{0,~-1,~0}},~\underset{\bar s}{\underbrace{0,-0.0061}},~\underset{\bar u}{\underbrace{0}},~\underset{\bar v}{\underbrace{1,~1,~0}},~\underset{\bar w}{\underbrace{0.0061,~0,~0.5061}}~\Big)^\top,\\
\bar \zeta_2&:=&\Big(\underset{\bar x}{\underbrace{0.5}},~\underset{\bar y}{\underbrace{0,~0.5}},~\underset{\bar z}{\underbrace{0,~0.5}},~\underset{\bar u}{\underbrace{0}},~\underset{\bar v}{\underbrace{1,~1,~0}},~\underset{\bar w}{\underbrace{0,~1,~0}}~\Big)^\top.
\end{eqnarray*}
 Direct calculations show that $\bar \zeta_1$ and $\bar \zeta_2$ satisfy \eqref{L-x}--\eqref{CP-w} with $\lambda_1=1$ and \eqref{KS-1}--\eqref{KS-3} with $\lambda_2=1$, respectively. In addition, $I^1=\{1\}$, $I^2=I^4=\{1,2\}$, $I^3=\{1,3\}$, $\nu^1=\emptyset$, $\nu^2=\{1,2\}$, $\nu^3=\{1,3\}$, $\nu^4=\{2\}$, $\theta^3\cup\eta^3=\{2\}$. Since $\left\{\nabla_{1, 2} \left(\nabla_2\bar\ell\right)(\bar x,\bar y, \bar z)^\top\right\}=0$,  $\nabla_2 g_{\theta^3\cup\eta^3}(\bar x,\bar y)=\left[-1 ~0~\right]$, and $\theta^4=\{1\}\neq\emptyset$,  conditions (i) in Theorem $\ref{SOSSC-Theorem 1-0}$, (ii) in Theorem $\ref{SOSSC-cor 1-0}$, and (ii) in  Theorem $\ref{SOSSC-Theorem 1-1}$ all fail.
 \end{exmp}
\begin{table}[ht]
\begin{center}
\begin{tabular}{l|l|c|l}
\hline
                   & \textbf{Needed for \eqref{eq:KKTR}}&                                            &  \textbf{ Needed for \eqref{eq:LLVFR}}      \\
\hline
   Computation of $\left.\partial\Phi^{\lambda_1}_{1}\right/\partial\Phi^{\lambda_2}_{2}$ &  $F$, $G$ are $\mathcal{C}^2$; $f$, $g$ are  $\mathcal{C}^3$ & $\Longrightarrow$  & $F$, $G$, $f$, and $g$  are $\mathcal{C}^2$  \\
\hline
    SC$^1$ of $\left.\Phi^{\lambda_1}_{1}\right/\Phi^{\lambda_2}_{2}$ &  $F$, $G$, $\nabla f$, and $\nabla g$  are SC$^1$ & $\Longrightarrow$  & $F$, $G$, $f$, and $g$  are SC$^1$  \\
\hline
    LC$^2$ of $\left.\Phi^{\lambda_1}_{1}\right/\Phi^{\lambda_2}_{2}$ &  $F$, $G$, $\nabla f$, and $\nabla g$  are LC$^2$ & $\Longrightarrow$  & $F$, $G$, $f$, and $g$  are LC$^2$  \\
\hline
\multirow{2}{*}{CD-regularity of $\left.\Phi^{\lambda_1}_{1}\right/\Phi^{\lambda_2}_{2}$}
& KKT-LICQ     &   &  LLVF-LICQ   \\
 & KKT-SSOSC &   &    LLVF-SSOSC         \\\hline
 \multirow{1}{*}{\textbf{$\sharp$ variables/equations in \eqref{KKT-Phi-lambda}/\eqref{VF-Phi-lambda}}}
& $n + 2m+ p + 3q$     &   &  $n+2m+p +2q$   \\
\hline
\end{tabular}
\end{center}
\caption{KKT-LICQ and KKT-SSOSC represent  (i) and (ii), respectively, in Theorem \ref{SOSSC-Theorem 1-0} or (i)--(ii) and (iii), respectively, in Theorem \ref{SOSSC-cor 1-0}. Similarly,  LLVF-LICQ and LLVF-SSOSC  correspond to (i) and (ii), respectively, in Theorem \ref{SOSSC-Theorem 1-1}.} %{As it is clear in these theorems, note that to get the CD-regularity  in the last row of the table, the SC$^1$ assumptions from the third row are also needed.}%\\[-1.5ex]
\end{table}
For the qualification conditions in Theorem \ref{SOSSC-Theorem 1-1} to guaranty that a point is locally optimal, much stronger second order sufficient conditions are needed; see \cite{FischerZemZhou2019,MehlitzZemkohoSufficient2019} for a detailed analysis of first and second order sufficient conditions for optimality in bilevel optimization.

Following up on the tradition adopted so far in this paper, we summarize the key features of the systems solved for \eqref{eq:KKTR} and \eqref{eq:LLVFR} and the corresponding requirements ensuring that Algorithm \ref{algorithm 1} converges.

\section{Numerical experiments}\label{Numerical experiments}
%{\color{blue}{
%  \begin{enumerate}
%    \item A quick look seems to suggest that things are showing more like anticipated? That is, the LLVF being more efficient?
%     \item It seems part of the plan was to also compare the reformulations by problem classes? (i.e., e.g., examples where there is no second order derivatives for lower-level problem involved and vice-versa?)
%    \item Maybe the large example from Patrick (Germany) should also be included in the calculations?
%    \item I think that for each value of lambda or so, feasibility and the constraint qualifications, mostly the KKT-MFCQ and LLVF-MFCQ, should also be compared. It seems to me that with comparison being done for given selections of lambda, maybe performance profiles might make sense? I think the paper ``Theoretical and numerical comparison of relaxation
%methods for mathematical programs
%with complementarity constraints'' by Tim Hoheisel,  Christian Kanzow, and
%Alexandra Schwartz would also be interesting to look at here (numerical section).
%\item As different selections of lambda might lead to different outcomes for both methods (vice-versa), maybe there could be a way to proceed by looking at average performance for various lambdas? Not sure this makes any sense.
%  \end{enumerate}
%}}

Based on our implementation of Algorithm \ref{algorithm 1} in \MATLAB (R2018a), we report and discuss test results obtained for the 124 nonlinear bilevel optimization problems in the current version of the BOLIB \cite{BOLIB2017} and a quadratic bilevel optimal control (BOC) program with large size from \cite{MehlitzGerd16}.

Recall that the necessary optimality conditions \eqref{L-x}--\eqref{CP-w} and \eqref{KS-1}--\eqref{KS-3} and their reformulation \eqref{KKT-Phi-lambda} and \eqref{VF-Phi-lambda} as nonsmooth system of equations contain the penalization parameter $\lambda>0$. Since there is no rule to select an appropriate $\lambda$, one may try all $\lambda$ from a certain finite discrete set in $(0,\infty)$, solve the corresponding optimality conditions, and then choose the best solution in terms of the upper-level objective function value.
For our approach, it turned out that a small set of $\lambda$-values  is sufficient to reach very good results.
To be precise, for all our experiments, we just used the 11 values of $\lambda$ in $\bar\Lambda:=\{2^{-3}, 2^{-2}, \cdots, 2^{6}, 2^{7}\}$.

\subsection{Implementation details and test problems}
Besides the selection of penalization parameters described before, the other parameters needed in Algorithm \ref{algorithm 1} are set to
\[
\beta:=10^{-8},\quad \epsilon:=10^{-8},\quad t:=2.1,\quad \rho:=0.5,\; \mbox{ and } \; \sigma:=10^{-4}.
\]
For each test example, we only use one starting point $(x^o, y^o)$ defined as follows.  If  an  example in the literature comes with a starting point, then we use this point for our experiments. Otherwise, we choose $x^o=\textbf{1}_n$ and $y^o=\textbf{1}_m$ except for the three examples $\sharp$20, 119, and 120 because their global optimal solutions are $\left(\textbf{1}_n, \textbf{1}_m\right)$. Note that $\textbf{1}_n:=(1,\cdots,1)^\top\in\mathbb{R}^n$, for example. So, for these three examples we use $x^o=-\textbf{1}_n$ and $y^o=-\textbf{1}_m$.  Detailed information on starting points can be found in \cite{FischerZemkohoZhouDetailed2018}.   % As a rule,  $(x^o,y^o)$ is chosen to be feasible for both the upper- and lower-level constraints, i.e., $G(x^o,y^o)\le 0$ and $g(x^o,y^o)\le 0$. In some cases, if such a feasible pair $(x^o,y^o)$ was not easy to find, we also used $x^o:=0$ and $y^o:=0$.
Moreover, to fully define $\zeta^o: =( x^o, y^o, z^o, s^o, u^o, v^o, w^o)$ in \eqref{KKT-Lagrangian} , we set
\begin{eqnarray*}
 z^o = - \left(\left|g_1(x^o,y^o)\right|, \ldots, \left|g_q(x^o,y^o)\right|\right)^\top, \;  u^o := \left(\left|G_1(x^o,y^o)\right|, \ldots, \left|G_p(x^o,y^o)\right|\right)^\top, \;
 v^o:=-z^o, \; w^o := v^o.
\end{eqnarray*}
As for \eqref{VF-Lagrangian} we define $\zeta^o=(x^o, y^o, z^o, u^o, v^o, w^o)$ by
\begin{eqnarray*}
z^o := y^o, \; u^o := \left(\left|G_1(x^o,y^o)\right|, \ldots, \left|G_p(x^o,y^o)\right|\right)^\top, \; v^o := \left(\left|g_1(x^o,y^o)\right|, \ldots, \left|g_q(x^o,y^o)\right|\right)^\top,\;
 w^o := v^o.
\end{eqnarray*}
In addition to the stopping criterion $\|\Phi^{\lambda}(\zeta^k)\|\le\epsilon$ used in Algorithm \ref{algorithm 1}, the algorithm is terminated if the iteration index $k$ reaches 2000.
Finally,  to pick an element from the generalized B-subdifferential $\partial_B\Phi^{\lambda}(\zeta^k)$ in Step 2 of Algorithm \ref{algorithm 1}, we adopt the technique in \cite{DeLuca1996}. Finally, we denote the semismooth Newton method (i.e., Algorithm \ref{algorithm 1}) for \eqref{L-x}--\eqref{CP-w} with $\lambda :=\lambda_1$ and \eqref{KS-1}--\eqref{KS-3}  with $\lambda :=\lambda_2$ as SNKKT  and SNLLVF, respectively.

%\newpage
%\begin{landscape}
{\liuhao \begin{longtable}[ht]{p{0.29cm}p{3.9cm}p{.9cm}rrrrrrrr}
\hline
$\sharp$ &\textbf{Example} &\textbf{Status}&\multicolumn{2}{c}{Known}& \multicolumn{2}{c}{\texttt{SNKKT}}& \multicolumn{2}{c}{\texttt{SNVF}}&\multicolumn{2}{c}{$\delta^{\lambda^*}$}\\\cline{4-11}
&&&$F_{known}$& $ f_{known}$&$F^{\lambda_1^*}$& $f^{\lambda_1^*}$& $F^{\lambda_2^*}$ & $f^{\lambda_2^*}$&$\delta^{\lambda_1^*}$&$\delta^{\lambda_2^*}$ \\\hline
1	&	\texttt{	AiyoshiShimizu1984Ex2	}&	optimal	&	5	&	0	&	34.77	&	0	&	4.97	&	0	&	5.95	&	0.01	\\\rowcolor[gray]{.9}	
2	&	\texttt{	AllendeStill2013	}&	optimal	&	1	&	-0.5	&	1	&	-0.5	&	0.99	&	-0.51	&	0	&	0.01	\\	
3	&	\texttt{	AnEtal2009	}&	optimal	&	2251.6	&	565.8	&	2251.6	&	565.8	&	2251.6	&	565.8	&	0	&	0	\\\rowcolor[gray]{.9}	
4	&	\texttt{	Bard1988Ex1	}&	optimal	&	17	&	1	&	17	&	1	&	17	&	1	&	0	&	0	\\	
5	&	\texttt{	Bard1988Ex2	}&	optimal	&	-6600	&	54	&	-6600	&	54	&	-6600	&	54	&	0	&	0	\\\rowcolor[gray]{.9}	
6	&	\texttt{	Bard1988Ex3	}&	optimal	&	-12.68	&	-1.02	&	-12.68	&	-1.02	&	-12.68	&	-1.02	&	0	&	0	\\	
7	&	\texttt{	Bard1991Ex1	}&	optimal	&	2	&	12	&	2	&	12	&	2	&	12	&	0	&	0	\\\rowcolor[gray]{.9}	
8	&	\texttt{	BardBook1998	}&	optimal	&	0	&	5	&	0	&	5	&	0	&	5	&	0	&	0	\\	
9	&	\texttt{	CalamaiVicente1994a	}&	optimal	&	0	&	0	&	0	&	0	&	0	&	0	&	0	&	0	\\\rowcolor[gray]{.9}	
10	&	\texttt{	CalamaiVicente1994b	}&	optimal	&	0.31	&	-0.41	&	0.31	&	-0.41	&	0.31	&	-0.41	&	0	&	0	\\	
11	&	\texttt{	CalamaiVicente1994c	}&	optimal	&	0.31	&	-0.41	&	0.31	&	-0.41	&	0.31	&	-0.41	&	0	&	0	\\\rowcolor[gray]{.9}	
12	&	\texttt{	CalveteGale1999P1	}&	optimal	&	-29.2	&	0.31	&	-29.2	&	0.31	&	-29.2	&	0.31	&	0	&	0	\\	
13	&	\texttt{	ClarkWesterberg1990a	}&	optimal	&	5	&	4	&	5	&	4	&	5	&	4	&	0	&	0	\\\rowcolor[gray]{.9}	
14	&	\texttt{	Colson2002BIPA1	}&	optimal	&	250	&	0	&	250	&	0	&	250	&	0	&	0	&	0	\\	
15	&	\texttt{	Colson2002BIPA2	}&	known	&	17	&	2	&	17	&	2	&	17	&	2	&	0	&	0	\\\rowcolor[gray]{.9}	
16	&	\texttt{	Colson2002BIPA3	}&	known	&	2	&	24.02	&	2	&	24.02	&	2	&	24.02	&	0	&	0	\\	
17	&	\texttt{	Colson2002BIPA4	}&	known	&	88.79	&	-0.77	&	88.79	&	-0.77	&	88.79	&	-0.77	&	0	&	0	\\\rowcolor[gray]{.9}	
18	&	\texttt{	Colson2002BIPA5	}&	known	&	2.75	&	0.57	&	2.75	&	0.55	&	2	&	-1	&	0	&	-0.27	\\	
19	&	\texttt{	Dempe1992a	}&	unknown	&		&		&	0	&	0.5	&	0	&	0.5	&		&		\\\rowcolor[gray]{.9}	
20	&	\texttt{	Dempe1992b	}&	optimal	&	31.25	&	4	&	31.25	&	4	&	31.25	&	4	&	0	&	0	\\	
21	&	\texttt{	DempeDutta2012Ex24	}&	optimal	&	0	&	0	&	0.12	&	0	&	0	&	0	&	0.12	&	0	\\\rowcolor[gray]{.9}	
22	&	\texttt{	DempeDutta2012Ex31	}&	optimal	&	-1	&	4	&	-0.7	&	3.33	&	-1.07	&	4.29	&	0.3	&	0.07	\\	
23	&	\texttt{	DempeFranke2011Ex41	}&	optimal	&	-1	&	-1	&	-1	&	-1	&	-1	&	-1	&	0	&	0	\\\rowcolor[gray]{.9}	
24	&	\texttt{	DempeFranke2011Ex42	}&	optimal	&	5	&	-2	&	5	&	-2	&	4.99	&	-2.01	&	0	&	0	\\	
25	&	\texttt{	DempeFranke2014Ex38	}&	optimal	&	2.13	&	-3.5	&	2.2	&	-3.5	&	2.13	&	-3.5	&	0.03	&	0	\\\rowcolor[gray]{.9}	
26	&	\texttt{	DempeEtal2012	}&	optimal	&	-1	&	-4	&	-1	&	-4	&	-1	&	-4	&	0	&	0	\\	
27	&	\texttt{	DempeLohse2011Ex31a	}&	optimal	&	-5.5	&	0	&	-5.5	&	0	&	-5.5	&	0	&	0	&	0	\\\rowcolor[gray]{.9}	
28	&	\texttt{	DempeLohse2011Ex31b	}&	optimal	&	-12	&	0	&	-12	&	0	&	-12	&	0	&	0	&	0	\\	
29	&	\texttt{	DeSilva1978	}&	optimal	&	-1	&	0	&	-1	&	0	&	-1.01	&	0	&	0	&	0.01	\\\rowcolor[gray]{.9}	
30	&	\texttt{	FalkLiu1995	}&	optimal	&	-2.2	&	0	&	-2.25	&	0	&	-2.22	&	0	&	0.02	&	0.01	\\	
31	&	\texttt{	FloudasEtal2013	}&	optimal	&	0	&	200	&	0	&	200	&	0	&	200	&	0	&	0	\\\rowcolor[gray]{.9}	
32	&	\texttt{	FloudasZlobec1998	}&	optimal	&	1	&	-1	&	1	&	-1	&	1	&	-1	&	0	&	0	\\	
33	&	\texttt{	GumusFloudas2001Ex1	}&	optimal	&	2250	&	197.8	&	2250	&	197.8	&	2250	&	197.8	&	0	&	0	\\\rowcolor[gray]{.9}	
34	&	\texttt{	GumusFloudas2001Ex3	}&	optimal	&	-29.2	&	0.31	&	-6	&	0.29	&	-29.2	&	0.31	&	0.79	&	0	\\	
35	&	\texttt{	GumusFloudas2001Ex4	}&	optimal	&	9	&	0	&	9	&	0	&	9	&	0	&	0	&	0	\\\rowcolor[gray]{.9}	
36	&	\texttt{	GumusFloudas2001Ex5	}&	optimal	&	0.19	&	-7.23	&	0.19	&	-7.23	&	0.19	&	-7.23	&	0	&	0	\\	
37	&	\texttt{	HatzEtal2013	}&	optimal	&	0	&	0	&	0	&	0	&	-0.13	&	0.02	&	0	&	0.13	\\\rowcolor[gray]{.9}	
38	&	\texttt{	HendersonQuandt1958	}&	known	&	-3266.7	&	-711.1	&	-3266.7	&	-711.1	&	-3275.2	&	-709.5	&	0	&	0	\\	
39	&	\texttt{	HenrionSurowiec2011	}&	optimal	&	0	&	0	&	0	&	0	&	0	&	0	&	0	&	0	\\\rowcolor[gray]{.9}	
40	&	\texttt{	IshizukaAiyoshi1992a	}&	optimal	&	0	&	-1.5	&	0	&	-0.02	&	0	&	0	&	0.98	&	1	\\	
41	&	\texttt{	KleniatiAdjiman2014Ex3	}&	optimal	&	-1	&	0	&	-1	&	0	&	-1	&	0	&	0	&	0	\\\rowcolor[gray]{.9}	
42	&	\texttt{	KleniatiAdjiman2014Ex4	}&	known	&	-10	&	-3.1	&	-2	&	-0.1	&	-6.04	&	-2.1	&	0.97	&	0.4	\\	
43	&	\texttt{	LamparSagrat2017Ex23	}&	optimal	&	-1	&	1	&	-1	&	1	&	-1	&	1	&	0	&	0	\\\rowcolor[gray]{.9}	
44	&	\texttt{	LamparSagrat2017Ex31	}&	optimal	&	1	&	0	&	1	&	0	&	1	&	0	&	0	&	0	\\	
45	&	\texttt{	LamparSagrat2017Ex32	}&	optimal	&	0.5	&	0	&	0.5	&	0	&	0.5	&	0	&	0	&	0	\\\rowcolor[gray]{.9}	
46	&	\texttt{	LamparSagrat2017Ex33	}&	optimal	&	0.5	&	0	&	0.5	&	0	&	0.5	&	0	&	0	&	0	\\	
47	&	\texttt{	LamparSagrat2017Ex35	}&	optimal	&	0.8	&	-0.4	&	0.8	&	-0.4	&	0.8	&	-0.4	&	0	&	0	\\\rowcolor[gray]{.9}	
48	&	\texttt{	LucchettiEtal1987	}&	optimal	&	0	&	0	&	0	&	0	&	0	&	0	&	0	&	0	\\	
49	&	\texttt{	LuDebSinha2016a	}&	known	&	1.14	&	1.18	&	1.11	&	1.95	&	1.14	&	1.18	&	0.65	&	0	\\\rowcolor[gray]{.9}	
50	&	\texttt{	LuDebSinha2016b	}&	known	&	0	&	1.66	&	0.05	&	1.17	&	0.03	&	1.2	&	0.05	&	0.03	\\	
51	&	\texttt{	LuDebSinha2016c	}&	known	&	1.12	&	0.06	&	1.64	&	0.03	&	1.12	&	0.06	&	0.47	&	0	\\\rowcolor[gray]{.9}	
52	&	\texttt{	LuDebSinha2016d	}&	unknown	&		&		&	-114.85	&	-16.07	&	-192	&	-192	&		&		\\	
53	&	\texttt{	LuDebSinha2016e	}&	unknown	&		&		&	1.1	&	-18.57	&	2.09	&	-17.68	&		&		\\\rowcolor[gray]{.9}	
54	&	\texttt{	LuDebSinha2016f	}&	unknown	&		&		&	0	&	0.13	&	0	&	0.13	&		&		\\	
55	&	\texttt{	MacalHurter1997	}&	optimal	&	81.33	&	-0.33	&	81.33	&	-0.34	&	81.33	&	-0.33	&	0	&	0	\\\rowcolor[gray]{.9}	
56	&	\texttt{	Mirrlees1999	}&	optimal	&	1	&	-1.02	&	0.01	&	-1.04	&	0.87	&	-1.07	&	0.99	&	0.13	\\	
57	&	\texttt{	MitsosBarton2006Ex38	}&	optimal	&	0	&	0	&	0	&	0	&	0	&	0	&	0	&	0	\\\rowcolor[gray]{.9}	
58	&	\texttt{	MitsosBarton2006Ex39	}&	optimal	&	-1	&	-1	&	-1	&	-1	&	-1	&	-1	&	0	&	0	\\	
59	&	\texttt{	MitsosBarton2006Ex310	}&	optimal	&	0.5	&	-0.1	&	0.5	&	-0.43	&	0.5	&	-0.1	&	0.33	&	0	\\\rowcolor[gray]{.9}	
60	&	\texttt{	MitsosBarton2006Ex311	}&	optimal	&	-0.8	&	0	&	-0.8	&	0	&	-0.5	&	0	&	0	&	0.3	\\	
61	&	\texttt{	MitsosBarton2006Ex312	}&	optimal	&	0	&	0	&	0	&	0	&	-0.02	&	0	&	0	&	0.02	\\\rowcolor[gray]{.9}	
62	&	\texttt{	MitsosBarton2006Ex313	}&	optimal	&	-1	&	0	&	0	&	-0.5	&	-1	&	0	&	1	&	0	\\	
63	&	\texttt{	MitsosBarton2006Ex314	}&	optimal	&	0.25	&	-0.08	&	0.06	&	0	&	0.22	&	-0.07	&	0.19	&	0.03	\\\rowcolor[gray]{.9}	
64	&	\texttt{	MitsosBarton2006Ex315	}&	optimal	&	0	&	-0.83	&	0	&	-0.83	&	0.65	&	-0.51	&	0	&	0.65	\\	
65	&	\texttt{	MitsosBarton2006Ex316	}&	optimal	&	-2	&	0	&	-1	&	0.25	&	-2.06	&	0	&	0.5	&	0.03	\\\rowcolor[gray]{.9}	
66	&	\texttt{	MitsosBarton2006Ex317	}&	optimal	&	0.19	&	-0.02	&	0.19	&	-0.02	&	0.24	&	0	&	0	&	0.05	\\	
67	&	\texttt{	MitsosBarton2006Ex318	}&	optimal	&	-0.25	&	0	&	0	&	0	&	0	&	0	&	0.25	&	0.25	\\\rowcolor[gray]{.9}	
68	&	\texttt{	MitsosBarton2006Ex319	}&	optimal	&	-0.26	&	-0.02	&	-0.26	&	-0.02	&	-0.26	&	-0.02	&	0	&	0	\\	
69	&	\texttt{	MitsosBarton2006Ex320	}&	optimal	&	0.31	&	-0.08	&	0.03	&	0	&	0.03	&	0	&	0.28	&	0.28	\\\rowcolor[gray]{.9}	
70	&	\texttt{	MitsosBarton2006Ex321	}&	optimal	&	0.21	&	-0.07	&	0.21	&	-0.07	&	0.21	&	-0.07	&	0	&	0	\\	
71	&	\texttt{	MitsosBarton2006Ex322	}&	optimal	&	0.21	&	-0.07	&	0.21	&	-0.07	&	0.21	&	-0.07	&	0	&	0	\\\rowcolor[gray]{.9}	
72	&	\texttt{	MitsosBarton2006Ex323	}&	optimal	&	0.18	&	-1	&	0.18	&	-1	&	0.18	&	-1	&	0	&	0	\\	
73	&	\texttt{	MitsosBarton2006Ex324	}&	optimal	&	-1.76	&	0	&	-1.75	&	0	&	-1.76	&	0	&	0	&	0	\\\rowcolor[gray]{.9}	
74	&	\texttt{	MitsosBarton2006Ex325	}&	known	&	-1	&	-2	&	0	&	0	&	0	&	0	&	1	&	1	\\	
75	&	\texttt{	MitsosBarton2006Ex326	}&	optimal	&	-2.35	&	-2	&	-2	&	-2	&	-2	&	-2	&	0.15	&	0.15	\\\rowcolor[gray]{.9}	
76	&	\texttt{	MitsosBarton2006Ex327	}&	known	&	2	&	-1.1	&	1.45	&	-0.4	&	1.16	&	-1.04	&	0.64	&	0.05	\\	
77	&	\texttt{	MitsosBarton2006Ex328	}&	known	&	-10	&	-3.1	&	-3.8	&	-1.78	&	-3.95	&	-0.1	&	0.62	&	0.97	\\\rowcolor[gray]{.9}	
78	&	\texttt{	MorganPatrone2006a	}&	optimal	&	-1	&	0	&	-1	&	0	&	-1	&	0	&	0	&	0	\\	
79	&	\texttt{	MorganPatrone2006b	}&	optimal	&	-1.25	&	0	&	-1.5	&	0.25	&	-1.25	&	0	&	0.25	&	0	\\\rowcolor[gray]{.9}	
80	&	\texttt{	MorganPatrone2006c	}&	optimal	&	-1	&	-0.25	&	0.75	&	0	&	-1	&	-0.25	&	1.75	&	0	\\	
81	&	\texttt{	MuuQuy2003Ex1	}&	known	&	-2.08	&	-0.59	&	-3.62	&	-1.2	&	-3.26	&	-1.13	&	-0.62	&	-0.55	\\\rowcolor[gray]{.9}	
82	&	\texttt{	MuuQuy2003Ex2	}&	known	&	0.64	&	1.67	&	0.64	&	1.68	&	0.64	&	1.68	&	0.01	&	0.01	\\	
83	&	\texttt{	NieWangYe2017Ex34	}&	optimal	&	2	&	0	&	2	&	0	&	2	&	0	&	0	&	0	\\\rowcolor[gray]{.9}	
84	&	\texttt{	NieWangYe2017Ex52	}&	optimal	&	-1.71	&	-2.23	&	-1.41	&	-2	&	-1.39	&	-2.03	&	0.17	&	0.19	\\	
85	&	\texttt{	NieWangYe2017Ex54	}&	optimal	&	-0.44	&	-1.19	&	0	&	0	&	-0.15	&	-0.03	&	1	&	0.98	\\\rowcolor[gray]{.9}	
86	&	\texttt{	NieWangYe2017Ex57	}&	known	&	-2	&	-1	&	0	&	0	&	-2.04	&	-0.99	&	1	&	0.01	\\	
87	&	\texttt{	NieWangYe2017Ex58	}&	known	&	-3.49	&	-0.86	&	-0.32	&	0.05	&	-3.53	&	-0.84	&	0.91	&	0.02	\\\rowcolor[gray]{.9}	
88	&	\texttt{	NieWangYe2017Ex61	}&	known	&	-1.02	&	-1.08	&	-1.02	&	-1.08	&	-1.02	&	-1.09	&	0	&	0	\\	
89	&	\texttt{	Outrata1990Ex1a	}&	known	&	-8.92	&	-6.05	&	-8.92	&	-6.14	&	-8.96	&	-6.08	&	0	&	0	\\\rowcolor[gray]{.9}	
90	&	\texttt{	Outrata1990Ex1b	}&	known	&	-7.56	&	-0.58	&	-7.58	&	-0.57	&	-7.62	&	-0.57	&	0.01	&	0.01	\\	
91	&	\texttt{	Outrata1990Ex1c	}&	known	&	-12	&	-112.71	&	-12	&	-76.45	&	-12	&	-174.81	&	0.32	&	0	\\\rowcolor[gray]{.9}	
92	&	\texttt{	Outrata1990Ex1d	}&	known	&	-3.6	&	-2	&	-3.6	&	-2	&	-3.6	&	-1.99	&	0	&	0.01	\\	
93	&	\texttt{	Outrata1990Ex1e	}&	known	&	-3.15	&	-16.29	&	-3.79	&	-17.95	&	-3.79	&	-17.94	&	-0.1	&	-0.1	\\\rowcolor[gray]{.9}	
94	&	\texttt{	Outrata1990Ex2a	}&	known	&	0.5	&	-14.53	&	0.5	&	-14.54	&	0.5	&	-14.53	&	0	&	0	\\	
95	&	\texttt{	Outrata1990Ex2b	}&	known	&	0.5	&	-4.5	&	0.5	&	-4.5	&	0.5	&	-4.5	&	0	&	0	\\\rowcolor[gray]{.9}	
96	&	\texttt{	Outrata1990Ex2c	}&	known	&	1.86	&	-10.93	&	1.86	&	-10.93	&	1.85	&	-10.93	&	0	&	0	\\	
97	&	\texttt{	Outrata1990Ex2d	}&	known	&	0.92	&	-19.47	&	0.4	&	-25.37	&	0.33	&	-25.8	&	-0.3	&	-0.33	\\\rowcolor[gray]{.9}	
98	&	\texttt{	Outrata1990Ex2e	}&	known	&	0.9	&	-14.94	&	0.9	&	-14.93	&	0.9	&	-15.11	&	0	&	0	\\	
99	&	\texttt{	Outrata1993Ex31	}&	known	&	1.56	&	-11.68	&	1.56	&	-11.68	&	1.56	&	-11.67	&	0	&	0	\\\rowcolor[gray]{.9}	
100	&	\texttt{	Outrata1993Ex32	}&	known	&	3.21	&	-20.53	&	3.21	&	-20.53	&	3.21	&	-20.49	&	0	&	0	\\	
101	&	\texttt{	Outrata1994Ex31	}&	known	&	3.21	&	-20.53	&	3.21	&	-20.53	&	3.2	&	-20.45	&	0	&	0	\\\rowcolor[gray]{.9}	
102	&	\texttt{	OutrataCervinka2009	}&	optimal	&	0	&	0	&	0	&	0	&	0	&	0	&	0	&	0	\\	
103	&	\texttt{	PaulaviciusEtal2017a	}&	optimal	&	0.25	&	0	&	0	&	0	&	0.31	&	-0.09	&	0.25	&	0.09	\\\rowcolor[gray]{.9}	
104	&	\texttt{	PaulaviciusEtal2017b	}&	optimal	&	-2	&	-1.5	&	-2	&	-1.5	&	-2	&	-1.5	&	0	&	0	\\	
105	&	\texttt{	SahinCiric1998Ex2	}&	optimal	&	5	&	4	&	5	&	4	&	5	&	4	&	0	&	0	\\\rowcolor[gray]{.9}	
106	&	\texttt{	ShimizuAiyoshi1981Ex1	}&	optimal	&	100	&	0	&	99.93	&	0	&	99.83	&	0	&	0	&	0	\\	
107	&	\texttt{	ShimizuAiyoshi1981Ex2	}&	optimal	&	225	&	100	&	225	&	100	&	225	&	100	&	0	&	0	\\\rowcolor[gray]{.9}	
108	&	\texttt{	ShimizuEtal1997a	}&	unknown	&	 	&		&	1.71	&	1.92	&	16.89	&	1.58	&		&		\\	
109	&	\texttt{	ShimizuEtal1997b	}&	optimal	&	2250	&	197.8	&	2250	&	197.8	&	2250	&	197.8	&	0	&	0	\\\rowcolor[gray]{.9}	
110	&	\texttt{	SinhaMaloDeb2014TP3	}&	known	&	-18.68	&	-1.02	&	-18.68	&	-1.02	&	-18.71	&	-1.02	&	0	&	0	\\	
111	&	\texttt{	SinhaMaloDeb2014TP6	}&	known	&	-1.21	&	7.62	&	-1.21	&	7.62	&	-1.21	&	7.62	&	0	&	0	\\\rowcolor[gray]{.9}	
112	&	\texttt{	SinhaMaloDeb2014TP7	}&	known	&	-1.96	&	1.96	&	-1.98	&	1.98	&	-1.96	&	1.96	&	0.01	&	0	\\	
113	&	\texttt{	SinhaMaloDeb2014TP8	}&	optimal	&	0	&	100	&	0.69	&	2.78	&	0	&	100	&	0.97	&	0	\\\rowcolor[gray]{.9}	
114	&	\texttt{	SinhaMaloDeb2014TP9	}&	known	&	0	&	1	&	0	&	1	&	0	&	1	&	0	&	0	\\	
115	&	\texttt{	SinhaMaloDeb2014TP10	}&	known	&	0	&	1	&	0	&	1	&	0	&	1	&	0	&	0	\\\rowcolor[gray]{.9}	
116	&	\texttt{	TuyEtal2007	}&	optimal	&	22.5	&	-1.5	&	4.44	&	-2	&	22.5	&	-1.5	&	0.8	&	0	\\	
117	&	\texttt{	Vogel2012	}&	optimal	&	1	&	-2	&	4	&	-2	&	1.78	&	-0.96	&	3	&	0.78	\\\rowcolor[gray]{.9}	
118	&	\texttt{	WanWangLv2011	}&	optimal	&	10.62	&	-0.5	&	10.63	&	-0.5	&	10.63	&	-0.5	&	0	&	0	\\	
119	&	\texttt{	YeZhu2010Ex42	}&	optimal	&	1	&	-2	&	5	&	2	&	1.05	&	-2	&	4	&	0.05	\\\rowcolor[gray]{.9}	
120	&	\texttt{	YeZhu2010Ex43	}&	optimal	&	1.25	&	-2	&	9	&	2	&	1	&	-2	&	6.2	&	0.2	\\	
121	&	\texttt{	Yezza1996Ex31	}&	optimal	&	1.5	&	-2.5	&	1.5	&	-2.5	&	1.5	&	-2.5	&	0	&	0	\\\rowcolor[gray]{.9}	
122	&	\texttt{	Yezza1996Ex41	}&	optimal	&	0.5	&	2.5	&	0.5	&	2.5	&	0.62	&	3	&	0	&	0.2	\\	
123	&	\texttt{	Zlobec2001a	}&	optimal	&	-1	&	-1	&	-1	&	-1	&	-1	&	-1	&	0	&	0	\\\rowcolor[gray]{.9}	
124	&	\texttt{	Zlobec2001b	}&	unknown	&		&		&	1	&	-1	&	1	&	-1	&		&		\\	
%\end{tabular}
${}$\\
\caption{Objective function values at the solution for different selections of  $\lambda\in\bar\Lambda$. \label{numerical-results}}
\end{longtable}}

\subsection{Test  examples}
We first apply SNKKT  and SNLLVF to solve 124 test examples from the BOLIB library \cite{BOLIB2017}. Table \ref{numerical-results} lists values of the leader's objective function $F$ and follower's objective function $f$. The columns   $F_{known}$  and   $f_{known}$  show  the best known $F$-values and $f$-values from the literature. Such a value was not available for 6 of the test problems. This is marked by ``unknown'' in the \textbf{Status} column. For 83 examples, the best known $F$-value and $f$-value are even optimal (with status labelled as ``optimal''). For the remaining 35 test problems, the known  $F$-value might not be optimal and its status is just set to ``known''.

Note that examples $\sharp$14, 39, and 40 contain a parameter that should be provided by the user. The first one is associated with $\rho\geq 1$, which separates the problem into 4 cases: (i) $\rho=1$, (ii) $1<\rho<2$, (iii) $\rho=2$, and (iv) $\rho>2$. The results presented in Table \ref{numerical-results} correspond to case (i). For the other three cases, our method still produces the true global optimal solutions.  Example $\sharp$39 has a unique global optimal solution and results given in Table \ref{numerical-results} are for $c=0$. We also tested our method when $c=\pm1$, and obtained the unique optimal solutions as well. Example $\sharp$40 contains the parameter $M>1$, and the results presented in Table \ref{numerical-results} correspond to $M=1.5$.

Among the 124 examples, there are 60  where the implementation of SNKKT requires the calculation of 3rd order derivatives for $f$ or $g$. Those examples are labelled as \textbf{Group B}.
The remaining 64 examples with $f$ or $g$ where there is no need to compute  3rd order derivatives are categorized as \textbf{Group A}. We will demonstrate that SNKKT and SNVF have similar computational speed on solving examples in  \textbf{Group A} but significantly different speed for problems in  \textbf{Group B}.

\begin{table}[!th]
{	\renewcommand{\arraystretch}{1.15}\addtolength{\tabcolsep}{-1.5pt}
\begin{tabular}{llrrrrrrrrrrr}\\\hline
$\lambda$ &&$2^{-3}$&$2^{-2}$&$2^{-1}$&$2^{0}$&$2^{1}$&$2^{2}$&$2^{3}$&$2^{4}$&$2^{5}$&$2^{6}$&$2^{7}$\\\hline
Average &	\mbox{{\footnotesize{SNKKT}}}	&	166.1	&	173.7	&	157.9	&	170.9	&	212.4	&	233.2	&	308.9	&	339.0	&	351.2	&	414.9	&	447.1\\
Iter	&		\mbox{{\footnotesize{SNVF}}}	&	154.0	&	113.4	&	181.7	&	85.2	&	144.3	&	154.4	&	198.6	&	300.5	&	377.3	&	361.1	&	465.8	\\\hline
Average &	\mbox{{\footnotesize{SNKKT}}}	&	5.11	&	5.55	&	3.99	&	1.23	&	4.27	&	5.68	&	5.75	&	5.12	&	5.67	&	7.74	&	7.43	\\
Time& \mbox{{\footnotesize{SNVF}}}	&	0.17	&	0.10	&	0.16	&	0.07	&	0.15	&	0.14	&	0.21	&	0.26	&	0.31	&	0.29	&	0.36		\\\hline
Number of&	\mbox{{\footnotesize{SNKKT}}}	&	8	&	8	&	6	&	6	&	9	&	12	&	12	&	15	&	13	&	17	&	17	\\
Failures	&	SNVF	&	6	&	2	&	8	&	3	&	3	&	1	&	6	&	9	&	13	&	13	&	17	\\\hline
\multirow{2}{*}{$\alpha_K=1$}	&	\mbox{{\footnotesize{SNKKT}}}	&	114	&	109	&	107	&	110	&	107	&	109	&	109	&	106	&	105	&	105	&	100	\\
	& \mbox{{\footnotesize{SNVF}}}	&	107	&	112	&	109	&	113	&	112	&	116	&	112	&	108	&	107	&	110	&	103		\\\hline
${}$\\
\end{tabular}}
\caption{Performance of {\footnotesize{SNKKT}} and {\footnotesize{SNVF}} on solving 124 examples  for $\lambda\in\bar\Lambda$.}\label{perf-tab:1}
\end{table}

\subsection{Comparison of SNKKT  and SNVF on solving  BOLIB examples}
The first comparison is to see the ability  of SNKKT  and SNVF on solving the 124 test examples from the BOLIB library. Detailed results are listed in Table \ref{numerical-results}, where  columns $F^{\lambda_1^*}$, $f^{\lambda_1^*}$ and columns $F^{\lambda_1^*}$, $f^{\lambda_2^*}$ show the  values obtained through SNKKT  and SNVF for one of the eleven  penalization  parameters in $\bar\Lambda$, respectively. Note that evaluating the performance of an algorithm for the bilevel optimization problem \eqref{eq:P} is a difficult task since the decision whether a computed point is (close
to) a global solution of \eqref{eq:P} basically requires computing the LLVF $\varphi$. Therefore, instead of doing this, we suggest the following way of comparing our obtained results  with the
results from literature known for the test problems. For an approximate solution $(x,y)$ obtained from Algorithm \ref{algorithm 1}, we first compute
\[
 \delta^\lambda_F:=\frac{F^\lambda-F_{known}}{\max\{1,|F_{known}|\}},\quad\delta^\lambda_f:=\frac{f^\lambda-f_{known}}{\max\{1,|f_{known}|\}},
\]
where $F_{known}$ and  $f_{known}$ are the best known $F$-value and $f$-value from literature, $F^\lambda$ and $f^\lambda$ are the  objective function values generated by Algorithm \ref{algorithm 1} for a given $\lambda\in \bar\Lambda$. Moreover, we set
\[
 \delta^\lambda:=\left\{
 \begin{array}{ll}
 \max\{|\delta^\lambda_F|,|\delta^\lambda_f|\},\quad&\mbox{if  \textbf{Status} is optimal},\\
 \max\{\delta^\lambda_F,\delta^\lambda_f\},\quad&\mbox{if  \textbf{Status} is known}.
 \end{array}
 \right.
\]
In the latter case, $\delta^\lambda$ can become negative. This means that both $F$ and $f$ are smaller than the values for the point with best $F$-value and $f$-value known in the literature. We then pick the value of $\lambda^*$ via the following rule:
\[
 \lambda^*=\left\{
 \begin{array}{ll}
 {\rm argmin}_{\lambda\in\bar\Lambda} F^\lambda,\quad&\mbox{if  \textbf{Status} is unkown},\\
 {\rm argmin}_{\lambda\in\bar\Lambda} \delta^\lambda,\quad&\mbox{otherwise}.
 \end{array}
 \right.
\]
%$$\lambda^*={\rm argmin}_{\lambda\in\Lambda} \delta^\lambda.$$
Then we report the $F$-value and $f$-value $(F^{\lambda^*}, f^{\lambda^*})$ under  $\lambda^*$ and compare them with $(F_{known}, f_{known})$.
Since Algorithm \ref{algorithm 1}  has two versions (SNKKT associated with $\lambda=\lambda_1$  and SNVF associated with $\lambda=\lambda_2$),
Table \ref{numerical-results} lists  $(F^{\lambda_1^*}, f^{\lambda_1^*})$, $(F^{\lambda_2^*}$, $f^{\lambda_2^*})$ and $\delta^{\lambda_1^*}$ and  $\delta^{\lambda_2^*}$.  Note that it is not necessary that $\lambda^*_1=\lambda^*_2$ for each test example. %Moreover, for those examples with {\bf{Status}} `unknown', namely, No 19, 52, 53, 54, 108 and 124, we report the $(F^{\lambda^*}, f^{\lambda^*})$ in terms of the smallest $F^{\lambda^*}$, i.e., $\lambda^*={\rm argmin}_{\lambda\in\Lambda} F^\lambda.$
One can observe that there are 33 (resp. 21) examples with $\delta^{\lambda_1^*}\geq0.05$ (resp. $\delta^{\lambda_2^*}\geq0.05$), which means SNKKT (resp. SNVF) did not get improved solutions for those examples by using the above given starting point.

\begin{table}[!th]
{\renewcommand{\arraystretch}{1.15}\addtolength{\tabcolsep}{-1.5pt}
\begin{tabular}{lcccccc|cccccc}\\\hline
&\multicolumn{6}{c|}{\textbf{Group A}}&\multicolumn{6}{c}{\textbf{Group B}}\\\cline{2-7}\cline{8-13}
&\multicolumn{2}{c}{Aver. Iter}&\multicolumn{2}{c}{Aver. Time}&\multicolumn{2}{c|}{Aver. Time/Iter}&\multicolumn{2}{c}{Aver. Iter}&\multicolumn{2}{c}{Aver. Time}&\multicolumn{2}{c}{Aver. Time/Iter}\\\cline{2-7}\cline{8-13}
$\lambda$	&	\mbox{{\footnotesize{SNKKT}}}	&	\mbox{{\footnotesize{SNVF}}}	&	\mbox{{\footnotesize{SNKKT}}}	&	\mbox{{\footnotesize{SNVF}}}	&	\mbox{{\footnotesize{SNKKT}}}	&	\mbox{{\footnotesize{SNVF}}}	&	\mbox{{\footnotesize{SNKKT}}}	&	\mbox{{\footnotesize{SNVF}}}	&	\mbox{{\footnotesize{SNKKT}}}	&	\mbox{{\footnotesize{SNVF}}}	&	\mbox{{\footnotesize{SNKKT}}}	&	\mbox{{\footnotesize{SNVF}}}	\\\hline
$2^{-3}$	&	161.7	&	117.1	&	0.25	&	0.18	&	0.0016	&	0.0015	&	171.1	&	196.1	&	10.67	&	0.26	&	0.0623	&	0.0013	\\
$2^{-2}$	&	112.7	&	114.2	&	0.15	&	0.14	&	0.0013	&	0.0012	&	243.1	&	112.5	&	12.17	&	0.14	&	0.0500	&	0.0012	\\
$2^{-1}$	&	128.7	&	144.3	&	0.19	&	0.18	&	0.0015	&	0.0013	&	191.1	&	224.3	&	8.13	&	0.26	&	0.0426	&	0.0011	\\
$2^{0}$	&	119.9	&	87.2	&	0.16	&	0.10	&	0.0013	&	0.0011	&	228.9	&	82.8	&	2.28	&	0.07	&	0.0100	&	0.0009	\\
$2^{1}$	&	176.5	&	101.5	&	0.23	&	0.13	&	0.0013	&	0.0013	&	253.3	&	193.1	&	8.54	&	0.22	&	0.0337	&	0.0011	\\
$2^{2}$	&	197.6	&	149.4	&	0.29	&	0.17	&	0.0015	&	0.0011	&	273.6	&	160.0	&	12.28	&	0.18	&	0.0449	&	0.0012	\\
$2^{3}$	&	232.7	&	162.7	&	0.38	&	0.20	&	0.0016	&	0.0012	&	395.7	&	239.4	&	11.44	&	0.32	&	0.0289	&	0.0014	\\
$2^{4}$	&	347.8	&	332.3	&	0.49	&	0.34	&	0.0014	&	0.0010	&	329.0	&	264.4	&	10.15	&	0.29	&	0.0308	&	0.0011	\\
$2^{5}$	&	328.1	&	379.9	&	0.44	&	0.40	&	0.0013	&	0.0011	&	377.6	&	374.3	&	11.59	&	0.41	&	0.0307	&	0.0011	\\
$2^{6}$	&	371.4	&	351.9	&	0.52	&	0.35	&	0.0014	&	0.0010	&	464.3	&	371.6	&	15.87	&	0.41	&	0.0342	&	0.0011	\\
$2^{7}$	&	375.6	&	406.4	&	0.50	&	0.42	&	0.0013	&	0.0010	&	528.5	&	533.3	&	15.37	&	0.63	&	0.0291	&	0.0012	\\
\hline
${}$\\
\end{tabular}}
\caption{Performance of {\footnotesize{SNKKT}} and {\footnotesize{SNVF}} on examples from two groups for $\lambda\in\bar\Lambda$.}\label{perf-tab:2}
\end{table}
We then compare the test runs/number of iterations and computational time (in seconds) of SNKKT and SNVF.  As shown in  the first four rows of Table \ref{perf-tab:1}, SNVF uses fewer iterations for all examples except for $\lambda=2^{-1}$ or $2^7$ and runs much faster than SNKKT for all $\lambda\in\bar\Lambda$ in the average sense. In addition, as we mentioned before, we separated the 124 examples into  \textbf{Group A}  and  \textbf{Group B} in order to see the behaviour of these two methods on solving examples from each group.
Results are presented in Table \ref{perf-tab:2}. When the two methods are applied to solve examples from \textbf{Group A}, SNVF takes a smaller average time (e.g. Aver. Time) for each $\lambda$ and uses fewer average iterations (e.g. Aver. Iter) for  most values of $\lambda$. However, the average computational time per iteration (e.g. Aver. Time/Iter) for SNVF  was almost as same as the one needed by SNKKT. By contrast, when the two methods are applied to solve examples from \textbf{Group B}, the picture is significantly different. For each $\lambda$, the average time of SNVF is much smaller than that of SNKKT. Most importantly, as what we expected, SNVF ran much faster than SNKKT for each iteration because the Aver. Time/Iter of SNKKT is dozens of times higher than that of SNVF.  Interestingly, the Aver. Iter of SNKKT on solving examples in  \textbf{Group B} is more than that of solving examples in  \textbf{Group A} for all $\lambda$ except for $\lambda=2^4$. Similar observation can be seen on Aver. Iter for SNVF. This implies that the more information used did not necessarily led to fewer iterations.  For instance, results from the two methods on solving Example $\sharp$114 \texttt{SinhaMaloDeb2014TP9} are presented in Figure \ref{ex114}. This example has a very complicated lower level objective function,
$$f(x,y)=\exp\left[\left(1+ \frac{1}{4000}\sum^{10}_{i=1}y^2_i - \prod^{10}_{i=1}\cos\left(\frac{y_i}{\sqrt{i}}\right)\right)\sum^{10}_{i=1}x^2_i\right]$$
and thus the complexity of computing the third order derivative of $f$ is relatively high. Despite the fact that the two methods obtain the best known optimal solutions for each $\lambda$, SNKKT uses more iterations and took much longer than SNVF. The latter only needs 2 iterations with cost less than 1 second to produce the solution, while the former takes more than 8 iterations  and spent hundreds of seconds for all $\lambda$ (except for $\lambda=2^0$).

\begin{figure}[ht]
\centering
    \includegraphics[width=1\linewidth]{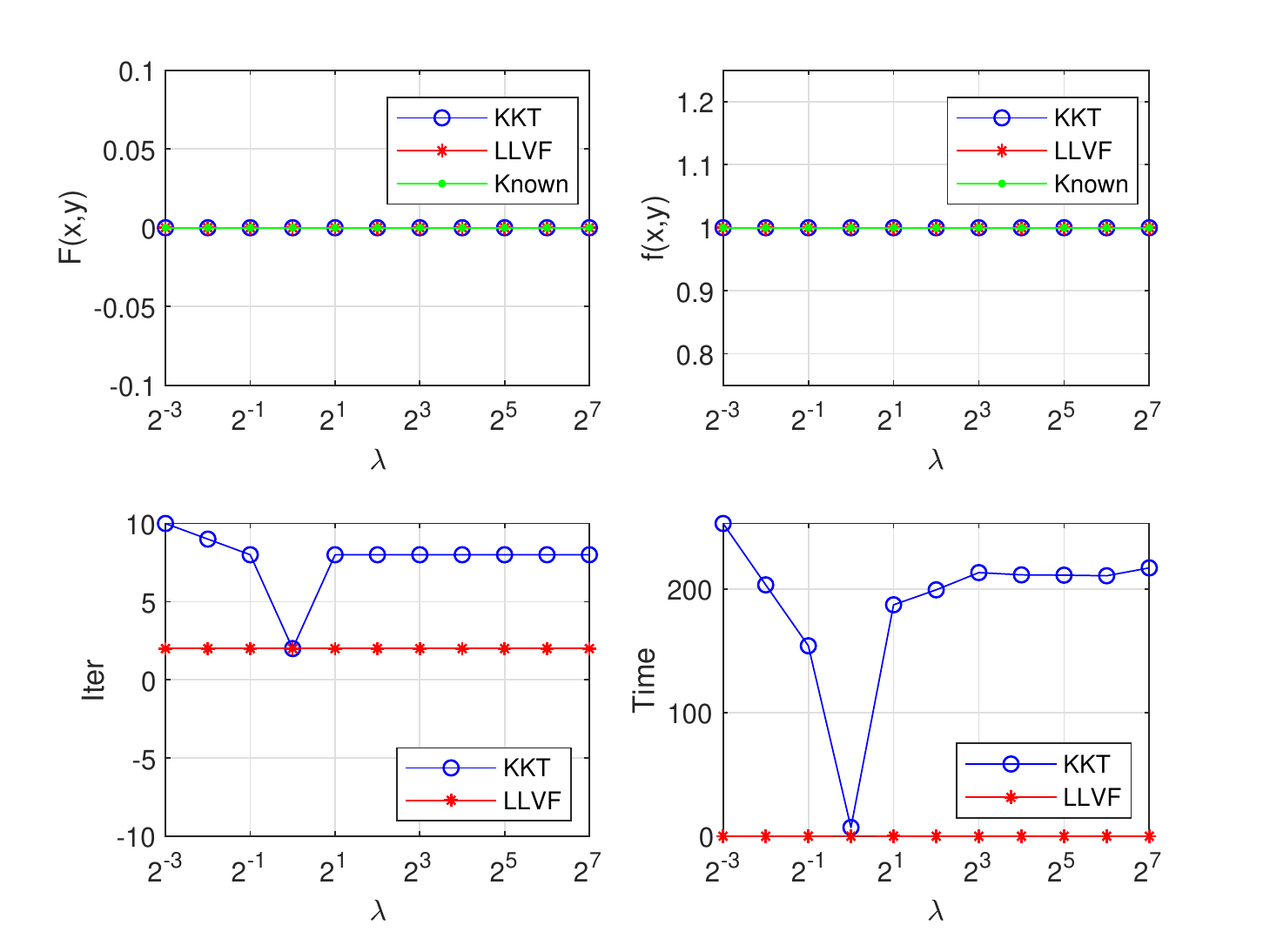}
  \caption{Performance of SNKKT and SNVF on solving Example $\sharp$114.}
  \label{ex114}
\end{figure}

\begin{figure}[ht]
\centering
    \includegraphics[width=.7\linewidth]{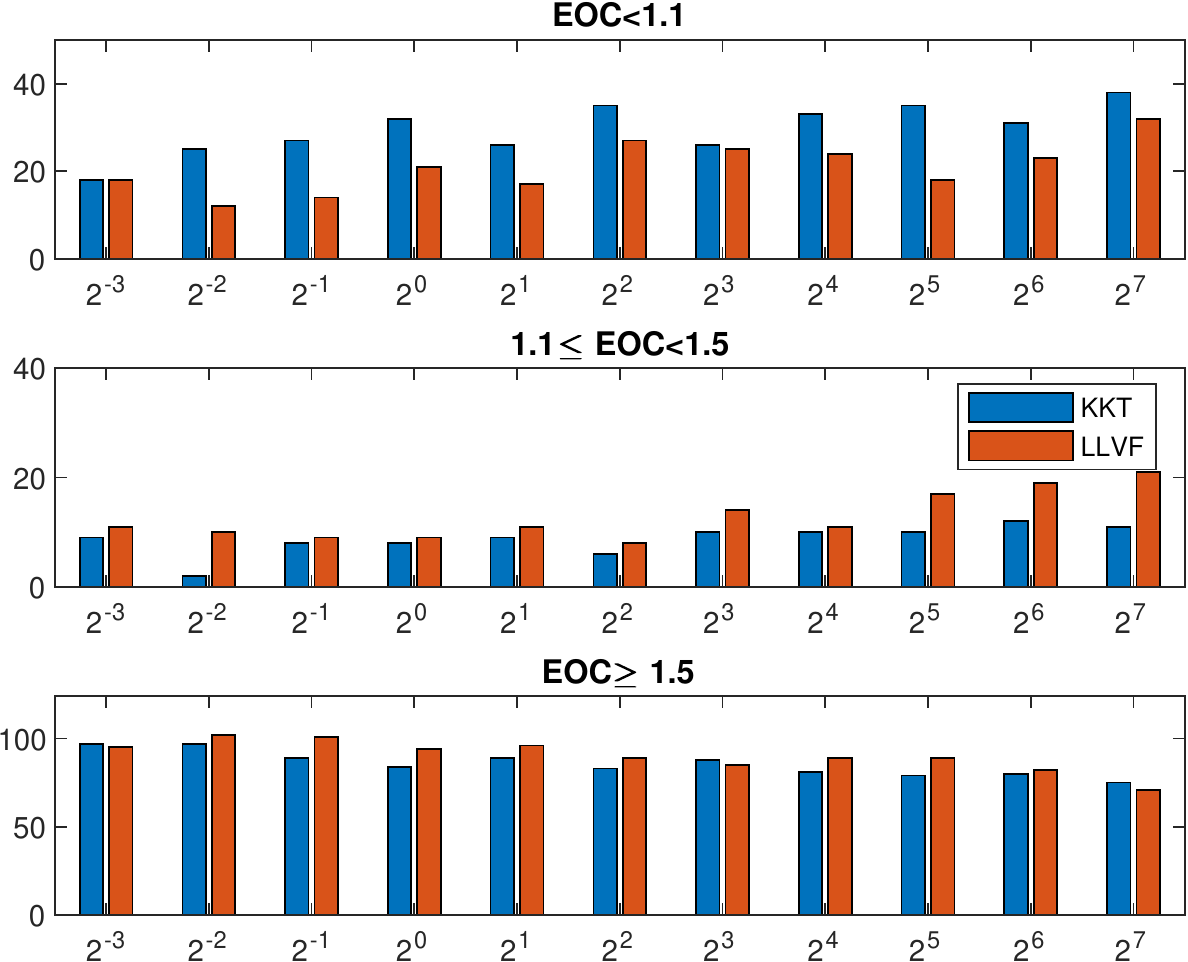}
  \caption{Experimental order of convergence (EOC) of SNKKT and SNVF for $\lambda\in\bar\Lambda$.}
  \label{eoc}
\end{figure}
We next report on examples where SNKKT and SNVF failed to solve, in the sense that the corresponding method ended with  $\|\Phi^{\lambda}(\zeta^K)\|\geq\epsilon$; hence, it did not converge before it stops. Here, note that $K$ denotes the iteration number,  where the methods are terminated.
 As shown in the third and forth rows in Table \ref{perf-tab:1}, it can be clearly seen that, for each $\lambda$ except for $\lambda=2^{-1}$, SNVF got better results than SNKKT in terms of number of failures. For instance, when $\lambda=2^2$, SNVF failed to solve one example, whilst SNKKT  failed to solve 12 examples. The last two rows  in Table \ref{perf-tab:1} list  the number of examples solved at last step with $\alpha_K=1$, a full Newton step. For all values of $\lambda\in \bar\Lambda$, both methods solved more than 100 examples while stopping with $\alpha_K=1$. And clearly, compared with SNKKT, more examples (except for the case of $\lambda=2^{-3}$) were handled by SNVF with $\alpha_K=1$.

\begin{figure}[ht]
\centering
    \includegraphics[width=.8\linewidth]{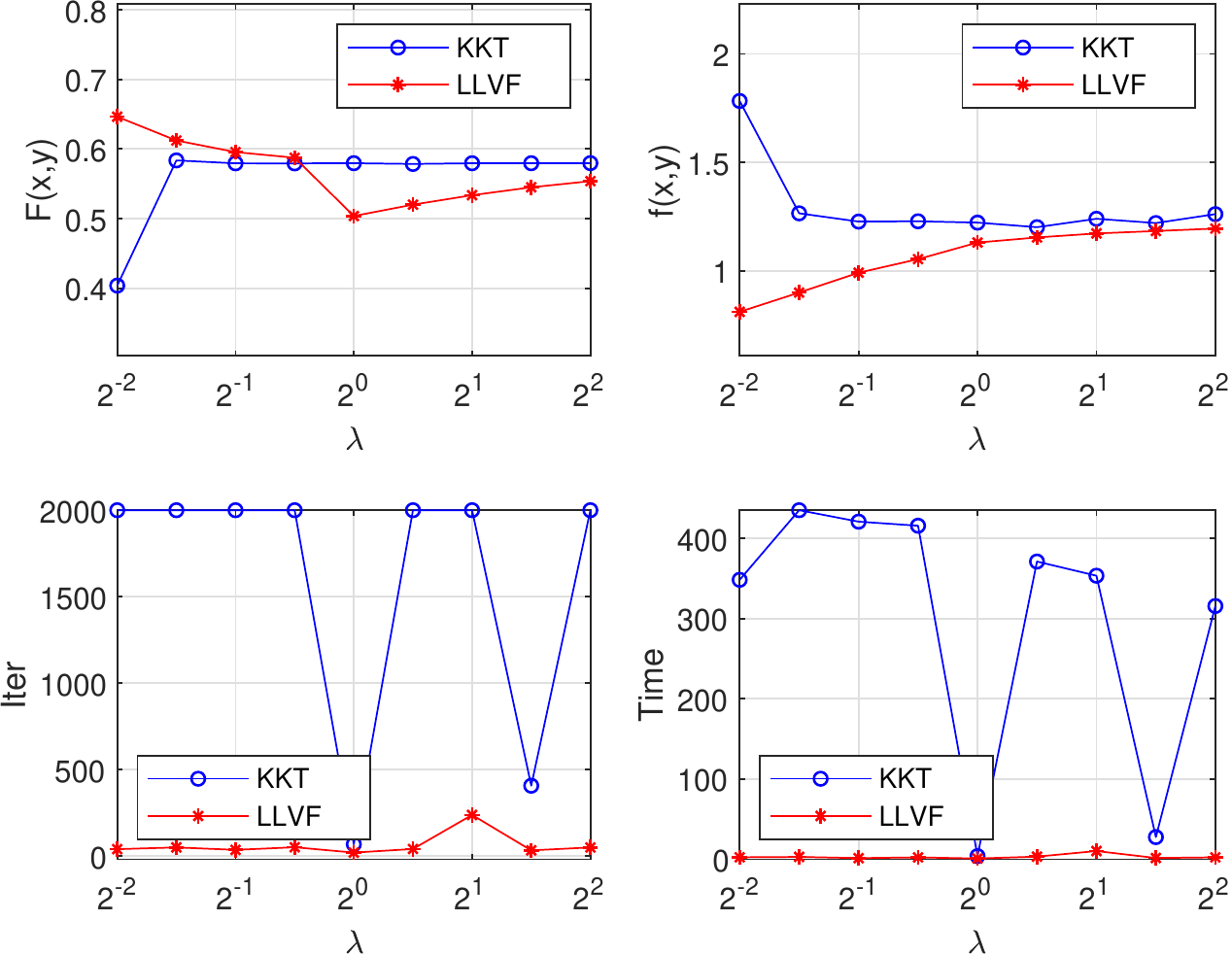}
  \caption{Performance of  SNKKT and SNVF on solving the quadratic BOC program.}
  \label{ioc}
\end{figure}

Finally, in order to  estimate the local behaviour of SNKKT and SNVF  on  our test examples, we  report on the experimental order of convergence (EOC) defined by
$$
\mbox{EOC}:=\max\left\{\frac{\log\|\Phi^\lambda(\zeta^{K-1})\|}{\log\|\Phi^\lambda(\zeta^{K-2})\|},\, \frac{\log\|\Phi^\lambda(\zeta^{K})\|}{\log\|\Phi^\lambda(\zeta^{K-1})\|} \right\}.
$$
As shown in Figure \ref{eoc},  when $\lambda=2^{-1}$, for examples, SNKKT (resp. SNVF)  solves 27 (resp. 14) problems with $\mbox{EOC}<1.1$, 8 (resp. 9) problems with $1\leq\mbox{EOC}<1.5$ and 89 (resp. 101) examples with $\mbox{EOC}\geq1.5$. Generally speaking, SNVF outperforms SNKKT because it solves more examples with $\mbox{EOC}\geq1.5$ and fewer examples with $\mbox{EOC}<1.1$.

\subsection{Comparison of SNKKT  and SNVF on solving  a quadratic BOC program}
Note that the examples in the BOLIB library are of a small scale, with dimensions satisfying $\max\{n, m, p, q\} \leq20$. Therefore, to see the performance of SNKKT and SNVF on solving problems with lager scale, we take advantage of a discretized bilevel optimal control (BOC) program from \cite{MehlitzGerd16}. This is a quadratic program and its dimensions $\{n,m,p,q\}$ can to be altered. The model is described by
$$
\begin{array}{rll}
  F(x,y)& := &\frac{1}{2}[(
  							y^1;0)-c]^\top
  						  D[(
  							y^1;
  							  0)-c] -d^\top x,\\
  G(x,y)& := &\left( - x_1+x_2-1;-x \right),\\	
  f(x,y)& := &\frac{1}{2}(Cy^1-Px)^\top U(Cy^1-Px)+\frac{\sigma}{2}(y^2-Qx)^\top V(y^2-Qx),\\
  g(x,y)& := & \left(
                       y^2-u;
                      -y^2+l;
                     Ay;
                       - Ay
     \right),
\end{array}
$$
where $x\in\mathbb{R}^2$,  $y=( y^1;y^2)$ with $y^i\in\mathbb{R}^{m_i}$, $D\in\mathbb{R}^{m_1\times m_1}$, $d\in\mathbb{R}^{n}$, $c\in\mathbb{R}^{m}$, $C\in\mathbb{R}^{s\times m_1}$, $P\in\mathbb{R}^{s\times n}$, $U\in\mathbb{R}^{s\times s}$, $Q\in\mathbb{R}^{m_2\times n}$, $V\in\mathbb{R}^{m_2\times m_2}$, $u\in\mathbb{R}^{m_2}$, $l\in\mathbb{R}^{m_2}$, $A\in\mathbb{R}^{t\times m}$ are given data. Here, $(a;b)=(a^\top~b^\top)^\top.$
 For simplicity, we fix the dimensions as $n=2, m=274, p=3, m_1=m_2=s=t=137$ and thus $q=548$.

\begin{table}[ht]
{	\renewcommand{\arraystretch}{1.15}\addtolength{\tabcolsep}{-0.25pt}
\begin{tabular}{llrrrrrrrrr}\\\hline
&$\lambda$ &$2^{-2}$&$2^{-1.5}$&$2^{-1}$&$2^{-0.5}$&$2^{0}$&$2^{0.5}$&$2^{1}$&$2^{1.5}$&$2^{2}$\\\hline
\multirow{2}{*}{$F(x,y)$}	&	SNKKT	&	0.40	&	0.58	&	0.58	&	0.58	&	0.58	&	0.58	&	0.58	&	0.58	&	0.58	\\
	&	SNVF	&	0.65	&	0.61	&	0.60	&	0.59	&	0.50	&	0.52	&	0.53	&	0.54	&	0.55	\\\hline
\multirow{2}{*}{$f(x,y)$}	&	SNKKT	&	1.78	&	1.27	&	1.23	&	1.23	&	1.22	&	1.20	&	1.24	&	1.22	&	1.26	\\
	&	SNVF	&	0.81	&	0.90	&	0.99	&	1.05	&	1.13	&	1.15	&	1.17	&	1.18	&	1.20	\\\hline
\multirow{2}{*}{Iter}	&	SNKKT	&	2000	&	2000	&	2000	&	2000	&	67	&	2000	&	2000	&	406	&	2000	\\
	&	SNVF	&	39	&	49	&	35	&	50	&	19	&	40	&	237	&	32	&	49	\\\hline
\multirow{2}{*}{Time}	&	SNKKT	&	348.34	&	435.07	&	420.82	&	415.74	&	4.04	&	371.10	&	353.46	&	27.74	&	315.67	\\
	&	SNVF	&	2.56	&	3.03	&	1.54	&	2.45	&	1.06	&	3.40	&	10.34	&	1.69	&	2.48	\\\hline
\multirow{2}{*}{Time/Iter} 	&	SNKKT	&	0.17	&	0.22	&	0.21	&	0.21	&	0.06	&	0.19	&	0.18	&	0.07	&	0.16	\\
	&	SNVF	&	0.07	&	0.06	&	0.04	&	0.05	&	0.06	&	0.08	&	0.04	&	0.05	&	0.05	\\
\hline
${}$\\
\end{tabular}}
\caption{Performance of SNKKT and SNVF on solving  the quadratic BOC program.}\label{perf-tab:3}
\end{table}

 We now compare the performance of SNKKT and SNVF on solving a BOC program with larger size $n=2$, $m=274$, $p=3$, and $q=548$.  The problem is quadratic, and thus there is no need to calculate the third derivatives of functions in lower level problem.   We tested different values of $\lambda$ on solving this problem and observed that larger values $\lambda$ (e.g., $\lambda>4$) led to a bad performance for both methods. Therefore, we used $\lambda\in\{2^{-2}, \; 2^{-1.5}, \ldots, 2^{1.5}, 2^{2}\}$.  As depicted in Figure \ref{ioc}, it can be clearly seen that SNVF outperforms SNKKT in terms of the number of iterations and the computing time. Most importantly, from the subfigure \emph{Iter}, among 9 choices of $\lambda$, SNKKT used 2000 iterations for  7 choices of $\lambda$, which means it did not get desired solutions before it stopped.  Detailed results were listed in Table \ref{perf-tab:3}. It seems that the best optimal upper level and lower level objective function values are 0.5 and 1.13, respectively. The last two rows reported the time per each iteration, in which SNVF ran much faster than SNKKT for all $\lambda$.

\section{Conclusions and future work}
In this work, we have considered the most common single-level reformulations of the bilevel optimization problem; i.e., problems \eqref{eq:KKTR} and \eqref{eq:LLVFR}. After a detailed theoretical analysis and comparison of the two problems, in terms of (1) the requirements for problems to $\mathcal{C}^1$ or Lipschitz continuous and equivalent to the original problem \eqref{eq:P}, (2) the necessary optimality conditions and qualification conditions necessary to derive them, (3) frameworks for semismooth Newton-type method and convergence, and (4) numerical efficiency. It has resulted from the theoretical framework that non of the reformulations can be said to be superior to the other, although the KKT reformulation provides some higher level of flexibility and tractability that it borrows from the well-established field of MPCCs. However, from the numerical perspective, the LLVF reformulation appears to be superior to the KKT reformulation based the experiments conducted in this paper. Our assessment is that this may be largely due to the fact the size of the equation solved for the KKT reformulation is larger by $q\times q$ ($q$ number of lower-level constraints) and the fact that this reformulation is very sensitive to lower-level convexity and regularity condition as shown in \cite{DempeDuttaBlpMpec2010} (see Subsection \ref{Nature of reformulations and relationships to original problem}), while the LLVF reformulation is completely equivalent (i.e., globally and locally) to problem \eqref{eq:P}. Moreover, due to the presence of the gradient of lower-level Lagrangian function in the KKT reformulation, the evaluation of 3rd order derivatives of functions involved in the lower-level problem is required in the methods considered in this paper. The latter issue does not appear to have played a big influence in the test set used in this paper, but will be potentially be very detrimental to bilevel programs of lager sizes, especially when 3rd order derivatives of lower-level functions are nonzero. In order to check whether the observations made on the numerical performance of the methods studied in this paper for \eqref{eq:KKTR} and \eqref{eq:LLVFR} are valid in general, in a future work, we will be studying and comparing the theoretical performance bounds of the semismooth Newton-type method for both problems in a general setting.

\section*{Acknowledgements}
The authors would like to thank the two anonymous referees for their constructive remarks that
have helped us to improve the presentation of this paper.
%\nocite{*}

\bibliographystyle{amsplain}

\end{document}